\documentclass[12pt]{article}

\usepackage{latexsym}
\usepackage{mathrsfs}
\usepackage{amsfonts,amsmath,amssymb}
\usepackage{indentfirst}
\usepackage{subeqnarray}
\usepackage{verbatim}
\usepackage[pdftex]{graphicx}
\usepackage{epstopdf}
\usepackage{subfigure}

\newcommand{\bx}{\mbox{\boldmath{$x$}}}

\newcommand{\bz}{\mbox{\boldmath{$z$}}}

\newcommand{\btau}{\mbox{\boldmath{$\tau$}}}
\newcommand{\bzero}{\mbox{\boldmath{$0$}}}
\newcommand{\ba}{\mbox{\boldmath{$a$}}}

\newcommand{\fb}{\mbox{\boldmath{$f$}}}

\newcommand{\bu}{\mbox{\boldmath{$u$}}}
\newcommand{\bq}{\mbox{\boldmath{$q$}}}
\newcommand{\bU}{\mbox{\boldmath{$U$}}}
\newcommand{\bv}{\mbox{\boldmath{$v$}}}

\newcommand{\bV}{\mbox{\boldmath{$V$}}}
\newcommand{\bw}{\mbox{\boldmath{$w$}}}
\newcommand{\bW}{\mbox{\boldmath{$W$}}}

\newcommand{\real}{\mbox{$\mathbb{R}$}}

\newcommand{\bvarepsilon}{\mbox{\boldmath{$\varepsilon$}}}
\newcommand{\bsigma}{\mbox{\boldmath{$\sigma$}}}

\newcommand{\bnu}{\mbox{\boldmath{$\nu$}}}

\newcommand{\cL}{\mbox{{${\cal L}$}}}

\newcommand{\weak}{\rightharpoonup}

\newtheorem{theorem}{Theorem}[section]
\newtheorem{lemma}[theorem]{Lemma}
\newtheorem{assumption}[theorem]{Assumption}

\newtheorem{definition}[theorem]{Definition}
\newtheorem{example}[theorem]{Example}
\newtheorem{problem}[theorem]{Problem}

\newtheorem{proposition}[theorem]{Proposition}
\newtheorem{remark}[theorem]{Remark}
\numberwithin{equation}{section}
\newenvironment{proof}[1][Proof]{\textbf{#1.} }
{\ \rule{0.75em}{0.75em}\smallskip}

\usepackage[pdftex,unicode,colorlinks,linkcolor=blue]{hyperref}

\begin{document}

\begin{center}
\Large\bf Numerical analysis of variational-hemivariational inequalities with applications in contact mechanics
\end{center}

\medskip
\begin{center}
Weimin Han\footnote{Department of Mathematics, University of Iowa, Iowa City, IA 52242, USA; email: weimin-han@uiowa.edu.  The work of this author was partially supported by Simons Foundation Collaboration Grants (Grant No.\ 850737).}, \quad
Fang Feng\footnote{School of Mathematics and Statistics, Nanjing University of Science and Technology, Nanjing 210094, Jiangsu, China; email: ffeng@njust.edu.cn.   The work of this author was partially supported by the National Natural Science Foundation of China (Grant No.\ 12401528) and the Fundamental Research Funds for the Central Universities, No.\ 30924010837.}, \quad
Fei Wang\footnote{School of Mathematics and Statistics, Xi'an Jiaotong University,
Xi'an, Shaanxi 710049, China; email: feiwang.xjtu@xjtu.edu.cn.  The work of this author was partially
supported by the National Natural Science Foundation of China (Grant No.\ 12171383).},
\quad
Jianguo Huang\footnote{Corresponding author. School of Mathematical Sciences, and MOE-LSC, Shanghai Jiao Tong University, Shanghai 200240, China; email: jghuang@sjtu.edu.cn.  The work of this author was partially supported by the National Natural Science Foundation of China (Grant No.\ 12071289). }
\end{center}

\medskip
\begin{quote}
{\bf Abstract.}  Variational-hemivariational inequalities are an important mathematical framework for 
nonsmooth problems.  The framework can be used to study application problems from physical sciences and
engineering that involve non-smooth and even set-valued relations, monotone or non-monotone, among 
physical quantities. Since no analytic solution formulas are expected for variational-hemivariational inequalities
from applications, numerical methods are needed to solve the problems.  
This paper focuses on numerical analysis of variational-hemivariational inequalities, reporting new 
results as well as surveying some recent published results in the area.  A general convergence result 
is presented for Galerkin solutions of the inequalities under minimal solution regularity 
conditions available from the well-posedness theory, and C\'{e}a's inequalities are derived for error estimation of 
numerical solutions.  The finite element method and the virtual element method are taken as examples 
of numerical methods, optimal order error estimates for the linear element solutions are derived when the methods
are applied to solve three representative contact problems under certain solution regularity assumptions.
Numerical results are presented to show the performance of both the finite element method and the 
virtual element method, including numerical convergence orders of the numerical solutions that 
match the theoretical predictions.  
\end{quote}

\medskip
\section{Introduction}

Variational-hemivariational inequalities (VHIs) are an important mathematical framework for studying
nonsmooth problems in applications. This framework contains variational inequalities (VIs) and
hemivariational inequalities (HVIs) as special cases.  VIs are inequality
problems in which non-smooth terms have a convex property, whereas HVIs are those
in which non-smooth terms are allowed to be non-convex. A VHI has the features
of both a VI and a HVI, i.e., both convex and non-convex non-smooth terms
are present in the problem.  In the literature, the two terms ``hemivariational inequalities'' and
``variational-hemivariational inequalities'' are used interchangeably.  In this paper, we use the term
``variational-hemivariational inequalities'' to also mean hemivariational inequalities.

Rigorous mathematical analysis on VIs began in 1960s (\cite{Fi64}).
By early 1970s, foundations of basic mathematical theory of VIs were established
in a series of papers \cite{Br72, HS66, LS67, St64}.  In \cite{DL1976}, many complicated application
problems were modeled and studied as VIs. Since one does not expect to have
solution formulas for VIs arising in applications, numerical methods are needed
to solve VIs.  Early comprehensive references on numerical methods for solving
VIs are \cite{Gl1984, GLT1981}.  Modelling, analysis and numerical solution of
VIs in mechanics are treated in \cite{HR2013, HS2002, HHN96, HHNL1988, KO1988}
and in many other references. Even though the area of VIs is now pretty mature,
it is still an active research area due to emerging new applications and the need of developing better
numerical methods and algorithms for solving VIs (e.g., \cite{CHY23, CHR2023, JA2023, UL2011, Yo21}).

HVIs, and more generally, VHIs, find their applications in problems involving non-smooth, non-monotone
and set-valued relations among physical quantities.  Since the pioneering work of Panagiotopoulos in
early 1980s (\cite{Pa83}), there has been extensive research on modeling, analysis, numerical solution
and applications of VHIs. Recent years have witnessed an explosive growth in the literature on modeling,
analysis, numerical approximation and simulations, and applications of VHIs.  Early comprehensive references
on mathematical theory, numerical solution and applications of VHIs include
\cite{HMP1999, MP1999, NP1995, Pa1993}.  More recent monographs covering the mathematical theory
and applications of VHIs include \cite{CL2021, CLM2007, GM2003, GMDR2003, MOS2013, SM2018}.
In these references and in the vast majority of other publications on well-posedness of VHIs,
abstract surjectivity results on pseudomonotone operators and a Banach fixed-point argument are
applied to show the solution existence. An alternative and more accessible approach, without
the use of abstract theory of pseudomonotone operators, has been developed for the mathematical theory
of VHIs.  This new approach starts with minimization principles for a special family of VHIs,
first established in \cite{Han20}; the theory is then extended in \cite{Han21} to cover general VHIs
through fixed-point arguments.  The book \cite{Han2024} is devoted to the mathematical theory of
VHIs using the new approach.

VHIs are more complicated than VIs, and numerical methods are needed to solve them.
The finite element method and a variety of solution algorithms are discussed in \cite{HMP1999}
to solve HVIs. An optimal order error estimate is first presented in \cite{HMS14} for the linear
finite element solutions of a VHI.  This is followed by a series of papers on further analysis of the
finite element method to solve VHIs, e.g., \cite{Han18, HSB17, HSD18, HZ19}.  The reference \cite{HS19AN}
provides a recent survey of numerical analysis of VHIs, including some time-dependent problems.
Other numerical methods have been studied for solving VHIs, e.g.,
\cite{FHH19, FHH21a, FHH21b, FHH22, LWH20, WWH21, WWH22, XL23a, XL23b, XL23c}
on the use of virtual element methods, and \cite{WSW23} on the use of discontinuous Galerkin methods.
Machine learning techniques have been explored recently to solve the problems, cf.\ \cite{CSWL23, HWW22}.

The aim of this paper is to provide a summary account on the numerical solution of VHIs.
We will only consider stationary/time-independent problems.
In Section \ref{sec:pre}, we present preliminary materials needed later in the paper.  In particular, we
review the notions of generalized subdifferentials and generalized subgradients, and their properties.
In Section \ref{sec:sample}, we introduce three contact problems; their weak formulations are VHIs.
In Section \ref{sec:abs}, we provide an analysis of the Galerkin method for an abstract VHI,
which contains the three VHIs introduced in Section \ref{sec:sample} as special cases.
In Section \ref{sec:contact}, we apply results presented in Section \ref{sec:abs} to
study the three contact problems, including optimal order error estimates for their numerical
solutions using the linear finite element method  under certain solution regularity assumptions.  
In Section \ref{sec:VEM}, we analyze the virtual element method for solving an abstract VHI. 
In Section \ref{VEM:contact}, we apply the VEM to solve the three contact problems and derive 
optimal order error estimates under certain solution regularity assumptions.  
In Section \ref{sec:ex}, we present numerical examples for solving the contact problems, 
and report the numerical convergence orders of the FEM and VEM solutions.  The paper ends with 
some concluding remarks in Section \ref{sec:final}.

\section{Preliminaries}\label{sec:pre}

In the study of VHIs, we need the notions of the generalized directional derivative and generalized subdifferential
for locally Lipschitz continuous functions introduced by F. H. Clarke (\cite{Cl75, Cl1983}). 
In this section, we use the symbol $V$ for a Banach space, and $U$ for an open subset in $V$.

\begin{definition}\label{def:Clarke}
Assume $\Psi\colon U\to \real$ is locally Lipschitz continuous. Then the generalized {\rm (\emph{Clarke})} 
directional derivative of $\Psi$ at $u\in U$ in the direction $v \in V$ is defined by
\begin{equation}
\Psi^0(u; v) := \limsup_{w\to u,\,\lambda \downarrow 0}\frac{\Psi(w+\lambda v) -\Psi(w)}{\lambda}, 
\label{Clarke1}
\end{equation}
and the generalized {\rm (\emph{Clarke})} subdifferential of $\Psi$ at $u\in U$ is defined by
\begin{equation}
\partial\Psi(u) := \left\{u^*\in V^*\mid\Psi^{0}(u;v)\ge\langle u^*,v\rangle\ \forall\,v\in V\right\}.  
\label{Clarke2}
\end{equation}
\end{definition}

We note that the upper limit in \eqref{Clarke1} is well-defined for a locally Lipschitz continuous functional $\Psi$.  
Often, we will use Definition \ref{def:Clarke} for the particular case $U=V$.  

Basic properties of the generalized directional derivative and the generalized gradient are recorded in the 
next result (cf.\ \cite{Cl1983, CLSW1998, GP2005} or \cite[Section 3.2]{MOS2013}).

\begin{proposition}\label{subdiff}
Assume that $\Psi \colon U\to \real$ is a locally Lipschitz function. Then the following statements are valid.

\vspace{0.1mm}
{\rm (i)}  $\Psi^0(u;\lambda\,v)= \lambda\, \Psi^0(u; v)$ $\forall\,\lambda\ge 0$, $u\in U$, $v\in V$.\\
\phantom{a}\qquad\ $\Psi^0(u; -v) = (-\Psi)^0(u; v)$ $\forall\,u\in U$, $v\in V$.

\vspace{0.1mm}
{\rm (ii)} $\Psi^0 (u; v_1 + v_2) \le \Psi^0(u; v_1)+\Psi^0(u; v_2)$ $\forall\,u\in U,\,v_1,v_2\in V$.

\vspace{0.1mm}
{\rm (iii)} $\Psi^0(u;v)=\max\left\{\langle u^*,v\rangle\mid u^*\in\partial\Psi(u)\right\}$ $\forall\,u\in U,
\,v\in V$.

\vspace{0.1mm}
{\rm (iv)} If $u_n\to u$ in $V$, $u_n,u\in U$, and $v_n \to v$ in $V$, then 
\[ \limsup_{n\to\infty} \Psi^0(u_n;v_n) \le \Psi^0(u; v). \]

\vspace{0.1mm}
{\rm (v)} \label{4-P1,3}
For every $u\in U$, $\partial \Psi(u)$ is nonempty, convex, and weak\-ly$^{\,*}$ compact in $V^*$.

\vspace{0.1mm}
{\rm (vi)} If $u_n \to u$ in $V$, $u_n,u\in U$, $u^*_n \in \partial \Psi(u_n)$,
and $u^*_n \to u^*$ weakly$^{*}$ in $V^*$, then $u^*\in\partial\Psi(u)$.

\vspace{0.1mm}
{\rm (vii)} If $\Psi \colon U \to \real$ is convex, then the generalized subdifferential $\partial \Psi(u)$
at any $u \in U$ coincides with the convex subdifferential $\partial \Psi(u)$.
\end{proposition}

Because of Proposition \ref{subdiff}\,(vii), the symbol $\partial$ is used for both the generalized 
subdifferential of locally Lipschitz continuous functions and the convex subdifferential of convex functions.
Detailed discussion on convex subdifferentials can be found in many references on convex functions, e.g., \cite{ET1976}.

One simple consequence of Proposition \ref{subdiff}\,(ii) is
\begin{equation}
\Psi^0(u;-v)\ge -\Psi^0(u;v)\quad\forall\,u\in U,\ v\in V.
\label{2.8a}
\end{equation}
This property can also be proved directly from the definition of the generalized directional directive.

In the description of the next result, we need the concept of regularity of a locally Lipschitz 
continuous function.

\begin{definition}\label{def:2.33c}\index{function!regular}\index{regular function}
A function $\Psi\colon U\to\mathbb{R}$ is regular at $u\in U$ if $\Psi$ is Lipschitz continuous near $u$ 
and the directional derivative $\Psi^\prime(u;v)$ exists such that 
\[ \Psi^\prime(u;v) = \Psi^0(u;v)\quad\forall\,v\in V.\]
\end{definition}

It is known that a function is regular at any point where the function is continuously differentiable.
In addition, a l.s.c.\ function is regular at any point in the interior of its effective domain.

\begin{proposition}\label{B.Pr}
Let $\Psi,\Psi_1,\Psi_2 \colon U\to \real$ be locally Lipschitz functions. Then:
\begin{list}{}{
\setlength{\leftmargin}{1.0cm}
\setlength{\rightmargin}{0.0in}
\setlength{\labelwidth}{0.05cm}
}

\item[{\rm (i)}] {\rm (scalar multiples)} 
\begin{equation}\label{B.PART10}
\partial(\lambda\,\Psi)(u)=\lambda\,\partial \Psi(u)\quad \forall\,\lambda \in \real, \ u\in U.
\end{equation}

\item[{\rm (ii)}] {\rm (sum rules)} 
\begin{equation}\label{B.PART11}
\partial (\Psi_1 + \Psi_2) (u) \subset \partial \Psi_1 (u) + \partial \Psi_2 (u)\quad\forall\,u\in U,
\end{equation}
or equivalently,
\begin{equation}\label{B.PART12}
(\Psi_1 + \Psi_2)^0(u; v) \le \Psi_1^0(u; v) + \Psi_2^0(u; v)\quad\forall\,u\in U,\, v\in V.
\end{equation}
If $\Psi_1$ and $\Psi_2$ are regular at $u$, then \eqref{B.PART11} and \eqref{B.PART12} hold with equalities.
\end{list}
\end{proposition}

In the study of VHIs, we will assume a condition of the form
\begin{equation}
\Psi^0(v_1;v_2-v_1) + \Psi^0(v_2;v_1-v_2) \le \alpha_\Psi \|v_1-v_2\|_V^2 \quad\forall\,v_1,v_2\in U
\label{2.12a}
\end{equation}
for a constant $\alpha_\Psi\ge 0$.  This condition characterizes the degree of non-convexity of 
the functional $\Psi$: the smaller the constant $\alpha_\Psi\ge 0$, the weaker the non-convexity of $\Psi$.
For a convex functional $\Psi$, \eqref{2.12a} holds with $\alpha_\Psi=0$.  The condition \eqref{2.12a}
is sometimes given as a condition on the generalized subdifferential (cf.\ \cite[p.\ 124]{SM2018}).

\begin{proposition}\label{prop:2.18a}
The condition \eqref{2.12a} is equivalent to 
\begin{equation}
\langle v^*_1-v^*_2,v_1-v_2\rangle \ge -\alpha_\Psi \|v_1-v_2\|_V^2 \quad\forall\,v_i\in U,
\,v^*_i\in \partial\Psi(v_i),\, i=1,2.
\label{2.12b}
\end{equation}
\end{proposition}

The condition \eqref{2.12b} is known as a relaxed monotonicity condition in the literature. 
\index{relaxed monotonicity condition} The inequality \eqref{2.12b} with $\alpha_\Psi=0$ is the monotonicity of 
$\partial\Psi$ for a convex functional $\Psi$.

For convenience, we will write \eqref{2.12b} as 
\begin{equation}
\langle \partial\Psi(v_1)-\partial\Psi(v_2),v_1-v_2\rangle\ge -\alpha_\Psi\|v_1-v_2\|_V^2 \quad\forall\,v_1,v_2\in U.
\label{2.12c}
\end{equation}

The following result is useful for verification of the condition \eqref{2.12a}; 
it is proved, e.g., in \cite[p.\ 26]{Han2024}.

\begin{theorem}\label{non_convex}
Assume $\Psi\colon U\to\mathbb{R}$ is locally Lipschitz continuous and $\alpha_\Psi\in\mathbb{R}$.  
Then \eqref{2.12a} holds if and only if the functional $v\mapsto \Psi(v)+(\alpha_\Psi/2)\,\|v\|_V^2$ is convex on $U$.
\end{theorem}

The following chain rule is proved in \cite[Lemma 4.2]{MOS10}.   More general chain rules for the generalized 
directional derivative and generalized subdifferential can be found in \cite[Chapter 10]{Cl2013}.

\begin{proposition}\label{B.Af}
Let $V$ and $W$ be Banach spaces, let $\Psi_0 \colon W \to \real$ be locally Lipschitz and let
$T\colon V\to W$ be given by $Tv=Av+w$ for $v\in V$, where $A \in {\mathcal L}(V, W)$ and $w\in W$ is fixed.
Then the function $\Psi \colon V \to \real$ defined by $\Psi(v) = \Psi_0 (Tv)$ is locally Lipschitz and
\begin{align}
& \Psi^0 (u; v) \le \Psi_0^0 (Tu; Av)\quad\forall\,u, v \in V,\label{Clark1}\\
& \partial \Psi (u) \subseteq A^* \partial \Psi_0 (Tu)\quad\forall\,u\in V,\label{Clark2}
\end{align}
where $A^*\in {\mathcal L}(W^*,V^*)$ is the adjoint operator of $A$.  Moreover, the equalities 
in \eqref{Clark1} and \eqref{Clark2} hold true if $A$ is surjective.
\end{proposition}

For detailed discussion of the properties of the Clarke subdifferential,  we refer the reader
to \cite{Cl75, Cl1983, CLSW1998, Cl2013, DMP1, DMP2}.

\smallskip
In virtually all the applications in mechanics, the locally Lipschitz continuous functional $\Psi$ is expressed as an 
integral of a locally Lipschitz continuous function $\psi$ of a real variable or of several real variables. 
The following formula is useful to compute the Clarke subdifferential of a function defined over a 
finite dimensional set (cf.\ \cite[Theorem 10.7]{Cl2013}, \cite[Prop.\ 3.34]{MOS2013}).

\begin{proposition} \label{prop:finite-dim}
Assume $U\subset \mathbb{R}^d$ is open, $\psi:U\to\mathbb{R}$ is locally Lipschitz continuous near $\bx\in U$,
$N\subset\mathbb{R}^d$ with $|N|=0$, and $N_\psi\subset\mathbb{R}^d$ with
$|N_\psi|=0$ such that $\psi$ is Fr\'{e}chet differentiable on $U\backslash N_\psi$.  Then,
\[ \partial\psi(\bx)={\rm conv}\left\{\lim \psi^\prime(\bx_k)\mid \bx_k\to\bx,\
 \bx_k\not\in N\cup N_\psi\right\}. \]
\end{proposition}

Next, we show some examples on the generalized subdifferential for locally Lipschitz continuous 
functions by applying Proposition \ref{prop:finite-dim}. 

For the function $\psi_1(x)=-|x|$, its generalized subdifferential is
\[ \partial\psi_1(x)=\left\{\begin{array}{ll} 1 & {\rm if}\ x<0,\\ \lbrack -1,1 \rbrack & {\rm if}\ x=0,\\
-1 & {\rm if}\ x>0.\end{array} \right. \]

For 
\[ \psi_2(x)=\left\{\begin{array}{ll} 0,& x\le 0,\\ x,& x>0, \end{array}\right. \]
we have
\[ \partial\psi_2(x)=\left\{\begin{array}{ll} 0,& x<0,\\ \lbrack 0,1\rbrack, & x=0,\\
1,& x>0. \end{array}\right. \]
Note that $\psi_2$ is a convex function, and $\partial\psi_2$ is also the convex subdifferential of $\psi_2$.

Consider
\[ 
\psi_3(x)=\left\{\begin{array}{ll} 2\,x+3 &\ {\rm if}\ x< -1,\\[1mm]
|x| &\ {\rm if}\ |x|\le 1,\\[1mm]
2\,x^2 -1 &\ {\rm if}\ x >1.
\end{array}\right.
\]
For its generalized subdifferential, we have
\[
\partial \psi_3(x)=\left\{\begin{array}{ll} 2&\ {\rm if}\ x< -1,\\[0.1mm]
[-1,2]&\ {\rm if}\ x=-1,\\[0.1mm]
-1 &\ {\rm if}\ -1<x<0,\\[0.1mm]
[-1,1] &\ {\rm if}\ x=0,\\[0.1mm]
1 &\ {\rm if}\ 0<x<1,\\[0.1mm]
[1,4]&\ {\rm if}\ x=1,\\[0.1mm]
4\,x &\ {\rm if}\ x> 1.
\end{array}\right.
\]

On several occasions, we will use the modified Cauchy-Schwarz inequality
\begin{equation}
a\,b\le \epsilon \,a^2+c\,b^2\quad \forall\,a,b\in\mathbb{R},
\label{mCS}
\end{equation}
where $\epsilon>0$ is an arbitrary positive number and the constant $c>0$ depends on
$\epsilon$, indeed, we may simply take $c=1/(4\,\epsilon)$.

\section{Sample problems from contact mechanics}\label{sec:sample}

\subsection{Notation}

We first introduce the notation.  We are interested in mathematical models which describe the equilibrium of
the mechanical state of an elastic body subject to the action of external forces and constraints on the boundary.
We denote by $\Omega$ the reference configuration of the body and assume $\Omega$ is an open, bounded,
connected set in $\real^d$ with a Lipschitz boundary $\Gamma=\partial\Omega$. In applications,
the dimension $d=2$ or $3$. The Lipschitz regularity assumption on $\Omega$ allows us to use most
of the basic properties of Sobolev spaces, including integration by parts formulas.  The unit outward normal
vector on $\Gamma$ exists a.e.\ and we denote it by $\bnu$.  We use boldface letters for vectors and tensors.
A typical point in $\mathbb{R}^d$ is denoted by $\bx=(x_i)$. The range of indices $i$, $j$, $k$, $l$
is between $1$ and $d$.  We adopt the summation convention over a repeated index, e.g., $a_i b_i$ stands for 
the summation $a_1 b_1 + \cdots + a_d b_d$.  The index following a comma indicates a partial derivative 
with respect to the corresponding component of the spatial variable $\bx$.  For example, for a function 
$g(\bx)$, $g_{,j}(\bx)$ denoted the partial derivative $\partial g(\bx)/\partial x_j$.

We denote by $\mathbb{S}^d$ the space of second order symmetric tensors
on $\mathbb{R}^d$.  For our purpose, we can simply view $\mathbb{S}^{d}$ as the space of symmetric matrices
of order $d$.  Over $\mathbb{R}^d$ and $\mathbb{S}^{d}$, we use the canonical inner products and norms defined by
\begin{align}
& \bu\cdot\bv=  u_i v_i, \quad |\bv|=(\bv\cdot\bv)^{1/2}\quad\forall\,\bu=(u_i),\bv=(v_i)\in\mathbb{R}^d,
\label{A1w}\\[0.5mm]
&\bsigma:\btau =\sigma_{ij}\tau_{ij},\quad |\btau|=(\btau:\btau)^{1/2}\quad\forall\,
\bsigma=(\sigma_{ij}),\btau=(\tau_{ij}) \in\mathbb{S}^{d}.  \label{A2w}
\end{align}

The primary unknown of the contact problem is the displacement of the elastic body,
$\bu\colon\overline{\Omega}\to \mathbb{R}^d$.  We consider the contact problems within the framework 
of the linearized strain theory.  The, for a displacement field $\bu$, we use the linearized strain tensor
\[ \bvarepsilon(\bu)=\frac 12\left(\nabla \bu+(\nabla \bu)^T\right). \]
In componentwise form,
\begin{equation*}
\varepsilon_{ij}(\bu)=(\bvarepsilon(\bu))_{ij} =\frac 12\, (u_{i,j} + u_{j,i}),\quad 1\le i,j\le d,
\end{equation*}
where $u_{i,j}=\partial u_i/\partial x_j$. In the description of the contact problems, another important 
mechanical quantity is the stress tensor $\bsigma\colon \Omega\to \mathbb{S}^d$.  Both $\bvarepsilon(\bu)$
and $\bsigma$ are symmetric matrix valued functions on $\Omega$.

We will use Sobolev and Lebesgue spaces on $\Omega$, $\Gamma$, or their subsets, such as $L^2(\Omega;\real^d)$,
$L^2(\Gamma_N;\real^d)$, $L^2(\Gamma_C;\real^d)$, $H^1(\Omega;\real^d)$, and $H^1(\Omega;\mathbb{S}^d)$,
endowed with their canonical inner products and associated norms.  For a function
$\bv\in H^1(\Omega;\real^d)$ we write $\bv$ for its trace $\gamma\bv\in L^2(\Gamma;\real^d)$ on $\Gamma$.
A standard reference on Sobolev spaces is \cite{AF2003}.  One may also consult \cite{Bre2011, Evans2010} and
many other books on Sobolev spaces.

To describe the contact problems, we split the boundary of $\Gamma$ into three non-overlapping measurable parts:
$\Gamma=\Gamma_D\cup\Gamma_N\cup\Gamma_C$.  We will specify a displacement boundary condition on $\Gamma_D$,
a traction boundary condition on $\Gamma_N$, and contact boundary conditions on $\Gamma_C$.
We assume $\Gamma_D$ and $\Gamma_C$ have positive measures, $|\Gamma_D|>0$, $|\Gamma_C|>0$.
The space for the unknown displacement field is
\begin{equation}
 \bV:=\left\{\bv\in H^1(\Omega;\mathbb{R}^d) \mid \bv=\bzero\ \mbox{on}\ \Gamma_D\right\}.
\label{SpV}
\end{equation}
For some contact problems, the displacement will be sought in a subspace or a subset of $\bV$.
The space for the stress field is
\begin{equation}
\mathbb{Q}:=L^2(\Omega;\mathbb{S}^d)=\left\{\bsigma=(\sigma_{ij})\mid\sigma_{ij}=\sigma_{ji}\in L^2(\Omega),
\ 1\le i,j\le d\right\}.  \label{SpQ}
\end{equation}
The space $\mathbb{Q}$ is a real Hilbert space endowed with the inner product
\[  (\bsigma,\btau )_\mathbb{Q}=\int_{\Omega}\bsigma\cdot\btau\,dx,\quad \bsigma,\btau\in \mathbb{Q}.\]
The corresponding norm is denoted by $\|\cdot\|_\mathbb{Q}$. Due to the assumption $|\Gamma_D|>0$, there is
a constant $c>0$, depending on $\Omega$ and $\Gamma_D$, such that
\begin{equation}
\|\bv\|_{H^1(\Omega;\mathbb{R}^d)}\le c\,\|\bvarepsilon(\bv)\|_\mathbb{Q}\quad\forall\,\bv\in \bV
\label{Korn}
\end{equation}
This is known as a Korn's inequality, and its proof can be found in numerous publications, 
e.g.\ \cite[p.\ 79]{NH1981}.  Consequently, $\bV$ is a real Hilbert space under the inner product
\begin{equation}
(\bu,\bv)_{\boldsymbol V}=(\bvarepsilon(\bu),\bvarepsilon(\bv))_\mathbb{Q}. \label{3.3}
\end{equation}
The induced norm is
\begin{equation}\label{3.3n}
\|\bv\|_{\boldsymbol V}=\|\bvarepsilon(\bv)\|_\mathbb{Q}.
\end{equation}
It follows from Korn's inequality \eqref{Korn} that $\|\cdot\|_{H^1(\Omega;\mathbb{R}^d)}$ and 
$\|\cdot\|_{\boldsymbol V}$ are equivalent norms on $\bV$. We will use $\|\cdot\|_{\boldsymbol V}$ as the norm on $\bV$.

Denote by $\bV^*$ the dual of the space $\bV$ and by $\langle\cdot,\cdot\rangle$ the corresponding
duality pairing. For any element $\bv\in \bV$, denote by $v_\nu$ and $\bv_\tau$ its normal and
tangential components on $\Gamma$ given by $v_\nu=\bv\cdot\bnu$ and $\bv_\tau=\bv-v_\nu\bnu$,
respectively. For a function $\bsigma:\overline\Omega\to\mathbb{S}^d$, we denote by
$\sigma_\nu$ and $\bsigma_\tau$ its normal and tangential components on $\Gamma$, defined by the relations
\[ \sigma_{\nu}=(\bsigma\bnu)\cdot\bnu, \quad \bsigma_{\tau} =\bsigma\bnu - \sigma_{\nu}\bnu.  \]
It is straightforward to show that
\begin{align}
\bu\cdot\bv & =u_\nu v_\nu+\bu_\tau\cdot\bv_\tau,
\label{1.42a}\\
\bsigma\bnu\cdot\bv & =\sigma_\nu v_\nu+\bsigma_\tau\cdot\bv_\tau.
\label{1.42b}
\end{align}

For a differentiable field $\bsigma\colon \Omega\to\mathbb{S}^d$, its divergence is a vector-valued function
${\rm div}\,{\bsigma}\colon \overline{\Omega}\to\real$ with components
\[ ({\rm div}\,{\bsigma})_i=\sigma_{ij,j},\quad 1\le i\le d.\]
For $\bsigma\in H^1(\overline{\Omega};\mathbb{S}^d)$ and $\bv\in H^1(\overline{\Omega};\real^d)$, we have Green's formula
\begin{equation}
\int_\Omega\,\bsigma:\bvarepsilon(\bv)\,dx+\int_\Omega\,{\rm div}\,\bsigma\cdot\bv\,dx
= \int_\Gamma\bsigma\bnu \cdot\bv\,ds. \label{Green}
\end{equation}
From the trace inequality
\begin{equation}\label{trace}
\|\bv\|_{L^2(\Gamma;\mathbb{R}^d)}\leq c_0 \|\bv\|_{\boldsymbol V}\quad \forall\, \bv\in \bV,
\end{equation}
we can derive similar trace inequalities for the normal component and tangential component:
\begin{align}
& \|v_\nu\|_{L^2(\Gamma_C)} \le \lambda_\nu^{-1/2} \|\bv\|_{\boldsymbol V}\quad\forall\,\bv\in \bV,
\label{tr_nu}\\
& \|\bv_\tau\|_{L^2(\Gamma_C;\mathbb{R}^d)}\le\lambda_\tau^{-1/2} \|\bv\|_{\boldsymbol V}\quad\forall\,\bv\in \bV,
\label{tr_tau}
\end{align}
where $\lambda_\nu>0$ and $\lambda_\tau>0$ are the smallest eigenvalues of the eigenvalue problems
\[ \bu\in \bV,\quad \int_\Omega \bvarepsilon({\bu}):\bvarepsilon(\bv)\,dx
=\lambda \int_{\Gamma_C} u_\nu v_\nu \, ds\quad\forall\,\bv\in \bV,\]
and
\[ \bu\in \bV,\quad \int_\Omega \bvarepsilon({\bu}):\bvarepsilon(\bv)\,dx
=\lambda \int_{\Gamma_C} \bu_\tau \cdot \bv_\tau \, ds\quad\forall\,\bv\in \bV,\]
respectively.

In the study of Problem \ref{prob:cont2} below, we need a subspace of the space $\bV$:
\begin{equation}
\bV_1=\left\{\bv\in\bV\mid v_\nu=0\ {\rm on}\ \Gamma_C\right\}.
\label{SpV1}
\end{equation}
We use the norm $\|\cdot\|_{\boldsymbol V}$ over the subspace $\bV_1$.  Similar to 
\eqref{tr_nu} and \eqref{tr_tau}, we have the trace inequalities
\begin{align}
& \|v_\nu\|_{L^2(\Gamma_C)} \le \lambda_{\nu,1}^{-1/2} \|\bv\|_{\boldsymbol V}\quad\forall\,\bv\in \bV_1,
\label{tr_nu1}\\
& \|\bv_\tau\|_{L^2(\Gamma_C;\mathbb{R}^d)}\le\lambda_{\tau,1}^{-1/2} \|\bv\|_{\boldsymbol V}\quad\forall\,\bv\in \bV_1,
\label{tr_tau1}
\end{align}
where $\lambda_{\nu,1}>0$ and $\lambda_{\tau,1}>0$ are the smallest eigenvalues of the eigenvalue problems
\[ \bu\in \bV_1,\quad \int_\Omega \bvarepsilon({\bu}):\bvarepsilon(\bv)\,dx
=\lambda \int_{\Gamma_C} u_\nu v_\nu \, ds\quad\forall\,\bv\in \bV_1,\]
and
\[ \bu\in \bV_1,\quad \int_\Omega \bvarepsilon({\bu}):\bvarepsilon(\bv)\,dx
=\lambda \int_{\Gamma_C} \bu_\tau \cdot \bv_\tau \, ds\quad\forall\,\bv\in \bV_1,\]
respectively.  We have $\lambda_{\nu,1}\ge \lambda_\nu$ and $\lambda_{\tau,1}\ge \lambda_\tau$.

\subsection{Three contact problems}

In this subsection, we present mathematical models of three representative contact problems between an elastic body
and a rigid foundation. A variety of mathematical models of contact problems can be found in many
publications, cf.\ e.g., the comprehensive references \cite{HS2002, KO1988, MOS2013}.
In all contact problems, we have the following pointwise relations:
\begin{align}
-\operatorname*{div}{\bsigma} & =\fb_0\quad\mathrm{in\ }\Omega, \label{cont1}\\
\bsigma & ={\cal E}\bvarepsilon(\bu)\quad\mathrm{in\ }\Omega, \label{cont2}\\
\bvarepsilon(\bu) & =\frac{1}{2}\left[\nabla\bu+(\nabla\bu)^T\right]\quad\mathrm{in\ }\Omega, \label{cont3}\\
\bu & =\mathbf{0}\quad\text{on}\ \Gamma_D, \label{cont4}\\
\bsigma{\bnu} & =\fb_2\quad\text{on}\ \Gamma_N. \label{cont5}
\end{align}
We comment that \eqref{cont1} is the equilibrium equation, \eqref{cont2} is the elastic constitutive law,
\eqref{cont3} defines the linearized strain tensor, \eqref{cont4} represents the homogeneous boundary condition
on $\Gamma_D$ whereas \eqref{cont5} describes the traction boundary conditions.

On the elasticity operator ${\cal E}\colon \Omega\times \mathbb{S}^d \to \mathbb{S}^d$ in the constitutive law \eqref{cont2},
we assume the following properties:
\begin{equation}
\left\{\begin{array}{ll}
{\rm (a)\  There\ exists\ a\ constant}\ L_{\cal E}>0\ {\rm such\ that\ a.e.\ in}\ \Omega,\\
{}\qquad |{\cal E}(\cdot,\bvarepsilon_1)-{\cal E}(\cdot,\bvarepsilon_2)|
\le L_{\cal E} |\bvarepsilon_1-\bvarepsilon_2| \quad\forall\, \bvarepsilon_1,\bvarepsilon_2\in \mathbb{S}^d; \\ [1mm]
{\rm (b)\  there\ exists\ a\ constant}\ m_{\cal E}>0\ {\rm such\ that\ a.e.\ in}\ \Omega,\\
{}\qquad ({\cal E}(\cdot,\bvarepsilon_1)-{\cal E}(\cdot,\bvarepsilon_2)):
(\bvarepsilon_1-\bvarepsilon_2)\ge m_{\cal E} |\bvarepsilon_1-\bvarepsilon_2|^2\\
{}\qquad\qquad \forall\, \bvarepsilon_1,\bvarepsilon_2 \in \mathbb{S}^d; \\ [1mm]
{\rm (c) } \ {\cal E}(\cdot,\bvarepsilon)\ {\rm is\ measurable\ on\ }\Omega
\ {\rm for\ all \ }\bvarepsilon\in \mathbb{S}^d;  \\ [1mm]
{\rm (d)}\ {\cal E}(\cdot,\bzero)=\bzero\  {\rm a.e.\ in\ } \Omega.
\end{array}\right.
\label{Ass:E}
\end{equation}
For the force densities, we assume
\begin{equation}
\fb_0\in L^2(\Omega;\mathbb{R}^d), \quad \fb_2\in L^2(\Gamma_N;\mathbb{R}^d)
\label{Ass:f}
\end{equation}

To complete the description of the contact problems, we need to specify contact conditions on $\Gamma_C$.
In the first contact problem, we use the normal compliance contact condition with Tresca's friction law
\begin{equation}
 -\sigma_\nu\in\partial \psi_\nu(u_\nu),\quad |\bsigma_\tau|\le f_b,
 \quad -\bsigma_\tau=f_b\,\frac{\bu_\tau}{|\bu_\tau|}\ {\rm if}\ \bu_\tau\not=\bzero\quad{\rm on}\ \Gamma_C. \label{cont6}
\end{equation}
Here, the function $\psi_\nu\colon \mathbb{R}\to\mathbb{R}$ is locally Lipschitz continuous and is not necessarily convex,
$f_b\ge 0$ is a constant upper bound of the friction force.  In particular, when $f_b=0$, the last two relations
in \eqref{cont6} degenerate to the frictionless condition
\[ -\bsigma_\tau=\bzero \quad{\rm on}\ \Gamma_C. \]
We assume the following properties on the function $\psi_\nu \colon \real \to\real$:
\begin{equation}
\left\{\begin{array}{ll}
{\rm (a)}  \ \psi_\nu(\cdot)\ \mbox{is locally Lipschitz on}\ \real;\\ [0.2mm]
{\rm (b)\  there\ exist\ constants}\ {\bar{c}}_0, {\bar{c}}_1 \ge 0\ {\rm such\ that}\\
{}\qquad  |\partial\psi_\nu(z)| \le {\bar{c}}_0+ {\bar{c}}_1\,|z|\quad \forall\,z\in\real;\\[0.2mm]
{\rm (c)\  there\ exists\ a\ constant}\ \alpha_{\psi_\nu} \ge 0\ {\rm such\ that}\\
{}\qquad \psi_\nu^0(z_1;z_2-z_1)+\psi_\nu^0(z_2; z_1-z_2)\le\alpha_{\psi_\nu}|z_1-z_2|^2\quad \forall\,z_1,z_2\in\real.
\end{array}
\right. \label{psi_nu}
\end{equation}

One can find derivations of weak formulations of contact problems in many references, e.g., 
\cite{MOS2013}, \cite[Chapter 4]{Han2024}.  We skip the derivations of weak formulations in this paper.
The weak formulation of the contact problem of \eqref{cont1}--\eqref{cont5} and \eqref{cont6} is the following.  
For convenience, we use $I_{\Gamma_C}(f)$ to denote the integral of a function $f$ over $\Gamma_C$.

\begin{problem}\label{prob:cont1}
{\it Find a displacement field $\bu\in \bV$ such that}
\begin{align}
& ({\cal E}(\bvarepsilon({\bu})),\bvarepsilon(\bv) -\bvarepsilon(\bu))_\mathbb{Q}
+I_{\Gamma_C}(f_b|\bv_\tau|)-I_{\Gamma_C}(f_b|\bu_\tau|)+I_{\Gamma_C} (\psi_\nu^0 (u_\nu; v_\nu-u_\nu))\nonumber\\
&{}\qquad \ge\langle \fb,\bv-\bu\rangle\quad\forall\,\bv\in \bV.
\label{cont7}
\end{align}
\end{problem}

In the second problem, we use the bilateral contact condition with a general friction law:
\begin{equation}
 u_\nu=0,\quad -\bsigma_\tau\in \partial\psi_\tau(\bu_\tau)\quad{\rm on}\ \Gamma_C. \label{cont8}
\end{equation}
Here, the function $\psi_\tau\colon \mathbb{R}^d\to\mathbb{R}$ is locally Lipschitz continuous and is not necessarily convex.  We assume the following properties on the function $\psi_\tau\colon\real^d \to\real$:
\begin{equation}
\left\{\begin{array}{ll}
{\rm (a)} \ \psi_\tau(\cdot)\ \mbox{is locally Lipschitz on} \ \real^d; \\
{\rm (b)} \ | \partial \psi_\tau(\bz) | \le {\bar{c}}_0 + {\bar{c}}_1|\bz|\ \forall\,\bz\in \real^d, \,
\mbox{with\ constants}\ {\bar{c}}_0, \, {\bar{c}}_1 \ge 0; \\
{\rm (d)} \ \psi_\tau^0(\bz_1;\bz_2 -\bz_1) + \psi_\tau^0(\bz_2;\bz_1 -\bz_2) \le\alpha_{\psi_\tau}|\bz_1-\bz_2|^2 \\
\qquad \mbox{for all}\ \bz_1,\bz_2 \in \real^d,\ \mbox{with\ a\ constant}\ \alpha_{\psi_\tau}\ge 0.
\end{array}
\right.
\label{psi_tau}
\end{equation}

Recall the space $\bV_1$ defined in \eqref{SpV1}.  The weak formulation of the contact problem of 
\eqref{cont1}--\eqref{cont5} and \eqref{cont8} is the following.

\begin{problem}\label{prob:cont2}
{\it Find a displacement field $\bu\in \bV_1$ such that}
\begin{equation}
({\cal E}(\bvarepsilon({\bu})),\bvarepsilon(\bv))_\mathbb{Q}
+I_{\Gamma_C} (\psi_\tau^0 (\bu_\tau; \bv_\tau))\ge\langle \fb,\bv\rangle\quad\forall\,\bv\in \bV_1.
\label{cont9}
\end{equation}
\end{problem}

In the third problem, we consider a frictional unilateral contact problem characterized by the following boundary conditions:
\begin{align} \label{cont30}
& u_\nu\le g,\ \sigma_\nu+\xi_\nu\le 0,\ (u_\nu-g)(\sigma_\nu+\xi_\nu)=0,
\ \xi_\nu\in\partial \psi_\nu(u_\nu)\quad{\rm on}\ \Gamma_C,\\
\label{cont31} & |\bsigma_\tau|\le f_b,
 \quad -\bsigma_\tau=f_b\,\frac{\bu_\tau}{|\bu_\tau|}\ {\rm if}\ \bu_\tau\not=\bzero\quad{\rm on}\ \Gamma_C.
\end{align}
These conditions model the frictional contact between an elastic body and a rigid foundation covered by 
a layer of elastic material. The constraint $u_\nu\le g$ limits the penetration of the body, where $g$ 
represents the thickness of the elastic layer. In cases where penetration occurs and the normal displacement 
does not reach the limit $ g $, the contact is governed by a multivalued normal compliance condition: 
$ -\sigma_\nu = \xi_\nu \in \partial \psi_\nu(u_\nu) $.  We assume \eqref{psi_nu} on the function $\psi_\nu$.

To treat the constraint $u_\nu\le g$ on $\Gamma_C$, we define a subset of the space $\bV$:
\begin{equation}
\bU:=\left\{\bv\in \bV\mid v_\nu\le g\ {\rm on\ }\Gamma_C\right\}.
\label{def:U}
\end{equation}
The weak formulation of the contact problem for \eqref{cont1}–-\eqref{cont5} and \eqref{cont30}-–\eqref{cont31} is as follows:

\begin{problem}\label{prob:cont3}
{\it Find a displacement field $\bu\in \bU$ such that}
\begin{align}
& ({\cal E}(\bvarepsilon({\bu})),\bvarepsilon(\bv) -\bvarepsilon(\bu))_\mathbb{Q}
+I_{\Gamma_C}(f_b|\bv_\tau|)-I_{\Gamma_C}(f_b|\bu_\tau|)+I_{\Gamma_C} (\psi_\nu^0 (u_\nu; v_\nu-u_\nu))\nonumber\\
&\qquad \ge\langle \fb,\bv-\bu\rangle\quad\forall\,\bv\in \bU.
\label{cont9_3}
\end{align}
\end{problem}

\section{Numerical analysis of an abstract variational-hemivariational inequality}\label{sec:abs}

In this section, we study the Galerkin method for an abstract variational-hemivariational inequality (VHI).
Any result on the abstract VHI applies to Problem \ref{prob:cont1} and Problem \ref{prob:cont2}.  In the abstract
VHI, we denote by $\Delta$ the physical domain or its sub-domain, or its boundary or part of the boundary, and
denote by $I_\Delta$ the integration over $\Delta$,
\[ I_\Delta(v)=\int_\Delta v\,dx\ {\rm if}\ \Delta\subset\Omega,\quad
I_\Delta(v)=\int_\Delta v\,ds\ {\rm if}\ \Delta\subset\Gamma.\]
We consider a function $\psi$ which generally depends on the spatial variable $\bx\in \Delta$.
Usually, we simply use the notation $\psi(\cdot)$ to stand for $\psi(\bx,\cdot)$.
For a positive integer $m$, we let
\begin{equation}
V_\psi=L^2(\Delta;\mathbb{R}^m).
\label{SpVpsi}
\end{equation}
For application in the study of Problems \ref{prob:cont1} and \ref{prob:cont3}, we take $m=1$, whereas for Problem \ref{prob:cont2}, $m=d$.

\subsection{The abstract variational-hemivariational inequality}\label{subsec:abs}

The abstract VHI assumes the following form.

\begin{problem}\label{prob:VHI}
Find $u\in K$ such that
\begin{equation}
\langle Au,v-u\rangle +\Phi(v)-\Phi(u)+I_\Delta(\psi^0(\gamma_\psi u;\gamma_\psi v-\gamma_\psi u))
\ge \langle f,v-u\rangle \quad\forall\,v\in K.
\label{eq1}
\end{equation}
\end{problem}

In \cite{Han2024}, this problem is called a VHI of rank $(1,1)$ to reflect the fact that in the VHI \eqref{eq1},
the convex function $\Phi$ depends on one argument and the locally Lipschitz continuous function $\psi$ depends on
one argument.  In the general case $K\not=V$, Problem \ref{prob:VHI} can be viewed as a constrained VHI of rank $(1,1)$.

When $K=V$ is the entire space, Problem \ref{prob:VHI} becomes an unconstrained VHI of rank $(1,1)$:
Find $u\in V$ such that
\begin{equation}
\langle Au,v-u\rangle +\Phi(v)-\Phi(u)+I_\Delta(\psi^0(\gamma_\psi u;\gamma_\psi v-\gamma_\psi u))
\ge \langle f,v-u\rangle \quad\forall\,v\in V.
\label{eq1aa}
\end{equation}

In the study of Problem \ref{prob:VHI} and its numerical approximation, we will make some assumptions on the data.
\smallskip

\noindent \underline{$H(V)$} $V$ is a real Hilbert space.

\smallskip
\noindent \underline{$H(K)$} $K$ is a non-empty, closed and convex set in $V$.

\smallskip
\noindent \underline{$H(A)$} $A\colon V\to V^*$ is $L_A$-Lipschitz continuous and $m_A$-strongly monotone.

\smallskip
\noindent \underline{$H(\Phi)$} $\Phi\colon V\to\mathbb{R}$ is convex and continuous on $V$.

Note that an operator $A\colon V\to V^*$ is said to be $L_A$-Lipschitz continuous if
\begin{equation}
\|Av_1-Av_2\|_{V^*} \le L_A \|v_1-v_2\|_V\quad\forall\,v_1,v_2\in V,
\label{eq:5.3b}
\end{equation}
and it is said to be $m_A$-strongly monotone if
\begin{equation}
\langle Av_1-Av_2,v_1-v_2\rangle \ge m_A\|v_1-v_2\|_V^2 \quad\forall\,v_1,v_2\in V.
\label{eq2z}
\end{equation}
A consequence of the assumption $H(\Phi)$ is that for some constants $c_3$ and $c_4$, not necessarily positive,
\begin{equation}
\Phi(v)\ge c_3+c_4 \|v\|_V\quad\forall\,v\in V,
\label{Phi:lb}
\end{equation}
cf.\ e.g., \cite[Lemma 11.3.5]{AH2009}.

Generally, we can consider the situation where $\psi=\psi(\bx,z)$ is a function defined for
$\bx\in\Delta$ and $z\in\mathbb{R}^m$.  To simplify the exposition, we will only consider the
case where $\psi=\psi(z)$ does not depend on $\bx\in\Delta$.  We introduce the following assumption.

\smallskip
\noindent \underline{$H(\psi)$}  $\gamma_\psi\in {\cal L}(V;V_\psi)$;  $\psi\colon \mathbb{R}^m\to\mathbb{R}$
is locally Lipschitz continuous and for some non-negative constants $c_\psi$ and $\alpha_\psi$,
\begin{align}
|\partial\psi(z)|_{\mathbb{R}^m} &\le c_\psi\left(1+|z|_{\mathbb{R}^m}\right)
\quad\forall\,z\in\mathbb{R}^m, \label{eq10w} \\
\psi^0(z_1;z_2-z_1)+\psi^0(z_2;z_1-z_2) & \le\alpha_\psi|z_1-z_2|_{\mathbb{R}^m}^2
\quad\forall\,z_1,z_2\in {\mathbb{R}^m}.   \label{eq4}
\end{align}

\noindent \underline{$H(f)$} $f\in V^*$.

\smallskip

Note that \eqref{eq4} is equivalent to the following inequality:
\begin{equation}
\langle v^*_1-v^*_2,v_1-v_2\rangle \ge -\alpha_\Psi |v_1-v_2|_{\mathbb{R}^m}^2\quad
\forall\,v_i\in \mathbb{R}^m,\, v^*_i\in\partial \psi(v_i),\ i=1,2.
\label{eq4'}
\end{equation}

As consequences of the assumptions on $\psi$, we have the next result.

\begin{lemma}\label{lem:psi2}
Under the assumption $H(\psi)$,
\begin{equation}
\left|I_\Delta(\psi^0(\gamma_\psi u;\gamma_\psi v))\right|
\le c\left(1+\|u\|_V\right)\|\gamma_\psi v\|_{V_\psi}\quad\forall\,u,v\in V.
\label{3.32a}
\end{equation}
\end{lemma}
\begin{proof}
By the assumption \eqref{eq10w}, we have
\[ \left|\psi^0(\gamma_\psi u;\gamma_\psi v)\right|\le c\left(1+|\gamma_\psi u|_{\mathbb{R}^m}\right)
|\gamma_\psi v|_{\mathbb{R}^m}.\]
Then by an application of the Cauchy-Schwarz inequality,
\[ \left|I_\Delta(\psi^0(\gamma_\psi u;\gamma_\psi v))\right|
\le c\left(1+\|\gamma_\psi u\|_{V_\psi}\right)\|\gamma_\psi v\|_{V_\psi}. \]
Since $\gamma_\psi\in {\cal L}(V;V_\psi)$, we deduce \eqref{3.32a} from the above inequality.  \hfill
\end{proof}

Introduce an auxiliary functional
\begin{equation}
\Psi(v)=I_\Delta(\psi(\gamma_\psi v)),\quad v\in V.
\label{def:Psi}
\end{equation}
Denote by $c_\Delta>0$ the smallest constant in the inequality\index{$c_\Delta$}
\begin{equation}
I_\Delta(|\gamma_\psi v|_{\mathbb{R}^m}^2)\le c_\Delta^2 \|v\|_V^2\quad\forall\,v\in V.
\label{eq12w}
\end{equation}
We have the next result on properties of $\Psi$ (\cite[Section 3.3]{MOS2013}).

\begin{lemma}\label{lem:Ipsi}
Assume $H(\psi)$.  Then $\Psi\colon V \to \real$ is well defined by \eqref{def:Psi}, is locally
Lipschitz continuous on $V$, and
\begin{equation}
\Psi^0(u;v)\le I_\Delta(\psi^0(\gamma_\psi u;\gamma_\psi v)),\quad u,v\in V.
\label{J3}
\end{equation}
Moreover, there exists a constant $c\ge 0$ such that
\begin{equation}
\|\partial \Psi(v)\|_{V^*} \le c\left(1+ \|v\|_V\right) \quad \forall\,v\in V \label{J1}
\end{equation}
and
\begin{equation}
\Psi^0(v_1;v_2-v_1)+\Psi^0(v_2;v_1-v_2)\le \alpha_\psi \|\gamma_\psi(v_1-v_2)\|_{V_\psi}^2
\quad\forall\,v_1,v_2 \in V.  \label{J2}
\end{equation}
\end{lemma}

A well-posedness result on Problem \ref{prob:VHI} is stated next; its proof can be found in, 
e.g., \cite[Section 5.4]{Han2024}.

\begin{theorem}\label{thm1}
Assume $H(V)$, $H(K)$, $H(A)$, $H(\Phi)$, $H(\psi)$, $H(f)$, and $\alpha_\psi c_\Delta^2<m_A$.  
Then Problem \ref{prob:VHI} has a unique solution $u\in K$. Moreover, the solution $u\in K$ 
depends Lipschitz continuously on $f\in V^*$.
\end{theorem}

\subsection{Galerkin method for the abstract VHI}

Since there is no analytic solution formula for a variational-hemivariational inequality (VHI) arising in
applications, numerical methods are needed to solve the inequality problem.  In this section,
we provide a detailed discussion for the numerical solution of Problem \ref{prob:VHI}.
The numerical method is of Galerkin type.  We prove convergence of the numerical solutions in 
Subsection \ref{sec:con}, and derive a C\'{e}a-type inequality for error estimation of the 
numerical solutions in Subsection \ref{sec:est}.  In the rest of this section,
for Problem \ref{prob:VHI}, we assume $H(V)$, $H(K)$, $H(A)$, $H(\Phi)$, $H(\psi)$, $H(f)$ and
$\alpha_\psi c_\Delta^2<m_A$, so that by Theorem \ref{thm1}, the problem has a unique solution.
Let $V^h$ be a finite dimensional subspace of $V$, $h>0$ being a spatial discretization parameter.
Let $K^h$ be a non-empty, closed and convex subset of $V^h$.
Then, a Galerkin approximation of Problem~\ref{prob:VHI} is the following.

\smallskip
\begin{problem} \label{prob:VHIh}
{\it Find an element $u^h\in K^h$ such that}
\begin{align}
& \langle Au^h,v^h-u^h\rangle+\Phi(v^h)-\Phi(u^h)+I_\Delta(\psi^0(\gamma_\psi u^h;\gamma_\psi v^h-\gamma_\psi u^h))\nonumber \\
&\qquad{}\ge\langle f,v^h-u^h\rangle\quad\forall\,v^h\in K^h.
\label{eq1h}
\end{align}
\end{problem}

For the well-posedness of Problem \ref{prob:VHIh}, we can apply Theorem \ref{thm1} which is valid in the  
setting of finite-dimensional spaces as well.  For completeness, we state the result formally as a theorem.

\begin{theorem}\label{thm1h}
Assume $H(V)$, $H(K)$, $H(A)$, $H(\Phi)$, $H(\psi)$, $H(f)$, and $\alpha_\psi c_\Delta^2<m_A$.
Let $V^h$ be a finite-dimensional subspace of $V$ and let $K^h$ be a non-empty, closed and convex subset of $V^h$.
Then Problem~\ref{prob:VHIh} has a unique solution.
\end{theorem}

The approximation is called external if $K^h\not\subset K$, and is internal if $K^h\subset K$.
In \cite{HSD18}, the internal approximation with the choice $K^h=V^h\cap K$ is considered for Problem \ref{prob:VHI}.

\subsection{Convergence under basic solution regularity}\label{sec:con}

In this section, we provide a general discussion of convergence for the numerical solution
defined by Problem \ref{prob:VHIh}.  The key point is that the convergence is shown under the minimal
solution regularity $u\in K$ that is available from Theorem \ref{thm1}.  For convergence analysis,
we will need $\{K^h\}_h$ to approximate $K$ in the sense of Mosco (cf.\ \cite{Mo68, GLT1981}):
\begin{align}
&\label{3.3i} v^h\in K^h\ {\rm and}\ v^h\weak v\ {\rm in}\ V\ {\rm imply}\ v\in K;\\
&\label{3.3h} \forall\,v\in K,\ \exists\,v^h\in K^h\ {\rm such\ that\ }v^h\to v\ {\rm in\ }V\ {\rm as}\
h\to 0.
\end{align}

The following uniform boundedness property will be useful for convergence analysis of the numerical
solutions.

\smallskip
\begin{proposition}\label{prop:bd}
Keep the assumptions stated in Theorem \ref{thm1h}.  In addition, assume \eqref{3.3h}.
The discrete solution $u^h$ of Problem \ref{prob:VHIh} is uniformly bounded with respect to $h$, i.e.,
there exists a constant $M>0$ independent of $h$ such that $\|u^h\|_V\le M$.
\end{proposition}
\begin{proof}
Since $K$ is non-empty, there is an element $u_0\in K$.  We fix one such
element.  Then by \eqref{3.3h}, there exists $u^h_0\in K^h$ such that
 \[ u^h_0\to u_0\ {\rm in\ }V\ {\rm as\ }h\to 0.\]

By the strong monotonicity of $A$ from $H(A)$,
\begin{align}
 m_A \|u^h-u^h_0\|_V^2 & \le  \langle Au^h-Au^h_0,u^h-u^h_0\rangle \nonumber\\
& \le  \langle Au^h,u^h-u^h_0\rangle- \langle Au^h_0,u^h-u^h_0\rangle.  \label{5.3a}
\end{align}
Let $v^h=u^h_0$ in \eqref{eq1h} to get
\[ \langle Au^h,u^h_0-u^h \rangle+\Phi(u^h_0)- \Phi(u^h)
+I_\Delta(\psi^0(\gamma_\psi u^h;\gamma_\psi u^h_0-\gamma_\psi u^h))\ge \langle f,u^h_0-u^h \rangle, \]
which is rewritten as 
\[ \langle Au^h,u^h-u^h_0 \rangle \le \Phi(u^h_0)- \Phi(u^h)
+I_\Delta(\psi^0(\gamma_\psi u^h;\gamma_\psi u^h_0-\gamma_\psi u^h))+ \langle f,u^h-u^h_0 \rangle. \]
Then, we have from \eqref{5.3a} that
\begin{align}
m_A \|u^h-u^h_0\|_V^2 & \le\Phi(u^h_0)-\Phi(u^h)
+I_\Delta(\psi^0(\gamma_\psi u^h;\gamma_\psi u_0^h-\gamma_\psi u^h))\nonumber\\
&\quad{} + \langle f,u^h-u_0^h\rangle- \langle Au^h_0,u^h-u^h_0\rangle.
\label{3.3a}
\end{align}

From \eqref{Phi:lb},
\begin{equation}
-\Phi(u^h)\le -c_3-c_4\|u^h\|_{V}.
\label{4.4a}
\end{equation}
Take $z_1=\gamma_\psi u^h$ and $z_2=\gamma_\psi u^h_0$ in \eqref{eq4} to obtain
\[  \psi^0(\gamma_\psi u^h;\gamma_\psi u_0^h-\gamma_\psi u^h)
\le \alpha_\psi|\gamma_\psi (u_0^h-u^h)|_{\mathbb{R}^m}^2- \psi^0(\gamma_\psi u^h_0;\gamma_\psi u^h-\gamma_\psi u^h_0).\]
It follows from \eqref{eq10w} that
\[ - \psi^0(\gamma_\psi u^h_0;\gamma_\psi u^h-\gamma_\psi u^h_0)
\le \left(c_0+c_1|\gamma_\psi u^h_0|_{\mathbb{R}^m}\right)|\gamma_\psi (u^h-u^h_0)|_{\mathbb{R}^m}.\]
Then,
\[  \psi^0(\gamma_\psi u^h;\gamma_\psi u_0^h-\gamma_\psi u^h)
\le \left(c_0+c_1|\gamma_\psi u^h_0|_{\mathbb{R}^m}\right)|\gamma_\psi (u^h-u^h_0)|_{\mathbb{R}^m}
+ \alpha_\psi|\gamma_\psi (u_0^h-u^h)|_{\mathbb{R}^m}^2,\]
and 
\begin{align*}
I_\Delta(\psi^0(\gamma_\psi u^h;\gamma_\psi u_0^h-\gamma_\psi u^h))
&\le I_\Delta(\left(c_0+c_1|\gamma_\psi u^h_0|_{\mathbb{R}^m}\right)|\gamma_\psi (u^h-u^h_0)|_{\mathbb{R}^m})\\
&{}\quad +I_\Delta(\alpha_\psi|\gamma_\psi (u_0^h-u^h)|_{\mathbb{R}^m}^2).
\end{align*}
Consequently, we apply the Cauchy-Schwarz inequality and the assumption $\gamma_\psi\in\cL(V;V_\psi)$
in $H(\psi)$,
\begin{equation}
I_\Delta(\psi^0(\gamma_\psi u^h;\gamma_\psi u_0^h-\gamma_\psi u^h))
\le c\left(1+\|u^h_0\|_{V}\right)\|u_0^h-u^h\|_V +\alpha_\psi c_\Delta^2\|u_0^h-u^h\|_V^2.
\label{3.3cc}
\end{equation}

Use \eqref{4.4a} and \eqref{3.3cc} in \eqref{3.3a} to obtain
\begin{align*}
m_A \|u^h-u^h_0\|_V^2 & \le  \Phi(u^h_0)-c_3-c_4\|u^h\|_{V} +c\left(1+\|u^h_0\|_{V}\right)\|u^h-u^h_0\|_V \\
&\quad{}  +\alpha_\psi c_\Delta^2\|u^h-u^h_0\|_{V}^2 + \langle f- Au^h_0,u^h-u^h_0\rangle,
\end{align*}
which is rewritten as
\begin{align*}
\left(m_A -\alpha_\psi c_\Delta^2\right) \|u^h-u^h_0\|_V^2
& \le  c\left(1+\|u^h_0\|_V\right) \|u^h-u_0\|_V+\Phi(u^h_0)-c_3-c_4\|u^h\|_{V}\\
& \quad{}  + \langle f- Au^h_0,u^h-u^h_0\rangle. 
\end{align*}

The convergence of $\{u^h_0\}_h$ in $V$ implies that $\{\|u^h_0\|_V\}$ and $\{\|Au^h_0\|_{V^*}\}$ are uniformly bounded
with respect to $h$.  By the smallness condition, $m_A-\alpha_\psi c_\Delta^2>0$, we can conclude
from the above inequality that $\|u^h-u^h_0\|_V$ is uniformly bounded in $h$, so is $\|u^h\|_V$. \hfill
\end{proof}

By Aubin-Clarke's Theorem (cf.\ \cite[Theorem 2.61]{CL2021}, we have the next result:

\begin{lemma}\label{lem:int}
Assume $H(\psi)$.  Then, for any $z\in V_\psi$ and any $z^*\in \partial (I_\Delta(\psi(z)))$,
we have $\zeta_z\in V_\psi$ such that $\langle z^*,v\rangle = I_\Delta(\zeta_z v)$ for all $v\in V_\psi$, 
and $\zeta_z \in\partial\psi(z)$ a.e.\ on $\Delta$.
\end{lemma}

We now prove the convergence of the numerical solutions under the minimal solution regularity $u\in K$.
The next result and its proof follow \cite{HZ19}.

\smallskip
\begin{theorem}\label{thm:converge}
Keep the assumptions made in Theorem \ref{thm1}.  Moreover, assume $V^h$ is a finite-dimensional subspace of $V$,
$K^h$ is a non-empty, closed and convex subset of $V^h$, and \eqref{3.3i}--\eqref{3.3h} hold.
Let $u$ and $u^h$ be the solutions of Problem \ref{prob:VHI} and Problem \ref{prob:VHIh}, respectively.  Then,
\begin{equation}
 u^h\to u\quad{\rm in\ }V\ {\rm as\ }h\to 0. \label{4.9a}
\end{equation}
\end{theorem}
\begin{proof}
We split the proof into three main steps.  In the first step, we discuss the weak convergence
of the numerical solutions.  In the second step, we prove that the weak convergence
of the numerical solutions can be strengthened to strong convergence.  In the third step, 
we show that the limit of the numerical solutions is the solution $u$  of Problem \ref{prob:VHI}.

\smallskip
\noindent\emph{Step 1}. By Proposition \ref{prop:bd}, $\{u^h\}$ is bounded in $V$.  Since $V$ is reflexive and
$\gamma_\psi\in {\cal L}(V,V_\psi)$, there exist a subsequence $\{u^{h'}\}\subset \{u^h\}$ and an element $w\in V$ such that
\begin{equation}
u^{h'}\weak w\ {\rm in}\ V, \quad \gamma_\psi u^{h'}\weak \gamma_\psi w\ {\rm in}\ V_\psi.
\label{Pr1}
\end{equation}
As a consequence of the assumption \eqref{3.3i}, $w\in K$.

\smallskip
\noindent\emph{Step 2}.  Next, we show that the weak convergence \eqref{Pr1} can be strengthened to the strong convergence:
\begin{equation}
u^{h'}\to w\ {\rm in}\ V.
\label{Pr3}
\end{equation}
By the assumption \eqref{3.3h} and the continuity of the operator $\gamma_\psi$, we have a sequence 
$\{w^{h'}\}\subset V$ with the properties that $w^{h'}\in K^{h'}$ and
\begin{equation}
w^{h'}\to w\ {\rm in}\ V, \quad \gamma_\psi w^{h'}\to \gamma_\psi w\ {\rm in}\ V_\psi.
\label{Pr4}
\end{equation}
Since $A$ is $m_A$-strongly monotone,
\[ m_A\|w-u^{h'}\|_V^2\le \langle Aw-Au^{h'},w-u^{h'}\rangle, \]
or
\begin{equation}
m_A\|w-u^{h'}\|_V^2\le \langle Aw,w-u^{h'}\rangle-\langle Au^{h'},w^{h'}-u^{h'}\rangle
-\langle Au^{h'},w-w^{h'}\rangle.
\label{Pr5}
\end{equation}
We take $v^{h'}=w^{h'}$ in \eqref{eq1h} with $h=h'$ to obtain
\begin{equation}
- \langle Au^{h'}, w^{h'} - u^{h'} \rangle\le \Phi(w^{h'})- \Phi(u^{h'})
+I_\Delta(\psi^0(\gamma_\psi u^{h'};\gamma_\psi w^{h'}-\gamma_\psi u^{h'})) -\langle f,w^{h'}-u^{h'}\rangle.
\label{Pr8}
\end{equation}
By the triangle inequality of the norm,
\[ \|u^{h'}-w^{h'}\|_{V} \le\|u^{h'}-w\|_{V}+\|w-w^{h'}\|_{V}. \]
Apply the sub-additivity property of the generalized directional derivative (Proposition \ref{subdiff}\,(ii)), 
\begin{align}
\psi^0(\gamma_\psi u^{h'};\gamma_\psi w^{h'}-\gamma_\psi u^{h'})
&\le \psi^0(\gamma_\psi u^{h'};\gamma_\psi w-\gamma_\psi u^{h'})
+ \psi^0(\gamma_\psi u^{h'};\gamma_\psi w^{h'}-\gamma_\psi w)\nonumber\\
& =\left[\psi^0(\gamma_\psi u^{h'};\gamma_\psi w-\gamma_\psi u^{h'})
+\psi^0(\gamma_\psi w;\gamma_\psi u^{h'}-\gamma_\psi w)\right] \nonumber\\
& \quad{} + \left[\psi^0(\gamma_\psi u^{h'};\gamma_\psi w^{h'}-\gamma_\psi w)
-\psi^0(\gamma_\psi w;\gamma_\psi u^{h'}-\gamma_\psi w)\right].  \label{Pr31}
\end{align}
By the assumption \eqref{eq4},
\[ \psi^0(\gamma_\psi u^{h'};\gamma_\psi w-\gamma_\psi u^{h'})+\psi^0(\gamma_\psi w;\gamma_\psi u^{h'}-\gamma_\psi w)
\le \alpha_\psi \left|\gamma_\psi ( w-u^{h'})\right|_{\mathbb{R}^m}^2. \]
Then, recalling \eqref{eq12w}, we have
\begin{equation}
I_\Delta(\psi^0(\gamma_\psi u^{h'};\gamma_\psi w-\gamma_\psi u^{h'})+\psi^0(\gamma_\psi w;\gamma_\psi u^{h'}-\gamma_\psi w))
\le \alpha_\psi c_\Delta^2\|w-u^{h'}\|_V^2.
\label{Pr32}
\end{equation}
We use the bounds \eqref{Pr8}, \eqref{Pr31} and \eqref{Pr32} in \eqref{Pr5} to obtain
\begin{align}
\left(m_A -\alpha_\psi c_\Delta^2\right) \|w-u^{h'}\|_V^2
& \le \langle Aw,w-u^{h'}\rangle-\langle Au^{h'},w-w^{h'}\rangle-\langle f,w^{h'}-u^{h'}\rangle\nonumber\\
&\quad{} +\Phi(w^{h'})-\Phi(u^{h'}) \nonumber\\
&\quad{} +I_\Delta(\psi^0(\gamma_\psi u^{h'};\gamma_\psi w^{h'}-\gamma_\psi w)
-\psi^0(\gamma_\psi w;\gamma_\psi u^{h'}-\gamma_\psi w)).
\label{Pr20}
\end{align}

Now consider the limits of the terms on the right side of \eqref{Pr20} as $h'\to 0$.
From the weak convergence \eqref{Pr1},
\[ \langle Aw,w-u^{h'}\rangle\to 0. \]
Since $\{u^{h'}\}$ is bounded and $A$ is continuous, $\{A u^{h'}\}$ is bounded.  
Thus, from the strong convergence \eqref{Pr4},
\[ \langle Au^{h'},w-w^{h'}\rangle\to 0. \]
Write
\[ \langle f,w^{h'}-u^{h'}\rangle= \langle f,w^{h'}-w\rangle+\langle f,w-u^{h'}\rangle.\]
From \eqref{Pr4},
\[ \langle f,w^{h'}-w\rangle\to 0. \]
From \eqref{Pr1},
\[ \langle f,w-u^{h'}\rangle\to 0. \]
Hence,
\[ \langle f,w^{h'}-u^{h'}\rangle \to 0. \]
Since $\Phi$ is continuous, from the convergence \eqref{Pr4},
\[ \Phi(w^{h'})\to \Phi(w). \]
The convexity and continuity of $\Phi$ imply that $\Phi$ is weakly sequentially lower semicontinuous.
Thus, due to the weak convergence \eqref{Pr1},
\[ \limsup_{h'\to 0}\left[-\Phi(u^{h'})\right]=-\liminf_{h'\to 0}\Phi(u^{h'})\le -\Phi(w). \]
By $H(\psi)$,
\[ I_\Delta(\psi^0(\gamma_\psi u^{h'};\gamma_\psi w^{h'}-\gamma_\psi w))
\le I_\Delta\!\left(c\left(1+|\gamma_\psi u^{h'}|_{\mathbb{R}^m}\right)|\gamma_\psi w^{h'}-\gamma_\psi w|_{\mathbb{R}^m}\right).\]
Since $\{\gamma_\psi u^{h'}\}$ is bounded in $V_\psi$, by \eqref{Pr4},
\[ \limsup_{h'\to 0}I_\Delta(\psi^0(\gamma_\psi u^{h'};\gamma_\psi w^{h'}-\gamma_\psi w))
\le \limsup_{h'\to 0} c\left(1+\|u^{h'}\|_V\right)\|w^{h'}-w\|_V = 0. \]

Apply Lemma \ref{lem:int} with $\xi_w(\bx)\in \partial \psi(\gamma_\psi w(\bx))$ for a.e.\ $\bx\in\Delta$, and note that
$\xi_w\in V_\psi$.  From the definition of the generalized directional derivative,
\[ \psi^0(\gamma_\psi w(\bx);\gamma_\psi u^{h'}(\bx)-\gamma_\psi w(\bx)) \ge
\langle \xi_w(\bx),\gamma_\psi u^{h'}(\bx)-\gamma_\psi w(\bx)\rangle\]
for a.e.\ $\bx\in\Delta$.  Then,
\[ -I_\Delta(\psi^0(\gamma_\psi w;\gamma_\psi u^{h'}-\gamma_\psi w))
\le -\langle \xi_w,\gamma_\psi u^{h'}-\gamma_\psi w\rangle.\]
Note that as $h'\to 0$,
\[ \langle \xi_w,\gamma_\psi u^{h'}-\gamma_\psi w\rangle \to 0. \]
Hence,
\[ \limsup_{h'\to 0}\left[-I_\Delta(\psi^0(\gamma_\psi w;\gamma_\psi u^{h'}-\gamma_\psi w))\right] \le 0. \]

Summarizing, we take the upper limit of both sides of \eqref{Pr20} as $h'\to0$ to conclude that
\[ \limsup_{h'\to 0}\|w-u^{h'}\|_V^2\le 0. \]
In other words, we have the strong convergence \eqref{Pr3}.

\smallskip
\noindent\emph{Step 3}.  In the last step, we show that the limit $w$ is the unique solution of Problem \ref{prob:VHI}.
Fix an arbitrary element $v\in K$.  By \eqref{3.3h}, we can find a sequence $\{v^{h'}\}\subset V$,
$v^{h'}\in K^{h'}$,  such that
$v^{h'}\to v$ in $V$ and $\gamma_\psi v^{h'}\to \gamma_\psi v$ in $V_\psi$.  By \eqref{eq1h} with $h=h'$,
\begin{equation}
\langle Au^{h'}, v^{h'} - u^{h'} \rangle  + \Phi(v^{h'})- \Phi(u^{h'})
+I_\Delta(\psi^0(\gamma_\psi u^{h'};\gamma_\psi v^{h'}-\gamma_\psi u^{h'}))\ge\langle f,v^{h'}-u^{h'}\rangle.
\label{Pr41}
\end{equation}
As $h'\to 0$,
\begin{equation}
\langle Au^{h'}, v^{h'} - u^{h'} \rangle\to \langle Aw, v-w\rangle,\quad
\langle f,v^{h'}-u^{h'}\rangle\to \langle f,v-w\rangle,
\label{Pr42}
\end{equation}
where the continuity of $A$ is used.  Moreover, by the continuity of $\Phi$,
\begin{equation}
\Phi(v^{h'})\to \Phi(v), \quad \Phi(u^{h'}) \to \Phi(w).
\label{5.21a}
\end{equation}
Note that $\gamma_\psi u^{h'}\to \gamma_\psi w$ and $\gamma_\psi v^{h'}\to \gamma_\psi v$ a.e.\ in $\Delta$.  So
\begin{equation}
I_\Delta(\psi^0(\gamma_\psi w;\gamma_\psi v-\gamma_\psi w))\ge
\limsup_{h'\to 0} I_\Delta(\psi^0(\gamma_\psi u^{h'};\gamma_\psi v^{h'}-\gamma_\psi u^{h'})).
\label{Pr45}
\end{equation}
Taking the upper limit $h'\to 0$ in \eqref{Pr41} and making use of the relations
\eqref{Pr42}--\eqref{Pr45}, we obtain
\[ \langle Aw,v-w\rangle+\Phi(v)-\Phi(w)+I_\Delta(\psi^0(\gamma_\psi w;\gamma_\psi v-\gamma_\psi w))\ge\langle f,v-u\rangle. \]
Note that the element $v\in K$ is arbitrary. This shows that $w$ is a solution of Problem \ref{prob:VHI}.   
Due to the uniqueness of a solution of Problem \ref{prob:VHI}, $w=u$.  Furthermore, since the limit $u$ does not depend on the subsequence,
the entire family of the numerical solutions converges, i.e., \eqref{4.9a} holds.  \hfill
\end{proof}

The convergence result in Theorem \ref{thm:converge} is rather general, and here we consider two special cases.

First, we consider the case of a hemivariational inequality with the choice $\Phi\equiv 0$ in Problem \ref{prob:VHI}.

\begin{problem}\label{P21}
{\it Find an element  $u \in K$ such that}
\begin{equation}
\langle Au,v-u\rangle+I_\Delta(\psi^0(\gamma_\psi u;\gamma_\psi v-\gamma_\psi u))\ge\langle f,v-u\rangle\quad\forall\,v\in K.
\label{hv2}
\end{equation}
\end{problem}

The corresponding numerical method Problem \ref{prob:VHIh} takes the following form.

\begin{problem} \label{P2h}
{\it Find an element $u^h\in K^h$ such that}
\begin{equation}
\langle Au^h,v^h-u^h\rangle+I_\Delta(\psi^0(\gamma_\psi u^h;\gamma_\psi v^h-\gamma_\psi u^h))
\ge\langle f,v^h-u^h\rangle\quad\forall\,v^h\in K^h.
\label{hvh2}
\end{equation}
\end{problem}

\begin{theorem}\label{thm:converge2}
Assume $H(V)$, $H(K)$, $H(A)$, $H(\psi)$,  $H(f)$, and $\alpha_\psi c_\Delta^2<m_A$.
Moreover, assume $V^h$ is a finite-dimensional subspace of $V$, $K^h$ is a non-empty, closed and convex subset
of $V^h$, and \eqref{3.3i}--\eqref{3.3h} hold.  Let $u$ and $u^h$ be the solutions of Problem \ref{P21} and
Problem \ref{P2h}, respectively. Then we have the convergence:
\[ u^h\to u\quad{\rm in\ }V\ {\rm as\ }h\to 0. \]
\end{theorem}

As another particular case, we consider a variational inequality, obtained from Problem \ref{prob:VHI}
by setting $\psi\equiv 0$.

\begin{problem}\label{P4}
{\it Find an element  $u \in K$ such that}
\begin{equation}
\langle Au,v-u\rangle+\Phi(v)-\Phi(u)\ge\langle f,v-u\rangle\quad\forall\,v\in K.
\label{hv4}
\end{equation}
\end{problem}

The numerical method is the following.

\begin{problem} \label{P4h}
{\it Find an element $u^h\in K^h$ such that}
\begin{equation}
\langle Au^h, v^h - u^h \rangle + \Phi(v^h)- \Phi(u^h)\ge\langle f,v^h-u^h\rangle\quad\forall\,v^h\in K^h.
\label{hvh4}
\end{equation}
\end{problem}

\begin{theorem}\label{thm:converge4}
Assume $H(V)$, $H(K)$, $H(A)$, $H(\Phi)$, and $H(f)$.  Moreover, assume $V^h$ is a finite-dimensional subspace
of $V$, $K^h$ is a non-empty, closed and convex subset of $V^h$, and \eqref{3.3i}--\eqref{3.3h} hold.
Let $u$ and $u^h$ be the solutions of Problem \ref{P4} and Problem \ref{P4h}, respectively.
Then we have the convergence:
\[ u^h\to u\quad{\rm in\ }V\ {\rm as\ }h\to 0. \]
\end{theorem}

We comment that this result is Theorem 11.4.1 in \cite{AH2009}.

\subsection{Error estimation}\label{sec:est}

We now turn to the derivation of error estimates for the numerical solution defined by
Problem \ref{prob:VHIh} for the approximation of the solution of Problem \ref{prob:VHI}. For this purpose,
we do not assume \eqref{3.3i} and \eqref{3.3h}.  Recall that we use $L_A$ and $m_A$ for the
Lipschitz constant and the strong monotonicity constant of the operator $A\colon V\to V^*$.

Let $v\in K$ and $v^h\in K^h$ be arbitrary.  By the strong monotonicity of $A$,
\[ m_A\|u-u^h\|_V^2\le\langle Au-Au^h,u-u^h\rangle, \]
which is rewritten as
\begin{align}
m_A\|u-u^h\|_V^2&\le\langle Au-Au^h,u-v^h\rangle+\langle Au,v^h-u\rangle+\langle Au, v-u^h\rangle
\nonumber\\
&{}\quad +\langle Au, u-v\rangle + \langle Au^h, u^h- v^h \rangle. \label{err3}
\end{align}
Applying \eqref{eq1},
\[ \langle Au,u-v\rangle\le\Phi(v)-\Phi(u)+I_\Delta(\psi^0(\gamma_\psi u;\gamma_\psi v-\gamma_\psi u))-\langle f,v-u\rangle. \]
Applying \eqref{eq1h},
\[ \langle Au^h, u^h-v^h\rangle \le \Phi(v^h)-\Phi(u^h)
+ I_\Delta(\psi^0(\gamma_\psi u^h;\gamma_\psi v^h-\gamma_\psi u^h))-\langle f,v^h-u^h\rangle. \]
Using these inequalities in \eqref{err3}, after some rearrangement of the terms, we have
\begin{equation}
m_A\|u-u^h\|_V^2 \le\langle Au-Au^h,u-v^h\rangle+R_u(v^h,u)+R_u(v,u^h)+I_\psi(v,v^h),
\label{err7}
\end{equation}
where
\begin{align}
R_u(v,w)&:=\langle Au,v-w\rangle+\Phi(v) -\Phi(w) + I_\Delta(\psi^0(\gamma_\psi u;\gamma_\psi v-\gamma_\psi w))
-\langle f,v-w\rangle, \label{err8} \\[1mm]
I_\psi(v,v^h)&:=I_\Delta(\psi^0(\gamma_\psi u;\gamma_\psi v-\gamma_\psi u)
+\psi^0(\gamma_\psi u^h;\gamma_\psi v^h-\gamma_\psi u^h))\nonumber\\
&{}\quad\ -I_\Delta(\psi^0(\gamma_\psi u;\gamma_\psi v^h-\gamma_\psi u)+\psi^0(\gamma_\psi u;\gamma_\psi v-\gamma_\psi u^h)).\label{err10}
\end{align}
Let us bound the first and the last two terms on the right hand side of \eqref{err7}.  First,
\[ \langle Au- Au^h, u- v^h\rangle \le L_A\|u-u^h\|_V \|u-v^h\|_V.\]
By the modified Cauchy-Schwarz inequality \eqref{mCS}, for any $\epsilon>0$ arbitrarily small,
\begin{equation}
\langle Au- Au^h, u- v^h\rangle \le \epsilon\,\|u-u^h\|_V^2+c\,\|u-v^h\|_V^2  \label{err7a}
\end{equation}
for some constant $c$ depending on $\epsilon$.  Applying the subadditivity of the generalized directional derivative,
\[ \psi^0(z;z_1+z_2)\le \psi^0(z;z_1)+\psi^0(z;z_2)\quad \forall\,z,z_1,z_2\in \mathbb{R}^m, \]
we have
\begin{align*}
\psi^0(\gamma_\psi u;\gamma_\psi v-\gamma_\psi u) &\le \psi^0(\gamma_\psi u;\gamma_\psi v-\gamma_\psi u^h)
  +\psi^0(\gamma_\psi u;\gamma_\psi u^h-\gamma_\psi u),\\
\psi^0(\gamma_\psi u^h;\gamma_\psi v^h-\gamma_\psi u^h) &\le \psi^0(\gamma_\psi u^h;\gamma_\psi v^h-\gamma_\psi u)
  +\psi^0(\gamma_\psi u^h;\gamma_\psi u-\gamma_\psi u^h).
\end{align*}
Thus,
\begin{align*}
I_\psi(v,v^h)&\le I_\Delta(\psi^0(\gamma_\psi u^h;\gamma_\psi v^h-\gamma_\psi u)
-\psi^0(\gamma_\psi u;\gamma_\psi v^h-\gamma_\psi u))\\
&{}\quad +I_\Delta(\psi^0(\gamma_\psi u;\gamma_\psi u^h-\gamma_\psi u)
+\psi^0(\gamma_\psi u^h;\gamma_\psi u-\gamma_\psi u^h)).
\end{align*}
By \eqref{eq4} and \eqref{eq12w},
\begin{align*}
I_\Delta(\psi^0(\gamma_\psi u;\gamma_\psi u^h-\gamma_\psi u)+\psi^0(\gamma_\psi u^h;\gamma_\psi u-\gamma_\psi u^h))
&\le \alpha_\psi I_\Delta(|\gamma_\psi u-\gamma_\psi u^h|^2)\\
&\le \alpha_\psi c_\Delta^2 \|u- u^h\|_{V}^2.
\end{align*}
Moreover, by \eqref{3.32a},
\begin{align*}
\left|I_\Delta(\psi^0(\gamma_\psi u^h;\gamma_\psi v^h-\gamma_\psi u))\right|&\le c\left(1+\|u^h\|_V\right)\|v^h-u\|_V,\\
\left|I_\Delta(\psi^0(\gamma_\psi u;\gamma_\psi v^h-\gamma_\psi u))\right| &\le c\left(1+\|u\|_V\right)\|v^h-u\|_V.
\end{align*}
Combining the above four inequalities and using the fact that $\{\|u^h\|_{V}\}$ is
bounded independent of $h$ (cf.\ Proposition \ref{prop:bd}), we find that
\begin{equation}
I_\psi(v,v^h)\le \alpha_\psi\|\gamma_\psi u-\gamma_\psi u^h\|_{V_\psi}^2+c\,\|\gamma_\psi u-\gamma_\psi v^h\|_{V_\psi}
\label{err13}
\end{equation}
for some constant $c>0$ independent of $h$. Using \eqref{err7a} and \eqref{err13} in \eqref{err7}, we have
\[  \left(m_A-\alpha_\psi c_\Delta^2-\epsilon\right)\|u-u^h\|_V^2
 \le c\,\|u-v^h\|_V^2+c\,\|\gamma_\psi u-\gamma_\psi v^h\|_{V_\psi} +R_u(v^h,u)+R_u(v,u^h). \]
Recall the smallness assumption $\alpha_\psi c_\Delta^2<m_A$.  We can choose
$\epsilon=(m_A-\alpha_\psi c_\Delta^2)/2>0$ and get the inequality
\[ \|u-u^h\|_V^2\le c\inf_{v^h\in K^h}\left[\|u-v^h\|_V^2+\|\gamma_\psi (u-v^h)\|_{V_\psi}+R_u(v^h,u)\right]
+c\inf_{v\in K}R_u(v,u^h). \]

We summarize the result in the form of a theorem.

\begin{theorem}\label{thm:num_1}
Assume $H(K)$, $H(A)$, $H(\Phi)$, $H(\psi)$,  $H(f)$, and $\alpha_\psi c_\Delta^2<m_A$.
Then for the solution $u$ of Problem \ref{prob:VHI} and the solution $u^h$ of Problem \ref{prob:VHIh},
we have the C\'{e}a-type inequality
\begin{equation}
\|u-u^h\|_V^2\le c\inf_{v^h\in K^h}\left[\|u-v^h\|_V^2+\|\gamma_\psi (u-v^h)\|_{V_\psi}+R_u(v^h,u)\right]
+c\inf_{v\in K} R_u(v,u^h). \label{err16}
\end{equation}
\end{theorem}

For internal approximations, $K^h\subset K$ and then
\[ \inf_{v\in K} R_u(v,u^h) =0. \]
So for internal approximations, the C\'{e}a-type inequality \eqref{err16} simplifies to
\begin{equation}
\|u-u^h\|_V^2 \le  c\inf_{v^h\in K^h}\left[\|u-v^h\|_V^2 +\|\gamma_\psi (u-v^h)\|_{V_\psi}+R_u(v^h,u)\right].
\label{err16.aa}
\end{equation}

We also remark that in the literature on error analysis of numerical solutions of variational inequalities,
it is standard that the C\'{e}a-type inequalities involve square root of approximation error of the
solution in certain norms due to the inequality form of the problems; cf.\ \cite{Fa74, KO1988, HS2002}.

To proceed further, we need to bound the residual term \eqref{err8} and this depends on the problem to be solved.

\section{Studies of the contact problems}\label{sec:contact}

In this section, we take Problem~\ref{prob:cont1} as an example for detailed theoretical studies.   We first explore the solution existence and uniqueness, then introduce a
linear finite element method to solve the problem and derive an optimal order error estimate under certain
solution regularity assumptions.  Finally, we present numerical simulation results for solving some contact problems.

\subsection{Studies of Problem~\ref{prob:cont1}}\label{subsec:cont1}

We start with an existence and uniqueness result for Problem~\ref{prob:cont1}.

\begin{theorem}\label{t21}
Assume \eqref{Ass:E}, \eqref{Ass:f}, \eqref{psi_nu}, $f_b\ge 0$, and
\begin{equation}\label{8sma}
\alpha_{\psi_\nu} \lambda_{\nu}^{-1} < m_{{\cal E}}.
\end{equation}
Then Problem \ref{prob:cont1} has a unique solution.
\end{theorem}
\begin{proof}
We apply Theorem~\ref{thm1} by placing Problem \ref{prob:cont1} in the framework of Problem \ref{prob:VHI}
with the following choices of the data: the space $V$ and the set $K$ are both the space $\bV$ defined 
in \eqref{SpV}, $\Delta=\Gamma_C$, the space $V_\psi$ of \eqref{SpVpsi} is $V_\psi=L^2(\Gamma_C)$, 
$\gamma_\psi\colon V\to V_\psi$ is the normal component trace operator, $\psi=\psi_\nu$,
and $A\colon \bV \to \bV^*$, $\Phi\colon \bV\to\mathbb{R}$, $f=\fb$ are defined by
\begin{align}
& \langle A\bu,\bv\rangle =\int_{\Omega}{\cal E}\bvarepsilon(\bu):\bvarepsilon(\bv)\,dx,\quad \bu,\bv\in \bV,
\label{8b1}\\[1mm]
& \Phi(\bv)=\int_{\Gamma_C}f_b\,|\bv_\tau|\,ds,\quad \bv\in \bV,
\label{8b3}\\[1mm]
& \langle\fb,\bv\rangle =\int_{\Omega}\fb_0\cdot\bv\,dx+\int_{\Gamma_N}\fb_2\cdot\bv\,ds,\quad \bv\in\bV.
\label{8ef}
\end{align}

Let us examine the assumptions stated in Theorem~\ref{thm1}.  The assumptions $H(V)$ and $H(K)$ are
the same and they are obviously true.  For the operator $A$ defined by \eqref{8b1}, we claim that $H(A)$
holds true with $L_A=L_{\cal E}$ and $m_A=m_{\mathcal F}$. Indeed, for $\bu,\bv,\bw\in V$,
by assumption (\ref{Ass:E})\,(a), we have
\[ \langle A\bu-A\bv,\bw\rangle \le ({\cal E}\bvarepsilon(\bu)-{\cal E}\bvarepsilon(\bv),
\bvarepsilon(\bw))_{\mathbb{Q}} \leq L_{\cal E}\|\bu-\bv\|_{\boldsymbol V}\|\bw\|_{\boldsymbol V}.\]
Thus,
\[ \|A\bu-A\bv\|_{{\boldsymbol V}^*}\le L_{\cal E}\|\bu-\bv\|_{\boldsymbol V}\quad\forall\,\bu,\bv\in\bV. \]
This shows that $A$ is Lipschitz continuous. Moreover,
\[\langle A\bu-A\bv,\bu-\bv\rangle=({\cal E}\bvarepsilon(\bu)-{\cal E}\bvarepsilon(\bv),
	\bvarepsilon(\bu)-\bvarepsilon(\bv))_{\mathbb{Q}}. \]
Then, assumption (\ref{Ass:E})\,(b) yields
\begin{equation}
\langle A\bu-A\bv,\bu-\bv\rangle\geq m_{\cal E}\|\bu-\bv\|_{\boldsymbol V}^2\quad\forall\,\bu,\bv\in \bV.
\label{Astrmo}
\end{equation}
This shows that the monotonicity condition \eqref{eq2z} is satisfied with $m_A=m_{\mathcal F}$.

Next, for $\Phi$ defined by \eqref{8b3}, it is easy to see that $\Phi\colon \bV\to \mathbb{R}$ is 
continuous and convex. The potential function $\psi_\nu\colon \mathbb{R}\to\mathbb{R}$
is assumed to satisfy $H(\psi_\nu)$ with $m=1$.  
For $\fb$, assumption \eqref{Ass:f} implies $H(f)$.  Considering the above relationships
among constants and noting that $c_\Delta=\lambda_{\nu}^{-1/2}$, we see that assumption \eqref{8sma}
implies the smallness condition $\alpha_\psi c_\Delta^2<m_A$ in Theorem~\ref{thm1}.

Therefore, we can apply Theorem~\ref{thm1} to conclude that there exists a unique element $\bu\in\bV$ such that
\eqref{cont7} is satisfied.   \hfill
\end{proof}

Theorem~\ref{t21} provides the unique weak solvability of the contact problem, in terms of the displacement.
Once the displacement field is obtained by solving the contact problem, the stress field $\bsigma$
is uniquely determined by using the constitutive law \eqref{cont2}.

We proceed with the discretization of Problem \ref{prob:cont1} using the finite element method.
For simplicity, assume $\Omega$ is a polygonal/polyhedral domain and express the three parts of the boundary,
as unions of closed flat components with disjoint interiors:
\[ \overline{\Gamma_Z}=\cup_{i=1}^{i_Z}\Gamma_{Z,i},\quad Z=D,N,C.\]
Let $\{{\cal T}^h\}$ be a regular family of partitions of $\overline{\Omega}$
into triangles/tetrahedrons that are compatible with the partition of
the boundary $\partial\Omega$ into $\Gamma_{Z,i}$, $1\le i\le i_Z$, $Z=D,N,C$,
in the sense that if the intersection of one side/face of an
element with one set $\Gamma_{Z,i}$ has a positive measure with
respect to $\Gamma_{Z,i}$, then the side/face lies entirely in $\Gamma_{Z,i}$.
Then construct a linear element space corresponding to ${\cal T}^h$,
\begin{equation}
\bV^h=\left\{\bv^h\in C(\overline{\Omega})^d \mid \bv^h|_T\in \mathbb{P}_1(T)^d
   \ {\rm for}\ T\in {\cal T}^h,\ \bv^h=\bzero\ {\rm on\ }\overline{\Gamma_D}\right\}.
\label{Vh}
\end{equation}

For any function $\bw\in H^2(\Omega)^d$, by the Sobolev embedding $H^2(\Omega)\subset C(\overline{\Omega})$
valid for $d\le 3$, we know that $\bw\in C(\overline{\Omega})^d$ and so its finite element interpolant
$\Pi^h\bw\in \bV^h$ is well defined.  Moreover, the following error estimate holds (cf.\ any of the 
references \cite{AH2009, BS2008, Ci1978}): for some constant $c>0$ independent of $h$,
\begin{equation}
\|\bw-\Pi^h\bw\|_{L^2(\Omega)^d}+h\,\|\bw-\Pi^h\bw\|_{H^1(\Omega)^d}\le c\,|\bw|_{H^2(\Omega)^d}
\quad\forall\,\bw\in H^2(\Omega)^d.
\label{estimate}
\end{equation}

The finite element approximation of Problem \ref{prob:cont1} is the following.

\begin{problem}\label{p2h}
{\it Find a displacement field $\bu^h\in \bV^h$ such that}
\begin{align}
&\int_{\Omega}{\cal E}\bvarepsilon(\bu^h):(\bvarepsilon(\bv^h)-\bvarepsilon(\bu^h))\,dx
+ \int_{\Gamma_C}f_b\left(|\bv^h_\tau|-|\bu^h_\tau|\right) ds \nonumber\\[1mm]
&\qquad{} +\int_{\Gamma_C} \psi^0_\nu (u_\nu^h; v_\nu^h-u_\nu^h)\,ds\ge \langle\fb, \bv^h-\bu^h\rangle \quad\forall\,\bv^h\in \bV^h.
\label{s7.2a}
\end{align}
\end{problem}

Similar to Problem \ref{prob:cont1}, we can apply a discrete analog of the arguments in the proof 
of Theorem \ref{t21} to conclude that Problem \ref{p2h} admits a unique solution $\bu^h\in \bV^h$.

For an error analysis, we notice that by Theorem \ref{thm:num_1},
\begin{equation}
\|\bu-\bu^h\|_{\boldsymbol V}^2  \le c\inf_{{\boldsymbol v}^h\in {\boldsymbol V}^h}\left[\|\bu-\bv^h\|_{\boldsymbol V}^2
+\|u_\nu-v^h_\nu\|_{L^2(\Gamma_C)}+R_{\boldsymbol u}(\bv^h,\bu)\right],
\label{s7.2aaa}
\end{equation}
where the residual-type term from \eqref{err8} is
\begin{align}
R_{\boldsymbol u}(\bv^h,\bu)&=({\cal E}(\bvarepsilon(\bu)),\bvarepsilon(\bv^h-\bu))_{\mathbb{Q}}
+  \int_{\Gamma_C}f_b\left(|\bv_\tau^h|-|\bu_\tau|\right) ds \nonumber\\[1mm]
&\quad{}+\int_{\Gamma_C} \psi^0_\nu (u_\nu; v_\nu^h-u_\nu)\,ds- \langle\fb, \bv^h-\bu\rangle. \label{s7.2c}
\end{align}

To proceed further, we make the following solution regularity assumptions:
\begin{equation}
\bu\in H^2(\Omega;\mathbb{R}^d),\quad \bsigma={\cal E}(\bvarepsilon(\bu))\in H^1(\Omega;\mathbb{S}^d).
\label{s7.2b}
\end{equation}
In many application problems, $\bsigma\in H^1(\Omega;\mathbb{S}^d)$ follows from
$\bu\in H^2(\Omega;\mathbb{R}^d)$, e.g., if the material is linearly elastic with
suitably smooth coefficients, or if the elasticity operator ${\cal E}$ depends on $\bx$ smoothly.
In the latter case, we recall that ${\cal E}(\bx,\bvarepsilon)$ is a Lipschitz function of $\bvarepsilon$,
and the composition of a Lipschitz continuous function and an $H^1(\Omega)$ function is an $H^1(\Omega)$ function.
Note that $\bsigma\in H^1(\Omega;\mathbb{S}^d)$ implies
\begin{equation}
\bsigma\bnu\in L^2(\Gamma;\mathbb{R}^d).
\label{s7.2bb}
\end{equation}
For an appropriate upper bound on $R_{\boldsymbol u}(\bv^h,\bu)$ defined in \eqref{s7.2c},
we need to derive some point-wise relations for the weak solution $\bu$ of Problem \ref{prob:cont1}.
We follow a procedure found in \cite[Section 8.2]{HS2002}.  Introduce a subspace $\tilde{\bV}$ of $\bV$ by
\begin{equation}
\tilde{\bV}:=\left\{\,\bw\in C^\infty(\overline{\Omega};\mathbb{R}^d)\mid
 \bw=\bzero \ {\rm on\ }\Gamma_D\cup\Gamma_C\,\right\}.
\label{s7.2cc}
\end{equation}
We take $\bv=\bu+\bw$ with $\bw\in \tilde{\bV}$ in \eqref{cont7} to get
\[ \int_\Omega {\cal E}(\bvarepsilon(\bu)):\bvarepsilon(\bw)\,dx\ge
\int_\Omega \fb_0{\cdot}\bw\,dx+\int_{\Gamma_N} \fb_2{\cdot}\bw\,ds. \]
By replacing $\bw\in \tilde{\bV}$ with $-\bw\in \tilde{\bV}$ in the above inequality, we find the equality
\begin{equation}
\int_\Omega {\cal E}(\bvarepsilon(\bu)):\bvarepsilon(\bw)\,dx=
\int_\Omega \fb_0{\cdot}\bw\,dx+\int_{\Gamma_N} \fb_2{\cdot}\bw\,ds \quad\forall\, \bw \in \tilde{\bV}.
\label{s7.2bbb}
\end{equation}
Thus,
\[ \int_\Omega {\cal E}(\bvarepsilon(\bu)):\bvarepsilon(\bw)\,dx=
\int_\Omega \fb_0{\cdot}\bw\,dx\quad\forall\, \bw \in  C^\infty_0(\Omega;\mathbb{R}^d), \]
and so in the distributional sense,
\[ {\rm div}\,{\cal E}(\bvarepsilon({\bu}))+\fb_0=\bzero.\]
Since ${\cal E}(\bvarepsilon(\bu))\in H^1(\Omega;\mathbb{S}^d)$ and $\fb_0\in L^2(\Omega;\mathbb{R}^d)$,
the above equality holds pointwise:
\begin{equation}
{\rm div}\,{\cal E}(\bvarepsilon({\bu}))+\fb_0=\bzero\quad{\rm a.e.\ in\ }\Omega. \label{s7.2d}
\end{equation}
Performing integration by parts in \eqref{s7.2bbb} and using the relation \eqref{s7.2d}, we have
\[\int_{\Gamma_N} \bsigma\bnu{\cdot}\bw\,ds=\int_{\Gamma_N} \fb_2{\cdot}\bw\,ds
  \quad\forall\, \bw \in \tilde{\bV}.\]
Since $\bsigma\bnu\in L^2(\Gamma;\mathbb{R}^d)$ (cf.\ \eqref{s7.2bb}) and $\bw \in \tilde{\bV}$
is arbitrary, we derive from the above equality that
\begin{equation}
\bsigma\bnu=\fb_2\quad{\rm a.e.\ on\ }\Gamma_N. \label{s7.2e}
\end{equation}
Now multiply \eqref{s7.2d} by $\bv-\bu$ with $\bv\in \bV$, integrate over $\Omega$, and integrate by parts
to get
\[ \int_{\Gamma}\bsigma\bnu{\cdot}(\bv-\bu)\,ds-\int_\Omega{\cal E}(\bvarepsilon(\bu)):\bvarepsilon(\bv-\bu)\,dx
+\int_\Omega \fb_0{\cdot}(\bv-\bu)\,dx=0, \]
i.e.,
\begin{equation}
\int_\Omega {\cal E}(\bvarepsilon(\bu)):\bvarepsilon(\bv-\bu)\,dx=
\langle\fb,\bv-\bu\rangle+\int_{\Gamma_C}\bsigma\bnu{\cdot}(\bv-\bu)\,ds\quad\forall\,\bv\in \bV.
\label{s7.2f}
\end{equation}
Thus,
\[ R_{\boldsymbol u}(\bv^h,\bu)=\int_{\Gamma_C} \left[ \bsigma\bnu{\cdot}(\bv^h-\bu)
 +f_b \left(|\bv^h_\tau|-|\bu_\tau|\right)+\psi_\nu^0(u_\nu;v^h_\nu-u_\nu)\right] ds,\]
and then,
\begin{equation}
\left|R_{\boldsymbol u}(\bv^h,\bu)\right| \le c\,\|\bu-\bv^h\|_{L^2(\Gamma_C)^d}.
\label{s7.2g}
\end{equation}
Finally, from \eqref{s7.2aaa}, we have the inequality
\begin{equation}
\|\bu-\bu^h\|_{\boldsymbol V} \le c\inf_{{\boldsymbol v}^h\in {\boldsymbol V}^h}\left[\|\bu-\bv^h\|_{\boldsymbol V}
+\|\bu-\bv^h\|_{L^2(\Gamma_C)^d}^{1/2}\right].
\label{s7.2h}
\end{equation}
Under additional solution regularity assumption
\begin{equation}
\bu|_{\Gamma_{C,i}}\in H^2(\Gamma_{C,i};\real^d), \quad 1\le i\le i_C,  \label{s7.2i}
\end{equation}
for the finite element interpolant $\Pi^h\bu$, we have
\begin{equation}
\|\bu-\Pi^h\bu\|_{L^2(\Gamma_C)^d}\le c\,h^2.
\label{5.20a}
\end{equation}
Then we derive from \eqref{s7.2h} the following optimal order error bound
\begin{equation}
\|\bu-\bu^h\|_{\boldsymbol V}\le c\left[\|\bu-\Pi^h\bu\|_{\boldsymbol V}
+\|\bu-\Pi^h\bu\|_{L^2(\Gamma_C)^d}^{1/2}\right]\le c\,h,
\label{s7.2j}
\end{equation}
where the constant $c$ depends on the quantities $\|\bu\|_{H^2(\Omega;\mathbb{R}^d)}$,
$\|\bsigma\bnu\|_{L^2(\Gamma_C;\mathbb{R}^d)}$ and $\|\bu\|_{H^2(\Gamma_{C,i};\mathbb{R}^d)}$
for $1\le i\le i_C$.

\smallskip
We comment that similar results hold for the frictionless version of the model, i.e., where the
friction condition \eqref{cont6} is replaced by
\begin{equation}
\bsigma_\tau=\bzero\quad{\rm on}\ \Gamma_C.
\label{s7.2k}
\end{equation}
Then the problem is to solve the inequality \eqref{cont7} without the term
\[ \int_{\Gamma_C}f_b \left(|\bv_\tau|-|\bu_\tau|\right) ds,  \]
i.e., to find a displacement field $\bu\in \bV$ such that
\begin{equation}
\int_{\Omega}{\cal E}\bvarepsilon({\bu}):\bvarepsilon(\bv)\,dx
+\int_{\Gamma_C}\psi^0_\nu (u_\nu; v_\nu)\, ds \ge \langle\fb, \bv\rangle\quad\forall\,\bv\in \bV.
\label{s7.2l}
\end{equation}
The inequality \eqref{s7.2h} and the error bound \eqref{s7.2j} still hold for the linear finite
element solutions under the solution regularity conditions \eqref{s7.2b} and \eqref{s7.2i}.

\subsection{Studies of Problem~\ref{prob:cont2}}\label{subsec:cont2}

Problem~\ref{prob:cont2} is simpler to analyze than Problem~\ref{prob:cont1} in the sense that the inequality 
\eqref{cont9} does not include the non-smooth convex terms $I_{\Gamma_C}(f_b|\bv_\tau|)$ 
and $I_{\Gamma_C}(f_b|\bu_\tau|)$.  Similar to Theorem \ref{t21}, we have the next result result,
derived from Theorem~\ref{thm1}.

\begin{theorem}\label{t22}
Assume \eqref{Ass:E}, \eqref{Ass:f}, \eqref{psi_tau}, and
\begin{equation*}  
\alpha_{\psi_\tau} \lambda_{\tau,1}^{-1} < m_{{\cal E}}.
\end{equation*}
Then Problem \ref{prob:cont2} has a unique solution.
\end{theorem}

For the finite element solution of Problem~\ref{prob:cont2}, we keep the setting on the finite element
partitions of $\overline{\Omega}$ in Subsection \ref{subsec:cont2}.  Then we introduce a subspace
of $\bV^h$ of \eqref{Vh}:
\begin{equation}
\bV^h_1=\left\{\bv^h\in \bV^h\mid v^h_\nu=0\ {\rm on}\ \overline{\Gamma_C}\right\}. 
\label{Vh1}
\end{equation}
Note that the constraint ``$v^h_\nu=0$ on $\overline{\Gamma_C}$'' is equivalent to 
``$v^h_\nu=0$ at all nodes on $\overline{\Gamma_C}$''.
The finite element method for solving Problem~\ref{prob:cont2} is the following.

\begin{problem}\label{prob:cont2h}
{\it Find a displacement field $\bu^h\in \bV^h_1$ such that}
\begin{equation}
({\cal E}(\bvarepsilon(\bu^h)),\bvarepsilon(\bv^h))_\mathbb{Q}
+I_{\Gamma_C} (\psi_\tau^0 (\bu^h_\tau; \bv^h_\tau))\ge\langle \fb,\bv^h\rangle\quad\forall\,\bv^h\in \bV^h_1.
\label{cont9h}
\end{equation}
\end{problem}

Under the assumptions stated in Theorem \ref{t22}, Problem \ref{prob:cont2h} has a unique solution $\bu^h$.  
Moreover, by \eqref{err16.aa},
\begin{equation}
\|\bu-\bu^h\|_{\boldsymbol V}^2  \le c\inf_{{\boldsymbol v}^h\in {\boldsymbol V}^h_1}
\left[\|\bu-\bv^h\|_{\boldsymbol V}^2+\|\bu_\tau-\bv^h_\tau\|_{L^2(\Gamma_C)^d}+R_{\boldsymbol u}(\bv^h-\bu)\right],
\label{Cea:cont2}
\end{equation}
where 
\begin{equation}
R_{\boldsymbol u}(\bw)=({\cal E}(\bvarepsilon(\bu)),\bvarepsilon(\bw))_{\mathbb{Q}}
+\int_{\Gamma_C} \psi^0_\tau (\bu_\tau; \bw_\tau)\,ds- \langle\fb,\bw\rangle. \label{Re:cont2}
\end{equation}

Similar to \eqref{s7.2h}, under the solution regularity condition \eqref{s7.2b},
\begin{equation}
\|\bu-\bu^h\|_{\boldsymbol V} \le c\inf_{{\boldsymbol v}^h\in {\boldsymbol V}^h_1}\left[\|\bu-\bv^h\|_{\boldsymbol V}
+\|\bu-\bv^h\|_{L^2(\Gamma_C)^d}^{1/2}\right].
\label{Cea1:cont2}
\end{equation}
Then, under the solution regularity conditions \eqref{s7.2b} and \eqref{s7.2i}, similar to \eqref{s7.2j},
we can show that
\[ \|\bu-\bu^h\|_{\boldsymbol V}\le c\,h. \]

\subsection{Studies of Problem~\ref{prob:cont3}}\label{subsec:cont3}

For a study of Problem~\ref{prob:cont3}, we apply Theorem~\ref{thm1} to get the following result.

\begin{theorem}\label{t23}
Assume \eqref{Ass:E}, \eqref{Ass:f}, \eqref{psi_nu}, $f_b\ge 0$, $g\in L^2(\Gamma_C)$, $g\ge 0$, and
\begin{equation*}  
\alpha_{\psi_\nu} \lambda_{\nu}^{-1} < m_{{\cal E}}.
\end{equation*}
Then Problem \ref{prob:cont3} has a unique solution.
\end{theorem}

For the finite element approximation of Problem~\ref{prob:cont3}, define
\begin{equation}
\bU^h = \left\{ \bv^h\in\bV^h \mid v^h_\nu\le g\ {\rm at\ all\ nodes\ on}\ \overline{\Gamma_C}\right\}.
\label{Uh}
\end{equation}
Then the finite element method for solving Problem~\ref{prob:cont3} is the following.

\begin{problem}\label{prob:cont3h}
{\it Find a displacement field $\bu^h\in \bU^h$ such that}
\begin{align}
& ({\cal E}(\bvarepsilon(\bu^h)),\bvarepsilon(\bv^h) -\bvarepsilon(\bu^h))_\mathbb{Q}
+I_{\Gamma_C}(f_b|\bv^h_\tau|)-I_{\Gamma_C}(f_b|\bu^h_\tau|)
+I_{\Gamma_C} (\psi_\nu^0 (u^h_\nu; v^h_\nu-u^h_\nu))\nonumber\\
&\qquad {} \ge\langle \fb,\bv^h-\bu^h\rangle\quad\forall\,\bv^h\in \bU^h.
\label{cont9_3h}
\end{align}
\end{problem}

Under the assumptions stated in Theorem \ref{t23}, Problem \ref{prob:cont3h} has a unique solution $\bu^h$.  
For the error estimation, for simplicity, we assume $g$ is a concave function.  Then, $\bU^h\subset \bU$,
and similar to \eqref{s7.2aaa}, 
\begin{equation}
\|\bu-\bu^h\|_{\boldsymbol V}^2\le c\inf_{{\boldsymbol v}^h\in {\boldsymbol U}^h}\left[\|\bu-\bv^h\|_{\boldsymbol V}^2
+\|u_\nu-v^h_\nu\|_{L^2(\Gamma_C)}+R_{\boldsymbol u}(\bv^h,\bu)\right].
\label{Cea:cont3}
\end{equation}
where 
\begin{align}
R_{\boldsymbol u}(\bv^h,\bu) & = ({\cal E}(\bvarepsilon(\bu)),\bvarepsilon(\bv^h) -\bvarepsilon(\bu))_\mathbb{Q}
+I_{\Gamma_C}(f_b|\bv^h_\tau|)-I_{\Gamma_C}(f_b|\bu_\tau|)\nonumber\\
&\quad {} +I_{\Gamma_C} (\psi_\nu^0 (u_\nu; v^h_\nu-u_\nu))-\langle \fb,\bv^h-\bu\rangle.
\label{Re:cont3}
\end{align}

Similar to \eqref{s7.2h}, under the solution regularity condition \eqref{s7.2b}, we can derive from \eqref{Cea:cont3} that
\[ \|\bu-\bu^h\|_{\boldsymbol V} \le c\inf_{{\boldsymbol v}^h\in {\boldsymbol U}^h}\left[\|\bu-\bv^h\|_{\boldsymbol V}
+\|\bu-\bv^h\|_{L^2(\Gamma_C)^d}^{1/2}\right]. \]
Again, assuming both \eqref{s7.2b} and \eqref{s7.2i}, we have the optimal order error estimate
\[ \|\bu-\bu^h\|_{\boldsymbol V}\le c\,h. \]

\section{Virtual element method for variational-hemivariational inequality}\label{sec:VEM}

In the previous sections, we studied the FEM to solve the contact problems.  Other numerical methods can be
applied for the contact problems as well.  In this section, we take the virtual element method (VEM) as
an example.  
The VEM was first proposed and analyzed in \cite{beirao2013basic, beirao2013virtual}. The method
has since been applied to a wide variety of mathematical models from applications in science and engineering 
thanks to its strengths in handling complex geometries and problems requiring high-regularity solutions.
The VEM was first applied to solve contact problems in \cite{wriggers2016virtual}.  Further applications of
the VEM can be found in a number of publications, e.g., \cite{AHABW20, CHKW22, WZ21, WWH24}.
The presentation on the VEM here follows \cite{FHH19, FHH21a, WWH21}.

\subsection{An abstract framework}

We reconsider Problem \ref{prob:VHI}, yet for the case where the operator $A\colon V \to V^*$ is generated
by a bilinear form $ a(\cdot, \cdot)\colon V \times V \to \mathbb{R} $ through the relation
\[ a(u,v) = \langle A u, v\rangle \quad \forall\, u, v\in V. \]

In other words, the abstract problem for the study of VEM is the following.

\begin{problem}\label{prob:VHI_bilinear}
Find $u\in K$ such that
\begin{equation}
a(u,v-u) +\Phi(v)-\Phi(u)+I_\Delta(\psi^0(\gamma_\psi u;\gamma_\psi v-\gamma_\psi u))
\ge \langle f,v-u\rangle \quad\forall\,v\in K.
\label{eq1V}
\end{equation}
\end{problem}

We will assume $H(V)$, $H(K)$, $H(\Phi)$, $H(\psi)$, $H(f)$ from Subsection \ref{subsec:abs}.  
Corresponding to $H(A)$, the bilinear form $a(\cdot,\cdot)$ is assumed to be bounded with a boundedness constant $L_a$
and $V$-elliptic with an ellipticity constant $m_a$. By Theorem \ref{thm1}, if we further assume the smallness 
condition $\alpha_\psi c_\Delta^2<m_a$.  We will also assume $a$ is symmetric.  

\smallskip
\noindent\underline{$H(a)$} The bilinear form $a\colon V\times V\to\mathbb{R}$ is symmetric, bounded 
with the boundedness constant $L_a$ and $V$-elliptic with the ellipticity constant $m_a$:
\begin{align}
\left|a(u,v)\right| & \le L_a \|u\|_V \|v\|_V\quad\forall\,u,v\in V, \label{a1}\\
a(v,v) & \ge m_a \|v\|_V^2 \quad\forall\,v\in V. \label{a2}
\end{align}

\smallskip

To develop a general framework for the VEM, let $\Omega$ be the spatial domain of Problem \ref{prob:VHI_bilinear}.  We assume $\Omega$ is a bounded 
polygonal domain, and denote by $\mathcal{T}^h$ a partition of $\overline{\Omega}$ into polygonal elements $\{T\}$. 
Define $ h_T = \text{diam}(T) $ for each element $T$, and define $ h = \max\{ h_T : T \in \mathcal{T}^h \} $
for the mesh-size of the partition $\mathcal{T}^h$. As the bilinear form $a(u, v)$ is typically an integral 
over the domain $\Omega$, we can split it element-wise as
\begin{equation}
a(u, v) = \sum_{T \in \mathcal{T}^h} a_T(u, v),
\label{a0}
\end{equation}
where $ a_T(u, v) $ denotes the restriction of $ a(u, v) $ to $ T $, which is an integral over $T$.  Let $V_T$ be the restriction of $V$ to $T$, which is a function space over $T$.  

\smallskip

For the setting of the VEM, we make the following assumptions.

\smallskip
\noindent\underline{$H(h)$} For each partition $\mathcal{T}^h$, there is a virtual element space $ V^h \subset V$.
For a positive integer $k$, $V_T^h\supset \mathbb{P}_k(T)$, where $V_T^h:=V^h|_T$ is the restriction 
of $V^h$ on $T$.\\
Related to local function spaces $V_T=V|_T$ on elements $T\in \mathcal{T}^h$, we have the 
decomposition \eqref{a0} in which, for each element $T$, $a_T\colon V_T\times V_T\to\mathbb{R}$ is symmetric, 
non-negative and bounded with the boundedness constant $L_a$:
\begin{align}
\left|a_T(u,v)\right| & \le L_a \| u \|_{V, T} \| v \|_{V, T}\quad\forall\,u,v\in V_T, \label{a3}\\
a_T(v,v) & \ge 0 \quad\forall\,v\in V_T. \label{a4}
\end{align}
The discrete bilinear form $a^h\colon V^h\times V^h\to\mathbb{R}$ can be split into the summation of 
local contributions
\begin{align}
a^h(u^h, v^h) = \sum_{T \in \mathcal{T}^h} a^h_T(u^h, v^h), \label{bilinear_split}
\end{align}
where $ a^h_T(\cdot, \cdot) $ is a symmetric bilinear form on $ V_T^h$ such that
\begin{equation}\label{P_consistency}
a^h_T(v^h, p) = a_T(v^h, p) \quad \forall v^h \in V_T^h, \, p \in \mathbb{P}_k(T);
\end{equation}
and for two positive constants $ \alpha_* $ and $ \alpha^* $, independent of $ h $ and $ T $, 
\begin{equation}\label{stability_ahT}
\alpha_* a_T(v^h, v^h) \leq a^h_T(v^h, v^h) \leq \alpha^* a_T(v^h, v^h) \quad \forall v^h \in V_T^h.
\end{equation}
The discrete linear functional $ f^h \in (V^h)^* $ is uniformly bounded: for a constant $c$ independent of $h$,
\[ \| f^h \|_{(V^h)^*} = \sup_{v^h \in V^h} \frac{\langle f^h, v^h \rangle}{\| v^h \|_V}\le c. \]

\smallskip

In the literature, the property \eqref{P_consistency} is called the $k$-consistency, and 
\eqref{stability_ahT} is known as the stability.

We comment that for simplicity in writing, we are using $L_a$ for the boundedness constants 
of $a(\cdot,\cdot)$ and $a_T(\cdot,\cdot)$ for $T\in{\cal T}^h$.  

It follows from $H(h)$ that
\begin{equation}\label{boundedness_aT}
a^h_T(u^h, v^h) \le \alpha^*  L_a \| u^h \|_{V, T} \| v^h \|_{V, T} \quad \forall\, u^h, v^h \in V_T^h.
\end{equation}
This inequality is proved as follows.  First, we notice that a consequence of \eqref{stability_ahT} 
and \eqref{a4} is 
\[ a^h_T(v^h,v^h) \ge 0 \quad\forall\,v^h\in V^h_T. \]
The above property and the symmetry of $a^h_T(\cdot,\cdot)$ imply
\[ a^h_T(u^h,v^h)\le a^h_T(u^h,u^h)^{1/2} a^h_T(v^h,v^h)^{1/2}. \]
By \eqref{stability_ahT}, 
\[ a^h_T(u^h,v^h)\le \alpha^* a_T(u^h,u^h)^{1/2} a_T(v^h,v^h)^{1/2}. \]
Finally, applying \eqref{a3}, we derive \eqref{boundedness_aT}.

By combining \eqref{bilinear_split}, \eqref{stability_ahT}, and \eqref{boundedness_aT}, we obtain
\begin{align}
\alpha_* a(v^h, v^h) &\le a^h(v^h, v^h) \le \alpha^* a(v^h, v^h) \quad \forall v^h \in V^h, \label{stability_vem} \\
a^h(u^h, v^h) &\le \alpha^*  L_a \| u^h \|_{V, h} \| v^h \|_{V, h} \quad \forall u^h, v^h \in V^h, \label{boundedness_vem}
\end{align}
where $ \| \cdot \|_{V, h} = \left( \sum_{T \in \mathcal{T}^h} \| \cdot \|_{V, T}^2 \right)^{1/2} $.

\subsection{Virtual element method for variational-hemivariational inequality}

We define $ K^h := V^h \cap K $ as the approximation of the convex set $ K $. The virtual element method for solving Problem \ref{prob:VHI} is formulated as follows:

\begin{problem}\label{prob:vem}
{\it Find $u^h\in K^h$ such that}
\begin{equation} \label{hvh}
a^h(u^h,v^h-u^h) +\Phi(v^h)-\Phi(u^h)+I_\Delta(\psi^0(\gamma_\psi u^h;\gamma_\psi v^h-\gamma_\psi u^h))
\ge\langle f^h,v^h-u^h \rangle\quad\forall\, v^h\in K^h.
\end{equation}
\end{problem}

The analog of Theorem \ref{thm1h} is the next result for Problem \ref{prob:vem}. 

\begin{theorem}\label{thm:well-VEM}
Assume the conditions $H(V)$, $H(K)$, $H(a)$, $H(\Phi)$, $H(\psi)$, $H(f)$, $H(h)$, and 
$\alpha_* m_a >\alpha_\psi c_\Delta^2$. Then, Problem \ref{prob:vem} has a unique solution $u^h\in K^h$.
\end{theorem}

In the following theorem, we establish a generalized form of C\'ea's inequality, for deriving error estimates for the virtual element method \eqref{hvh} used to solve Problem \ref{prob:VHI}.
In the theorem, we assume $\{u^h\}$ is bounded independent of $h$.  The boundedness of $\{u^h\}$ is 
valid if \eqref{3.3h} holds for the VEM sets $\{K^h\}$, as in Proposition \ref{prop:bd}.  
As a simpler situation, the boundedness of $\{u^h\}$ is valid if we assume $K$ and $\{K^h\}$ contain a 
common element, say $0$, by an argument shown in \cite{HSB17}.  

\begin{theorem}\label{thm:main}
Keep the assumptions stated in Theorem \ref{thm:well-VEM}, and $m_a >\alpha_\psi c_\Delta^2$. 
Let $ u $ and $ u^h $ be the solutions of Problem \ref{prob:VHI} and Problem \ref{prob:vem}, respectively.
Assume $\{u^h\}$ is bounded independent of $h$.  Then there exist two positive constants $c_1$ and $c_2$,
depending only on $\alpha$, $M$, $\alpha_*$ and $\alpha^*$, such that for any approximation $u^I\in K^h$
of $u$ and any piecewise polynomial approximation $u^{\pi}$ of $u$ with $u^{\pi}|_T\in \mathbb{P}_k(T)$ 
for all $ T \in \mathcal{T}^h $, we have
\begin{align} 
\|u-u^h\|_{V}^2 &\leq c_1\left(\|u-u^I\|_{V}^2 + \|u-u^{\pi}\|_{V,h}^2 + \|f-f^h\|_{(V^h)^*}^2
+ \|\gamma_\psi (u-u^I)\|_{V_\psi}\right)\nonumber\\[0.2mm]
&{}\quad  + c_2 R_u(u^I,u^h),
\label{err_general}
\end{align}
where 
\[ \|f-f^h\|_{(V^h)^*}:=\sup_{v^h\in V^h}\frac{\langle f,v^h\rangle-\langle f^h,v^h\rangle}{\|v^h\|_{V}},\]
and
\[ R_u(u^I,u^h):=a(u,u^I-u^h)+\Phi(u^I)-\Phi(u^h)-I_\Delta(\psi^0(\gamma_\psi u;\gamma_\psi u^h-\gamma_\psi u^I))
-\langle f, u^I-u^h\rangle. \]
\end{theorem}
\begin{proof}
We begin by decomposing the error $ e = u - u^h $ into two parts:
\[ e = e^I + e^h, \]
where
\[ e^I := u - u^I, \quad e^h:=u^I - u^h. \]
From \eqref{stability_vem} and the assumption $H(a)$, we obtain
\[ \alpha_*m_a  \|e^h\|_{V}^2 \le \alpha_* a(e^h,e^h) \le a^h (e^h,e^h) =  a^h (u^I,e^h) - a^h (u^h,e^h). \]
Using \eqref{hvh} with $ v^h = u^I $ for an upper bound on the term $- a^h (u^h,e^h)$, we find from 
the above inequality that
\begin{equation}
\alpha_*m_a \|e^h\|_V^2\le a^h (u^I,e^h)-\langle f^h, e^h\rangle +\Phi(u^I)-\Phi(u^h)+I_\Delta(\psi^0(\gamma_\psi u^h;\gamma_\psi u^I-\gamma_\psi u^h)).
\label{6.10a}
\end{equation}
Write 
\[  a^h (u^I,e^h)=\sum_T  a^h_T(u^I,e^h) =\sum_T \left[ a^h_T(u^I-u^{\pi},e^h) +a^h_T(u^{\pi},e^h)\right].\]
By \eqref{P_consistency} and the symmetry of $a^h_T(\cdot,\cdot)$, 
\[ a^h_T(u^{\pi},e^h)=a_T(u^{\pi},e^h).\]
Hence, 
\begin{align*}
\sum_T a^h_T(u^{\pi},e^h) & = \sum_T a_T(u^{\pi},e^h)\\
& = \sum_T a_T(u^{\pi}-u,e^h) + a(u,e^h).
\end{align*}
So from \eqref{6.10a}, 
\begin{align*}
\alpha_*m_a  \|e^h\|_{V}^2 & \le \sum_T\left(a^h_T(u^I-u^{\pi},e^h)+a_T(u^{\pi}-u,e^h)\right)
+a(u,e^h)-\langle f^h,e^h\rangle\\
&\quad{} +\Phi(u^I)-\Phi(u^h)+I_\Delta(\psi^0(\gamma_\psi u^h;\gamma_\psi u^I-\gamma_\psi u^h)),
\end{align*}
which is rewritten as
\begin{equation}
\alpha_*m_a  \|e^h\|_{V}^2 \le R_1 + R_2 + R_3 + R_u(u^I,u^h),
\label{error0}
\end{equation}
where
\begin{align*}
R_1 &= \sum_{T} \left(a^h_T (u^I-u^{\pi},e^h) + a_T (u^{\pi} - u,e^h) \right),\\
R_2& = \langle f, e^h\rangle  -  \langle f^h, e^h\rangle,\\[0.1mm]
R_3& = I_\Delta(\psi^0(\gamma_\psi u^h;\gamma_\psi u^I-\gamma_\psi u^h))
+I_\Delta(\psi^0(\gamma_\psi u;  \gamma_\psi u^h - \gamma_\psi u^I)),\\[0.1mm]
R_u(u^I,u^h)&= a(u,e^h)+\Phi(u^I)-\Phi(u^h)-I_\Delta(\psi^0(\gamma_\psi u;\gamma_\psi u^h-\gamma_\psi u^I))
-\langle f,e^h\rangle.
\end{align*}

Next, we bound the first three terms on the right side of \eqref{error0}.
By \eqref{a1} and \eqref{boundedness_aT}, we get
\begin{equation}
R_1 \leq \alpha^*  L_A \|u^I-u^{\pi}\|_{V,h} \|e^h\|_{V} +  L_A \|u^{\pi}-u\|_{V,h} \|e^h\|_{V}.
\label{errorR1}
\end{equation}
In addition,
\begin{align} \label{errorR2}
R_2  \leq \|f-f^h\|_{(V^h)^*} \|e^h\|_V.
\end{align}
To bound $R_3$, we first apply Proposition \ref{subdiff}\,(ii) on the subadditivity of the generalized directional derivative,
\begin{align*}
\psi^0(\gamma_\psi u^h;\gamma_\psi u^I-\gamma_\psi u^h) 
& \le \psi^0(\gamma_\psi u^h;\gamma_\psi u^I-\gamma_\psi u)+\psi^0(\gamma_\psi u^h;\gamma_\psi u-\gamma_\psi u^h),\\
\psi^0(\gamma_\psi u;\gamma_\psi u^I-\gamma_\psi u^h) 
& \le \psi^0(\gamma_\psi u;\gamma_\psi u^I-\gamma_\psi u)+\psi^0(\gamma_\psi u;\gamma_\psi u-\gamma_\psi u^h).
\end{align*}
By \eqref{eq4} and \eqref{eq12w}, 
\[ I_\Delta(\psi^0(\gamma_\psi u^h;\gamma_\psi u-\gamma_\psi u^h)) 
+I_\Delta(\psi^0(\gamma_\psi u;\gamma_\psi u^h-\gamma_\psi u))\le \alpha_\psi c_\Delta^2 \|u-u^h\|_V^2.\]
Thus,
\[ R_3 \le \alpha_\psi c_\Delta^2 \|u-u^h\|_V^2+
I_\Delta(\psi^0(\gamma_\psi u^h;\gamma_\psi u^I-\gamma_\psi u))
+I_\Delta(\psi^0(\gamma_\psi u;\gamma_\psi u-\gamma_\psi u^I)).\]
By \eqref{eq10w}, 
\begin{align*}
\psi^0(\gamma_\psi u^h;\gamma_\psi u^I-\gamma_\psi u)
& \le c\left(1+|\gamma_\psi u^h|_{\mathbb{R}^m}\right)|\gamma_\psi u^I-\gamma_\psi u|_{\mathbb{R}^m},\\
\psi^0(\gamma_\psi u;\gamma_\psi u-\gamma_\psi u^I)
& \le c\left(1+|\gamma_\psi u|_{\mathbb{R}^m}\right)|\gamma_\psi u-\gamma_\psi u^I|_{\mathbb{R}^m}.
\end{align*}
Then,
\begin{align*}
& I_\Delta(\psi^0(\gamma_\psi u^h;\gamma_\psi u^I-\gamma_\psi u))
+I_\Delta(\psi^0(\gamma_\psi u;\gamma_\psi u-\gamma_\psi u^I))\\
&\qquad \le c\left(1+\|u^h\|_V+\|u\|_V\right)\|\gamma_\psi u-\gamma_\psi u^I\|_{V_\psi}.
\end{align*}
Since $ \| u^h \|_V $ is bounded independent of $h$, we conclude that
\begin{equation}
R_3\le \alpha_\psi c_\Delta^2 \|u- u^h\|_{V}^2 + c\,\|\gamma_\psi (u-u^I)\|_{V_\psi}.
\label{errorR3}
\end{equation}

Combining \eqref{error0}--\eqref{errorR3}, we have a constant $c>0$ such that
\begin{align*}
\|e^h\|_V^2 &\le c\left(\|u^I-u^\pi\|_{V,h}+\|u^\pi-u\|_{V,h}+\|f-f^h\|_{(V^h)^*}\right)\|e^h\|_V\\
&\quad{}+\frac{\alpha_\psi c_\Delta^2}{\alpha_*m_a}\|u-u^h\|_V^2
+c\,\|\gamma_\psi(u-u^I)\|_{V_\psi}+\frac{1}{\alpha_*m_a}\,R_u(u^I,u^h).
\end{align*}
Applying the modified Cauchy-Schwarz inequality \eqref{mCS}, for any small $\epsilon>0$, we have a constant $c$ depending on $\epsilon$ such that
\begin{align}
\left(1-\epsilon\right) \|e^h\|_V^2 & \le c\left(\|u^I-u^\pi\|_{V,h}^2+\|u^\pi-u\|_{V,h}^2
+\|f-f^h\|_{(V^h)^*}^2 + \|\gamma_\psi(u-u^I)\|_{V_\psi}\right)\nonumber\\
&\quad{} +\frac{\alpha_\psi c_\Delta^2}{\alpha_*m_a}\,\|u-u^h\|_V^2+\frac{1}{\alpha_*m_a}\,R_u(u^I,u^h). 
\label{6.19a}
\end{align}
From the triangle inequality
\begin{equation} 
\|u-u^h\|_V\le\|u-u^I\|_V+\|e^h\|_V
\label{6.19b}
\end{equation}
and the modified Cauchy-Schwarz inequality, we have 
\begin{equation} 
\|u-u^h\|_V^2\le c\,\|u-u^I\|_V^2 + \left(1+\epsilon\right) \|e^h\|_V^2.
\label{6.19c}
\end{equation}
Hence, from \eqref{6.19a},
\begin{align*}
& \left(1-\epsilon - \frac{\alpha_\psi c_\Delta^2}{\alpha_*m_a}\,(1+\epsilon) \right) \|e^h\|_V^2 \\
&\qquad \le c\left(\|u-u^I\|_V^2 + \|u^I-u^\pi\|_{V,h}^2+\|u^\pi-u\|_{V,h}^2
+\|f-f^h\|_{(V^h)^*}^2 + \|\gamma_\psi(u-u^I)\|_{V_\psi}\right)\\
&\qquad \quad{} +\frac{1}{\alpha_*m_a}\,R_u(u^I,u^h). 
\end{align*}
Since $\alpha_\psi c_\Delta^2<\alpha_*m_a$, we can choose $\epsilon>0$ small enough and deduce 
from the above inequality that
\begin{align*}
\|e^h\|_V^2 & \le c\left(\|u-u^I\|_V^2 + \|u^I-u^\pi\|_{V,h}^2+\|u^\pi-u\|_{V,h}^2
+\|f-f^h\|_{(V^h)^*}^2 + \|\gamma_\psi(u-u^I)\|_{V_\psi}\right)\\
&\quad{} +\frac{2}{\alpha_*m_a}\,R_u(u^I,u^h). 
\end{align*}
Finally, the bound \eqref{err_general} follows from an application of \eqref{6.19c}.  \hfill
\end{proof}

\section{Virtual element method for contact problems}\label{VEM:contact}

We now apply the VEM to solve the contact problems.  For this purpose, we construct the virtual element space 
$V^h \subset V$, along with the corresponding bilinear form $ a^h $ and right-hand side $ f^h $ satisfying $H(h)$.
The discussion in this section is restricted to the spatial dimension $d=2$.

Consider a family of partitions $\{ \mathcal{T}^h \}$ of the closure $\overline{\Omega}$ into elements $ T $. Let $ h_T = \mathrm{diam}(T) $ and $ h = \max\{h_T: T \in \mathcal{T}^h \} $. Define $ E^h_0 $ as the set of edges that do not lie on $ \Gamma_D $ and $ P^h_0 $ as the set of vertices not on $ \Gamma_D $.

Following \cite{beirao2013basic, beirao2013virtual, beirao2016stability}, we make the following assumption:
\begin{assumption}\label{assumption:mesh}
There exists a constant $\delta>0$ such that for each $h$ and every $T\in {\cal T}^h$,
\begin{itemize}
\item $T$ is star-shaped with respect to a ball of radius $\delta h_T$;
\item The distance between any two vertices of $T$ is at least $\delta h_T$.
\end{itemize}
\end{assumption}

\subsection{Construction of the virtual element space}

Let $T$ be a polygon with $n$ edges. For $k\geq1$, we define the local finite dimensional space $W^h_T$ on the element $T$ as
\begin{equation}\label{VhTc}
\bW^h_T:=\{\bv\in H^1(T;\real^2)\mid\nabla\cdot{\cal E}\bvarepsilon({\bv})\in\mathbb P_{k-2}(T;\real^2),
\,\bv|_{\partial T}\in C^0(\partial T), \,\bv|_e\in\mathbb P_{k}(e;\real^2)~\forall\, e\subset\partial T\}
\end{equation}
with the convention that $\mathbb P_{-1}(T)=\{0\}$.
For each $\bv\in W^h_T$, we define the following degrees of freedom:
\begin{align}
&\bullet~\mathrm{The}~\mathrm{values~of}~ \bv(\ba)~~~\forall\,\mathrm{vertex}~\ba\in T,\label{lfd_c1}\\
&\bullet~\mathrm{The~moments}~ \int_e\bq\cdot\bv\, ds~~~\forall\, \bq\in\mathbb P_{k-2}(e;\real^2)~~~\forall\,\mathrm{edge}~ e\subset\partial T, ~~~ k\geq 2,\label{lfd_c3}\\
&\bullet~\mathrm{The~moments}~ \int_T \bq\cdot\bv\, dx~~~\forall\, \bq\in\mathbb P_{k-2}(T;\real^2), ~~~k\geq 2.\label{lfd_c5}
\end{align}

For any partition $\mathcal T^h$ and $k\geq 1$, we define the global virtual element space
\begin{equation}\label{Vhc}
\bW^h:=\{\bv\in \bW \mid \bv|_T\in W^h_T\quad \forall \, T\in\mathcal T^h\},
\end{equation}
with global degrees of freedom for $\bv\in \bW^h$ given by:
\begin{align}
&\bullet~\mathrm{The}~\mathrm{values~of}~\bv(\ba)\quad\forall\,\mathrm{vertex}~\ba\in P^h_0,\label{gfd_c1}\\
&\bullet~\mathrm{The~moments}~ \int_e\bq\cdot\bv\, ds\quad\forall\, \bq\in\mathbb P_{k-2}(e;\real^2)
\quad\forall\,\mathrm{edge}~ e\in E^h_0,\, k\geq 2,\label{gfd_c3}\\
&\bullet~\mathrm{The~moments}~ \int_T \bq\cdot\bv\, dx\quad\forall\, \bq\in\mathbb P_{k-2}(T;\real^2)\quad
\forall\, \mathrm{element}~T\in \mathcal T^h,\, k\geq 2.\label{gfd_c5}
\end{align}
It is shown in \cite{beirao2013virtual} that the degrees of freedom \eqref{gfd_c1}--\eqref{gfd_c5} 
are unisolvent for $\bW^h$.

Let $ \chi_i $ represent the $ i $-th degree of freedom for $ \bW^h $, where $ i = 1, 2, \ldots, N_{\mathrm{dof}} $. 
Due to the unisolvence of the degrees of freedom for $\bW^h$, for any sufficiently smooth function $\bw$, 
there exists a unique element $ \bw^I \in W^h $ such that
$$\chi_i(\bw-\bw^I)=0,~~~i=1,2,\ldots,N_{\rm dof}.$$
By a scaling argument and the Bramble-Hilbert lemma, the following approximation property holds (\cite{beirao2013virtual}):
\begin{equation}\label{g_approximation_c}
\|\bw-\bw^I\|_{H^j(\Omega)}\leq c\,h^{l-j}|\bw|_{H^l(\Omega)},\quad j=0,1,\  2\leq l\leq k+1.
\end{equation}
Moreover, for each $T\in\mathcal T^h$ and $\bw\in H^l(T;\real^2)$, there exists
$\bw^\pi\in\mathbb P_k(T;\real^2)$ such that (\cite{BS2008,beirao2013virtual})
\begin{equation}\label{l_approximation_c}
\|\bw-\bw^\pi\|_{H^j(T)}\leq c\,h^{l-j}_T|\bw|_{H^l(T)},~~~j=0,1,\quad 1\leq l\leq k+1.
\end{equation}

\subsection{Construction of $a^h$ and $\fb^h $}

Using the approaches in \cite{beirao2013virtual, wriggers2016virtual}, we construct a symmetric and computable discrete bilinear form $ a^h $ and discrete linear form $f^h$ so that $H(h)$ is valid.

For any element $T$, denote by $ n_V^T $ the number of vertices and by $N^{\rm dof}_T$ the number
of degrees of freedom.  Also, let $a_T(\cdot,\cdot)$ be the restriction of $a(\cdot,\cdot)$ on $T$.
Following \cite{wriggers2016virtual}, we first introduce a projection operator
$\Pi_k^{T}\colon \bW^h_T\rightarrow\mathbb{P}_{k}(T;\real^2)$ defined by
\begin{align}
a_T(\Pi_k^{T}\bv^h,\bq) &= a_T(\bv^h,\bq)\quad\forall\,\bq\in \mathbb{P}_{k}(T;\real^2), \label{projection_gradient}\\
\frac{1}{n_V^T}\sum_{i=1}^{n_V^T}\Pi_k^{T}\bv^h(\bx_i)&=\frac{1}{n_V^T}\sum_{i=1}^{n_V^T}\bv^h(\bx_i),\label{7.11a}\\
\frac{1}{n_V^T}\sum_{i=1}^{n_V^T}\bx_i\times\Pi_k^{T}\bv^h(\bx_i)&=\frac{1}{n_V^T}\sum_{i=1}^{n_V^T}\bx_i\times\bv^h(\bx_i),\label{7.11b}
\end{align}
where $\bx_i$ denotes the coordinates of the vertices of $T$. Here, ``$\times$" denotes the cross product of two vectors.

We then define the local bilinear form
\begin{equation}\label{ah}
a^h_T(\bu^h,\bv^h): = a_T(\Pi_k^{T}\bu^h,\Pi_k^{T}\bv^h) + S_{T}\big((I-\Pi_k^{T})\bu^h,(I-\Pi_k^{T})\bv^h\big)
\quad\forall\, \bu^h,\bv^h\in \bW^h_T,
\end{equation}
where
\[ S_{T}(\bu^h,\bv^h) = \sum_{i=1}^{N^{\rm dof}_T} \chi_i(\bu^h)\; \chi_i(\bv^h) \]
is the stabilization term. The bilinear form 
\[ a^h(\bu^h,\bv^h)=\sum_{T\in\mathcal T^h}a^h_T(\bu^h,\bv^h)\]
ensures properties \eqref{P_consistency} and \eqref{stability_ahT}.
Other constructions of $a^h$ that meet these criteria can also be applied, such as the bilinear form proposed in \cite{artioli2017arbitrary} and used in \cite{FHH19}.

The term $(\fb_0,\bv)_{L^2(\Omega;\mathbb{R}^2)}$ in \eqref{cont9h} is not computable for
$\bv\in W^h$, and we approximate $\fb_0$ by $\fb^h_0$ constructed as follows.  
For $ k \geq 2 $, we define $ \fb_0^h $ such that
\[ \fb_{0T}^{h} :=\fb_0^h|_T = P_{k-2}^T \fb_0 \quad\forall\, T\in\mathcal T^h \]
is the $L^2(T)$-projection of $\fb_0$ onto the space of polynomials of order $k-2$ on each element $T$. 
Then we define
\[ \langle \fb_0^h,\bv^h\rangle=\sum_{T\in\mathcal T^h}\int_T\fb_{0T}^h\cdot\bv^h\,dx\quad\forall\,\bv^h\in \bW^h. \]
For $k=1$, we choose
\[ \fb_{0T}^{h}:=\fb_0^h|_T=P_{0}^T \fb_0\quad\forall\, T\in\mathcal T^h \]
to be the mean value of $\fb_0$ on $T$,  and define
\[ \langle \fb_0^h,\bv^h\rangle=\sum_{T\in\mathcal T^h}\int_T\fb_{0T}^{h}\cdot\bar{\bv^h}\,dx
\quad\forall\, \bv^h\in \bW^h,\]
where $\bar{\bv^h}$ represents the average value of $\bv^h$ over all vertices of $T$.

To approximate the right-hand side term $\langle\fb, \bv\rangle_{{\boldsymbol W}^*\times {\boldsymbol W}}$, we set
\[  \langle\fb^h, \bv^h\rangle=\langle \fb_{0}^{h} ,\bv^h\rangle+(\fb_2,\bv^h)_{L^2(\Gamma_2;\mathbb{R}^2)}
\quad \forall\,\bv\in \bW^h. \]
This setup ensures the optimal order error bound (\cite{beirao2013virtual}):
\begin{equation}\label{f_fh}
\|\fb-\fb^h \|_{(W^h)*}\leq c\,h^{k} |\fb|_{H^{k-1}(\Omega)}.
\end{equation}

\subsection{Error analysis for contact problems}\label{s4}

We apply the framework developed in Section \ref{sec:VEM} to perform error estimation for 
virtual element solutions of the three static contact problems.

\noindent{\bf VEM for Problem \ref{prob:cont1}}

The function space associated with the virtual element method is defined as:
\begin{align}\label{cont11a-1}
\bV^h = \left\{ \bv^h \in \bW^h \mid \bv^h =\bzero \, \text{ on } \, \Gamma_D \right\}.
\end{align}
The virtual element scheme for Problem \ref{prob:cont1} is formulated as follows:

\begin{problem}\label{p1h_vem}
{\it Find a displacement field $\bu^h\in \bV^h$ such that}
\begin{align}
& a^h(\bu^h,\bv^h - \bu^h)+ \int_{\Gamma_C}f_b\left(|\bv^h_\tau|-|\bu^h_\tau|\right) ds
+\int_{\Gamma_C}\psi^0_\nu(u_\nu^h;v_\nu^h-u_\nu^h)\,ds\ge\langle\fb^h, \bv^h-\bu^h\rangle\quad\forall\,\bv^h\in\bV^h.
\label{s7.2a-1}
\end{align}
\end{problem}

To apply Theorem \ref{thm:main}, we estimate the residual term
\begin{align*}
R_{\boldsymbol u}(\bu^I,\bu^h) & = \int_\Omega {\cal E}(\bvarepsilon(\bu)):\bvarepsilon(\bu^I - \bu^h)\,dx
+ \int_{\Gamma_C}f_b\left(|\bu^I_\tau|-|\bu^h_\tau|\right) ds\\
&{}\quad -\int_{\Gamma_C} \psi^0_\nu (u_\nu; u_\nu^h - u_\nu^I)\,ds - \langle\fb, \bu^I-\bu^h\rangle.
\end{align*}
Using relations \eqref{s7.2d} and \eqref{s7.2e}, similar to \eqref{s7.2f}, we derive
\begin{equation}\label{weak_rel}
\int_\Omega {\cal E}(\bvarepsilon(\bu)):\bvarepsilon(\bu^I-\bu^h)\,dx=
\langle\fb,\bu^I-\bu^h\rangle+\int_{\Gamma_C}\bsigma\bnu{\cdot}(\bu^I-\bu^h)\,ds.
\end{equation}
Thus,
\begin{align}
R_{\boldsymbol u}(\bu^I,\bu^h) & =\int_{\Gamma_C} (\bsigma \bnu)\cdot (\bu^I - \bu^h) \,ds
+ \int_{\Gamma_C}f_b\left(|\bu^I_\tau|-|\bu^h_\tau|\right) ds
-\int_{\Gamma_C} \psi^0_\nu (u_\nu; u_\nu^h - u_\nu^I)\,ds\nonumber\\
&= \int_{\Gamma_C}\sigma_\nu (u^I_\nu - u^h_\nu)\,ds+\int_{\Gamma_C} \bsigma_{\tau}\cdot(\bu^I_\tau-\bu^h_\tau)\,ds\nonumber\\
&{}\quad + \int_{\Gamma_C}f_b\left(|\bu^I_\tau|-|\bu^h_\tau|\right) ds
-\int_{\Gamma_C} \psi^0_\nu (u_\nu; u_\nu^h - u_\nu^I)\,ds.
\label{Re1}
\end{align}

To proceed further, we continue the arguments presented in Subsection \ref{subsec:cont1} to derive 
pointwise relations for the weak solution. We assume the solution regularities \eqref{s7.2b}.
Recalling \eqref{s7.2d} and \eqref{s7.2e}, we can derive from 
\eqref{cont7} that
\begin{align}
& I_{\Gamma_C}\left( \sigma_\nu (v_\nu - u_\nu )+\psi^0(u_\nu;v_\nu-u_\nu)
+ \bsigma_{\tau} \cdot (\bv_\tau - \bu_\tau)+f_b(|\bv_\tau|-|\bu_\tau|)\right)\ge 0\quad\forall\,\bv\in \bV.
\label{Re2}
\end{align}
By the independence of the normal component and the tangential component of an arbitrary vector 
field $\bv\in \bV$ and the densities of $\{v_\nu\mid \bv\in\bV\}$ in $L^2(\Gamma_C)$ and of
$\{\bv_\tau\mid \bv\in\bV\}$ in $L^2(\Gamma_C)^2$, we conclude from \eqref{Re2} that a.e.\ on $\Gamma_C$, 
\begin{align}
& \sigma_\nu (z - u_\nu )+\psi^0(u_\nu;z-u_\nu)\ge 0\quad\forall\,z\in L^2(\Gamma_C), \label{Re3}\\
& \bsigma_{\tau} \cdot (\bz - \bu_\tau)+f_b(|\bz|-|\bu_\tau|)\ge 0\quad\forall\,\bz\in L^2(\Gamma_C)^2. \label{Re4}
\end{align}
Taking $\bz=\bzero$ and $2\,\bu_\tau$ in \eqref{Re4}, we see that \eqref{Re4} is equivalent to 
\begin{equation}
\bsigma_{\tau}\cdot\bu_\tau+f_b |\bu_\tau|=0,\quad \bsigma_{\tau} \cdot\bz+f_b |\bz|\ge 0
\quad\forall\,\bz\in L^2(\Gamma_C)^2. \label{Re5}
\end{equation}
Then,
\begin{align}
&\int_{\Gamma_C} \bsigma_{\tau} \cdot (\bu^I_\tau - \bu^h_\tau) \ ds + \int_{\Gamma_C}f_b\left(|\bu^I_\tau|-|\bu^h_\tau|\right) ds\nonumber\\
& \qquad =\int_{\Gamma_C} \bsigma_{\tau} \cdot (\bu^I_\tau - \bu_\tau) \ ds
+ \int_{\Gamma_C} \bsigma_{\tau} \cdot (\bu_\tau - \bu^h_\tau) \ ds + \int_{\Gamma_C}f_b\left(|\bu^I_\tau|-|\bu^h_\tau|\right) ds\nonumber\\
& \qquad\le\int_{\Gamma_C}f_b |\bu^I_\tau - \bu_\tau|\,ds + \int_{\Gamma_C}f_b\left(|\bu^h_\tau|-|\bu_\tau|\right)ds
+ \int_{\Gamma_C}f_b\left(|\bu^I_\tau|-|\bu^h_\tau|\right) ds\nonumber\\
& \qquad \le 2\int_{\Gamma_C}f_b|\bu^I_\tau-\bu_\tau|\,ds \le 2\|f_b\|_{L^2(\Gamma_C)}\|\bu^I_\tau-\bu_\tau\|_{L^2(\Gamma_C)^2}.\label{fri_esti}
\end{align}
Furthermore, we derive from \eqref{Re3} that 
\begin{equation}
\sigma_\nu z +\psi^0(u_\nu;z)\ge 0\quad\forall\,z\in \mathbb{R},\ {\rm a.e.\ on}\ \Gamma_C.\label{Re6}
\end{equation}
Hence, 
\begin{align*}
\int_{\Gamma_C} -\sigma_\nu (u^h_\nu - u^I_\nu ) \, ds
-\int_{\Gamma_C} \psi^0_\nu (u_\nu; u_\nu^h - u_\nu^I)\,ds
&\le \int_{\Gamma_C}\psi^0_\nu(u_\nu;u_\nu^h - u_\nu^I)\,ds-\int_{\Gamma_C}\psi^0_\nu(u_\nu;u_\nu^h-u_\nu^I)\,ds\\
& =0.
\end{align*}
Thus, applying Theorem \ref{thm:main}, we obtain
\begin{equation}\label{cea1}
\|\bu-\bu^h\|_{\boldsymbol V} \le c \left(\|\bu-\bu^I\|_{\boldsymbol V} + \|\bu-\bu^{\pi}\|_{{\boldsymbol V},h} 
+ \|\fb-\fb^h\|_{({\boldsymbol V}^h)^*}+  \| \bu_\tau - \bu^I_\tau\|_{L^2(\Gamma_C)^2}^{1/2}\right).
\end{equation}
Let $k=1$ and assume solution regularities \eqref{s7.2b} and \eqref{s7.2i}.  Recall the approximation properties \eqref{g_approximation_c}, \eqref{l_approximation_c}, and \eqref{f_fh}.  In addition, we have the analog of \eqref{5.20a} in the setting of VEM:
\[ \|\bu-\bu^I\|_{L^2(\Gamma_C)^2} \le c\,h^2.\]
Thus, we conclude that for $k=1$, the optimal order error bound is valid under solution regularities \eqref{s7.2b} and \eqref{s7.2i}
\begin{equation}
\|\bu-\bu^h\|_{\boldsymbol V}\le c\,h.
\label{cont13}
\end{equation}

\smallskip

\noindent{\bf VEM for Problem \ref{prob:cont2}}\label{subsec2}

The function space associated with the virtual element method is:
\begin{align}\label{cont11a}
& \bV^h_1=\left\{\bv^h\in \bW^h \mid \bv^h=\bzero\ {\rm on\ }\Gamma_D,\; v^h_\nu=0\ {\rm on\ }\Gamma_C\  \right\}.
\end{align}

The virtual element scheme for Problem \ref{prob:cont2} is formulated as follows:

\begin{problem}\label{p2h_vem}
{\it Find a displacement field $\bu^h\in \bV^h_1$ such that}
\begin{align}
& a^h(\bu^h,\bv^h - \bu^h)
+\int_{\Gamma_C} \psi^0_\tau (\bu_\tau^h; \bv_\tau^h-\bu_\tau^h)\,ds\ge \langle\fb^h, \bv^h-\bu^h\rangle \quad\forall\,\bv^h\in \bV^h_1.
\label{s7.2a-2}
\end{align}
\end{problem}

Following a similar approach as for Problem \ref{p1h_vem}, we can show that $R_{\boldsymbol u}(\bu^I,\bu^h)\le 0 $.
Consequently, the optimal order error bound for $ k = 1 $ is
\begin{equation*}
\|\bu-\bu^h\|_{\boldsymbol V}\le c\,h
\end{equation*}
under the regularity assumptions \eqref{s7.2b}.

\smallskip

\noindent{\bf VEM for Problem \ref{prob:cont3}}\label{subsec3}

To approximate the admissible set $\bU$, we define:
\begin{equation}\label{cont11U}
\bU^h=\left\{\bv^h\in \bV^h \mid v^h_\nu\le g\ {\rm at\ node\ points\ on\ }\overline{\Gamma_C}\right\}.
\end{equation}
Assuming that $ g $ is a concave function, we have $\bU^h \subset \bU$. The following numerical method is proposed for Problem \ref{prob:cont3}.

\begin{problem}\label{p3h_vem}
{\it Find a displacement field $\bu^h\in \bU^h$ such that}
\begin{equation}
a^h(\bu^h,\bv^h-\bu^h)+ \int_{\Gamma_C}f_b\left(|\bv^h_\tau|-|\bu^h_\tau|\right) ds
+\int_{\Gamma_C}\psi_\tau^0 (u_\nu^h; v_\nu^h-u_\nu^h)\,ds\ge\langle \fb^h,\bv^h-\bu^h\rangle\quad\forall\,\bv\in \bU^h.
\label{cont9_33}
\end{equation}
\end{problem}

We apply Theorem \ref{thm:main} to derive an error estimate.  The key step is to bound
the residual term
\begin{align*}
R_{\boldsymbol u}(\bu^I,\bu^h) &= \int_\Omega {\cal E}(\bvarepsilon(\bu)):\bvarepsilon(\bu^I - \bu^h)\,dx
+ \int_{\Gamma_C}f_b\left(|\bu^I_\tau|-|\bu^h_\tau|\right) ds\\
&{}\quad -\int_{\Gamma_C} \psi^0_\nu (u_\nu; u_\nu^h - u_\nu^I)\,ds - \langle\fb, \bu^I-\bu^h\rangle. 
\end{align*}
By a similar argument, we can derive the relation \eqref{weak_rel} as well, so
\begin{align*}
R_{\boldsymbol u}(\bu^I,\bu^h)&=\int_{\Gamma_C} (\bsigma \nu)\cdot (\bu^I - \bu^h) \, ds
+ \int_{\Gamma_C}f_b\left(|\bu^I_\tau|-|\bu^h_\tau|\right) ds
-\int_{\Gamma_C} \psi^0_\nu (u_\nu; u_\nu^h - u_\nu^I)\,ds\\
& =\int_{\Gamma_C} \sigma_\nu (u^I_\nu - u^h_\nu)\,ds+\int_{\Gamma_C}\bsigma_{\tau}\cdot(\bu^I_\tau-\bu^h_\tau)\,ds\\
&{}\quad + \int_{\Gamma_C}f_b\left(|\bu^I_\tau|-|\bu^h_\tau|\right) ds
-\int_{\Gamma_C} \psi^0_\nu (u_\nu; u_\nu^h - u_\nu^I)\,ds.
\end{align*}
Using an argument similar to \eqref{fri_esti}, we can deduce that
\begin{align*}
&\int_{\Gamma_C}\bsigma_{\tau}\cdot(\bu^I_\tau-\bu^h_\tau)\,ds+\int_{\Gamma_C}f_b\left(|\bu^I_\tau|-|\bu^h_\tau|\right) ds
\leq 2 \|f_b\|_{L^2(\Gamma_C)}  \|\bu^I_\tau - \bu_\tau\|_{L^2(\Gamma_C)^2}.
\end{align*}
Furthermore, we consider
\begin{align*}
\sigma_\nu(u^I_\nu-u^h_\nu) & = (\sigma_\nu+\xi_\nu) (u^I_\nu-u^h_\nu) -  \xi_\nu (u^I_\nu - u^h_\nu) \\
&=(\sigma_\nu + \xi_\nu) (u^I_\nu - u_\nu) + (\sigma_\nu + \xi_\nu) (u_\nu - g) + (\sigma_\nu + \xi_\nu) (g - u^h_\nu)-  \xi_\nu (u^I_\nu - u^h_\nu)\\
& \le (\sigma_\nu + \xi_\nu) (u^I_\nu - u_\nu) + \xi_\nu (u^h_\nu - u^I_\nu).
\end{align*}
Here, we use the fact that $(\sigma_\nu+\xi_\nu)(u_\nu-g)=0$ and $u_\nu^h\le g$, given that $\bu^h\in\bU^h\subset\bU$.
Since $\xi_\nu\in\partial \psi_\nu(u_\nu)$, we get
\begin{align*}
\int_{\Gamma_C}\sigma_\nu(u^I_\nu- u^h_\nu)\,ds-\int_{\Gamma_C}\tilde{\psi}^0_\nu(u_\nu;u_\nu^h-u_\nu^I)\,ds
&\le \int_{\Gamma_C} (\sigma_\nu+\xi_\nu)(u^I_\nu-u_\nu)\,ds+\int_{\Gamma_C}\xi_\nu(u^h_\nu-u^I_\nu)\,ds\\
&\quad{} -\int_{\Gamma_C} \psi^0_\nu (u_\nu; u_\nu^h - u_\nu^I)\,ds\\
&\le c \,\|u_\nu-u^I_\nu\|_{L^2(\Gamma_C)}.
\end{align*}

Finally, the optimal order error bound for $ k = 1 $ is
\begin{equation*}
\|\bu-\bu^h\|_{\boldsymbol V}\le c\,h,
\label{cont39}
\end{equation*}
under the regularity assumptions \eqref{s7.2b} and \eqref{s7.2i}.

\begin{remark}
In the above analysis, we assumed $ g $ to be a concave function. However, this assumption can be removed by applying the argument in \cite{FHH21a}. For simplicity, we retain this assumption here.
\end{remark}

\section{Numerical examples}\label{sec:ex}

In this section, we report numerical simulation results on sample contact problems, by applying both
the finite element method and the virtual element method.  In all the examples, we let  $\Omega$  
be the unit square: $\Omega = (0,1)\times (0,1)\subset\mathbb{R}^2$, and split the boundary into three parts:
\[ \Gamma_D=[0,1]\times\{1\},\quad \Gamma_N = (\{0\}\times(0,1))\cup(\{1\}\times(0,1)),
\quad \Gamma_C = [0,1]\times\{0\}. \]
The domain $\Omega$ is the initial configuration of an elastic body.  We adopt the linear elasticity 
constitutive law
\begin{equation}
\bsigma = {\cal E}\bvarepsilon(\bu)\quad \text{ in }\Omega,\label{M1_1}
\end{equation}
where
\[ ({\cal E} \btau)_{ij} = \frac{E \kappa}{(1+\kappa)(1-2\kappa)}(\tau_{11}+\tau_{22})\delta_{ij}
	+\frac{E}{1+\kappa}\tau_{ij}, \quad 1\le i,j\le 2,\ \forall\,\btau \in \mathbb{S}^2. \]
A volume force of density $\fb_0$ is applied to the elastic body and the equilibrium equation is
\begin{equation}
\text{Div}\,\bsigma +\fb_0 = \bzero \quad\text{ in }\Omega.  \label{M1_2}
\end{equation}
The $\Gamma_D$ part of the boundary is fixed, 
\begin{equation}
\bu = \bzero \quad \text{ on }\Gamma_D,\label{M1_3}
\end{equation}
and the $\Gamma_N$ part of the boundary is subject to the action of a traction force of the density $\fb_2$:
\begin{equation}
\bsigma \bnu = \fb_2  \quad \text{ on }\Gamma_N.  \label{M1_4}
\end{equation}
Different boundary conditions will be considered on the contact boundary $\Gamma_C$.
The physical setting of the problem is as depicted in Figure \ref{fig_config}.

\begin{figure}[htb]\vspace*{-1em}
	\begin{center}
		\includegraphics[width=8cm,height=6.5cm]{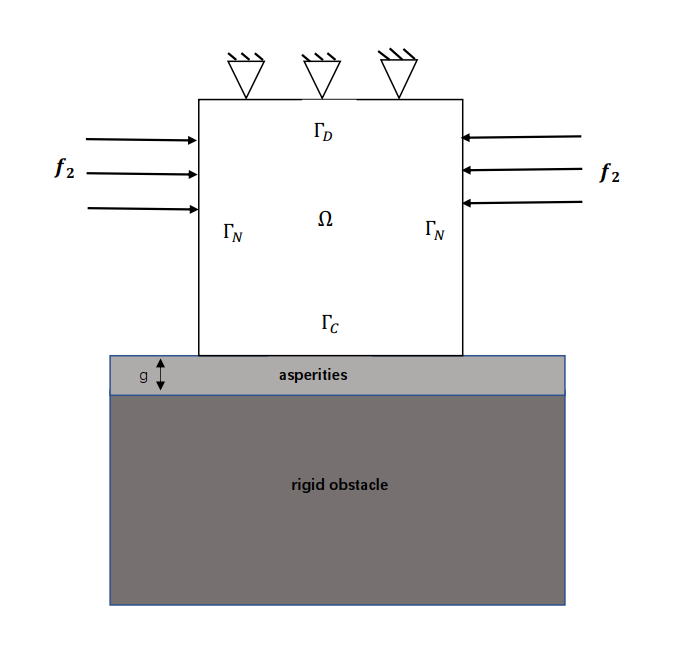}
	\end{center}\vspace*{-3em}
	\caption{Physical setting}\label{fig_config}
\end{figure}

In the numerical experiments, uniform triangulations of the domain $\Omega$ are used for the linear 
triangular finite elements (FEM). The uniform square partitions of the domain $\Omega$ are used for the 
lowest-order (i.e., $k=1$) virtual element method (VEM). The boundary of the spatial domain is divided
into $1/h$ equal parts, and $h$ is used as the discretization parameter. In order to illustrate that 
the VEM can be applied to polygonal meshes, we present the deformed meshes on the Voronoi meshes, 
which are generated by the MATLAB toolbox - PolyMesher introduced in \cite{TPPM12}. The corresponding 
deformed meshes are presented in Figure \ref{fig12}, Figure \ref{fig22} and Figure \ref{fig32}.
In the following numerical examples, we choose $f_b=0$ in the friction conditions, i.e., we consider 
the frictionless contact.

The relative errors of the numerical solutions in the $H^1$-norm, i.e.,
\[ \frac{\|\bu_{\text{ref}}-\bu^h\|_{\boldsymbol V}} {\|\bu_{\text{ref}}\|_{\boldsymbol V}} \]
will be used to compute the numerical convergence orders of the numerical solutions for the 
linear FEM and the lowest order VEM on the square meshes.
For both FEM and VEM, we take the numerical solution with $h = 1/512$ 
as the ``reference'' solution in computing the errors of numerical solutions on coarse meshes.

\begin{example} \label{nexample1}
{\rm In this example, we consider a bilateral contact problem with friction.  Let the contact conditions on $\Gamma_3$ be
\[ u_\nu = 0,\qquad -\bsigma_\tau \in \partial \psi_\tau(\bu_\tau), \]
where
\[  \psi(\bz) = \int_0^{|{\boldsymbol z}|}\mu(t)\, dt,\quad  \mu(r)  = (a-b)e^{-\beta r} +b.  \]
Note that the contact condition $-\bsigma_\tau \in \partial \psi(\bu_\tau)$ is equivalent to
\[ |\bsigma_\tau| \le \mu(0)\ \text{ if } \bu_\tau = \bzero,
\qquad-\bsigma_\tau = \mu(|\bu_\tau|)\,\frac{\bu_\tau}{|\bu_\tau|}\ \text{ if } \bu_\tau \neq \bzero. \]

The parameters are given as follows:
\begin{align*}
&E = 2000\,kg/cm^2,\quad \kappa = 0.3,\\
&a=3\times 10^{-3}, \quad b= 2.5\times 10^{-3}, \quad \beta = 2\times 10^{3},\\
&\fb_0 = \left(0,-0.05\right)\,kg/cm^2,\\
&\fb_2 =
\left\{
\begin{array}{cc}
\left(800,0\right)\,kg/cm& \text{ on } \{0\}\times [0.5,1),\\
\left(-800,0\right)\,kg/cm& \text{ on } \{1\}\times [0.5,1).\\
\end{array}
\right.
\end{align*}
}
\end{example}
We illustrate the numerical performance of both the virtual element method and the linear finite element method. 
In the VEM, we present the numerical solution on square mesh for different values of mesh numbers $N$ in 
Figure \ref{fig11}. In Figure \ref{fig12}, we present the initial and deformed Voronoi meshes 
corresponding to $N=8000$ for the VEM. Numerical solutions obtained by linear FEM on uniform triangulation 
and lowest order VEM on the square grid along the tangential direction on the boundary $[0,1] \times \{0\}$
are shown in Figure \ref{fig14}.  \hfill$\Box$

\begin{figure}[htb]
	\setlength{\abovecaptionskip}{0.cm}
	\setlength{\belowcaptionskip}{-0.cm}
	\centering
	\begin{minipage}[t]{0.5\linewidth}
		\centering
		\includegraphics[height=3.5cm,width=5cm]{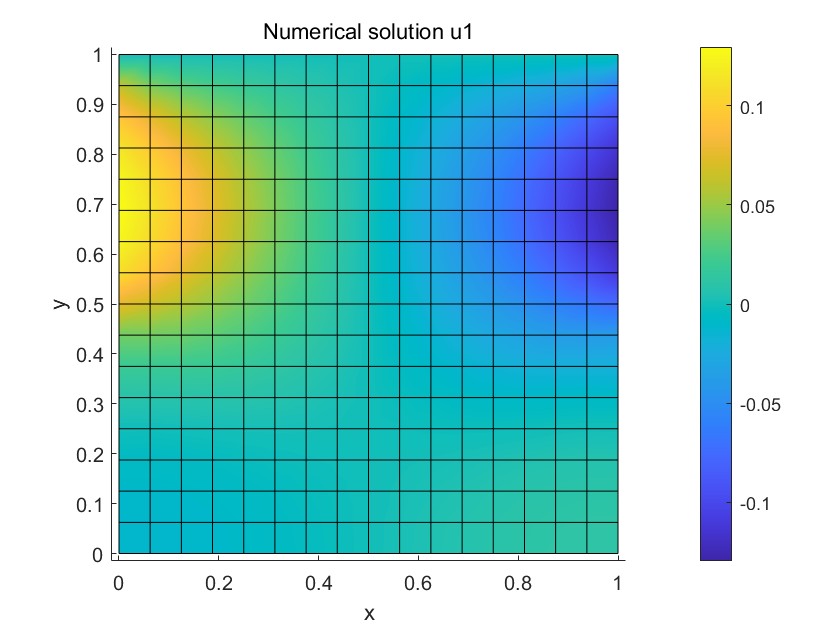}
	\end{minipage}%
	\begin{minipage}[t]{0.6\linewidth}
		\centering
		\includegraphics[height=3.5cm,width=5cm]{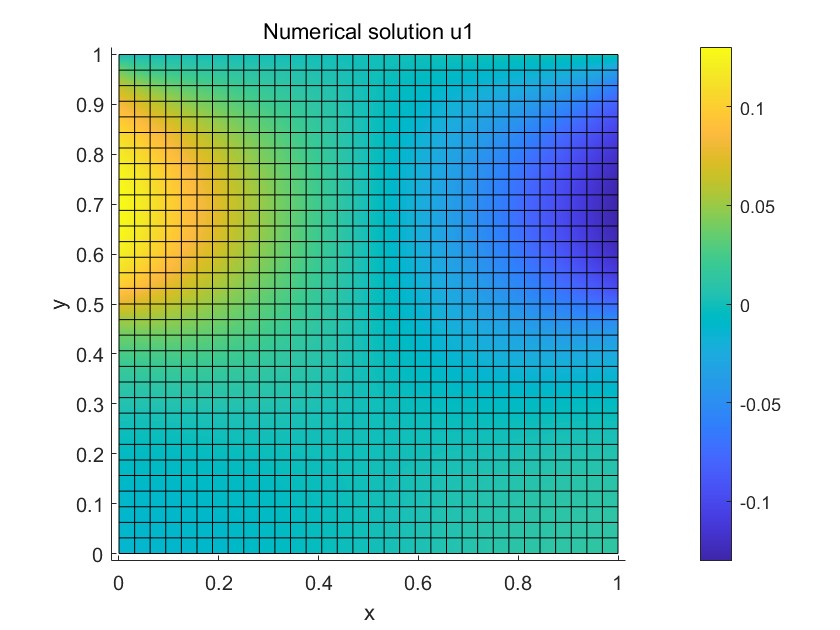}
	\end{minipage}
	\begin{minipage}[t]{0.5\linewidth}
		\centering
		\includegraphics[height=3.5cm,width=5cm]{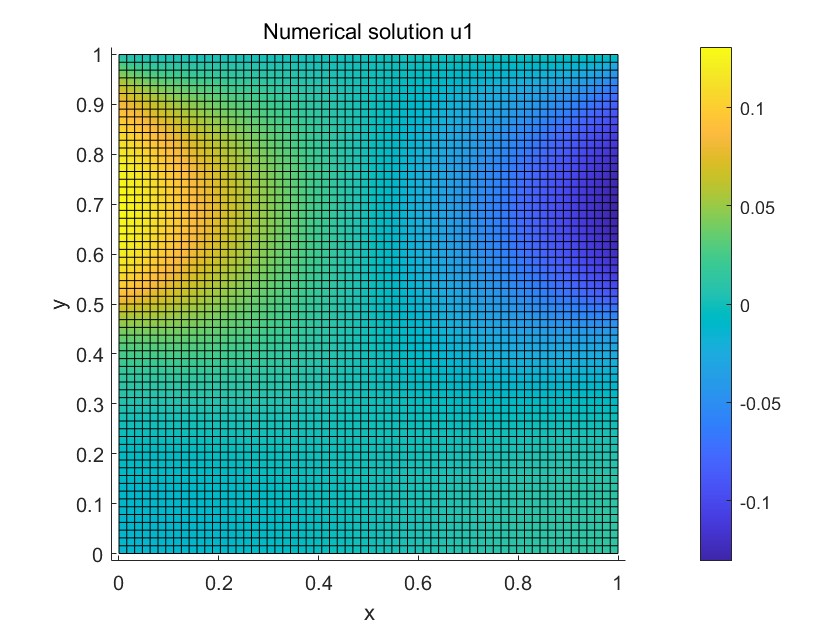}
	\end{minipage}%
	\begin{minipage}[t]{0.6\linewidth}
		\centering
		\includegraphics[height=3.5cm,width=5cm]{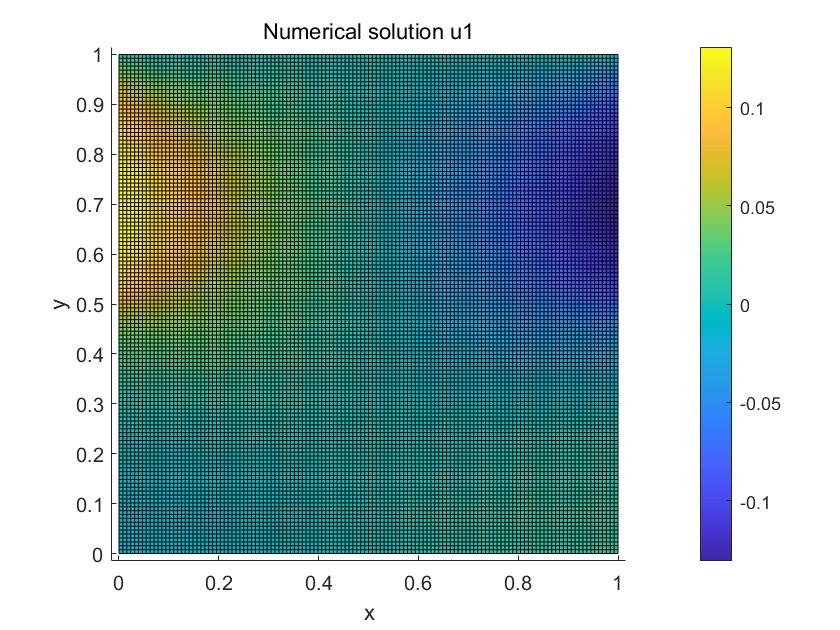}
	\end{minipage}
	\caption{Example \ref{nexample1}: Numerical solutions with $N$ elements: $N=256$ (upper left),
		$N=1024$ (upper right), $N=4096$ (bottom left) and $N=16384$ (bottom right).}
	\label{fig11}
\end{figure}

\begin{figure}[!htb]
	\centering
	\subfigure[]{\includegraphics[scale=0.2]{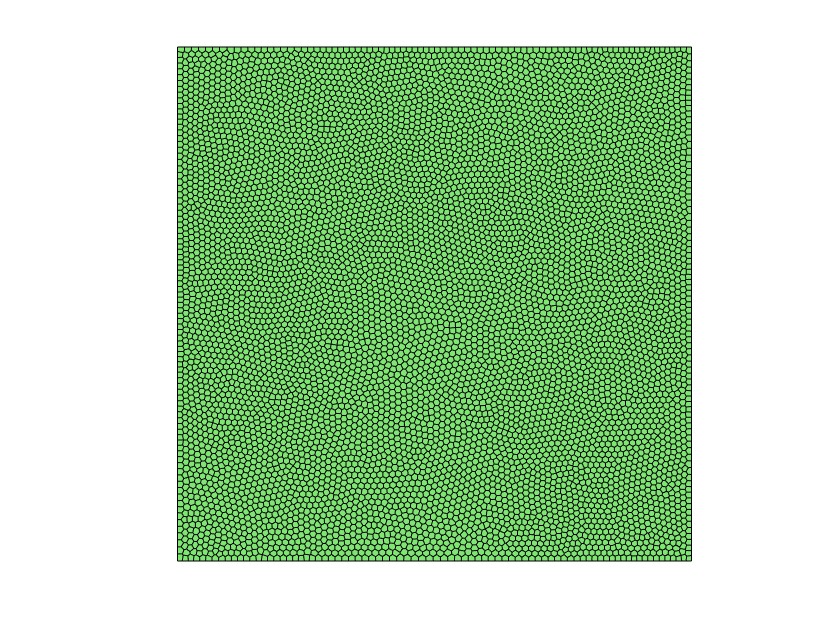}}
	\subfigure[]{\includegraphics[scale=0.2]{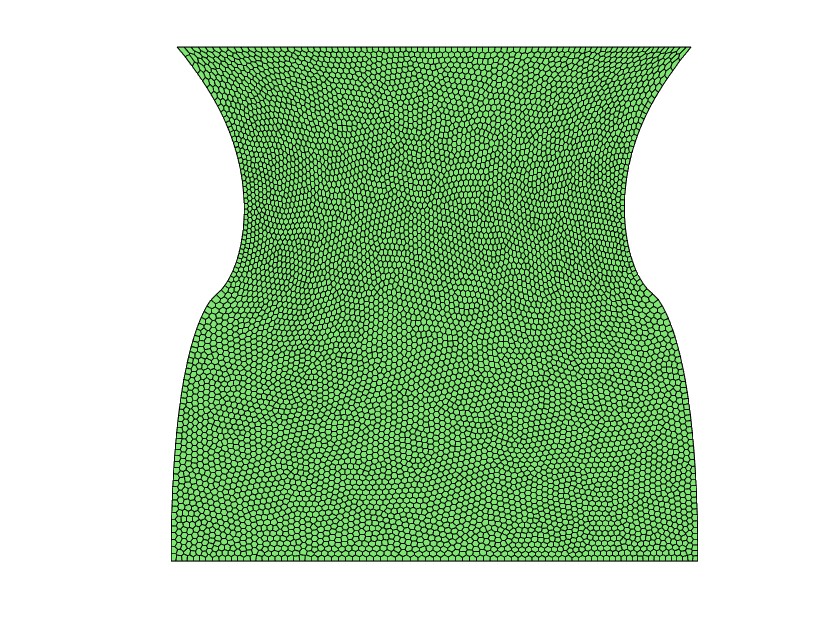}}\\
\caption{Example \ref{nexample1}: (a) Initial mesh with $N=8000$; (b) deformed meshes with $N=8000$}\label{fig12}
\end{figure}

\vskip 0.2cm

\begin{figure}[!htb]
	\centering
	\subfigure[]{\includegraphics[scale=0.2]{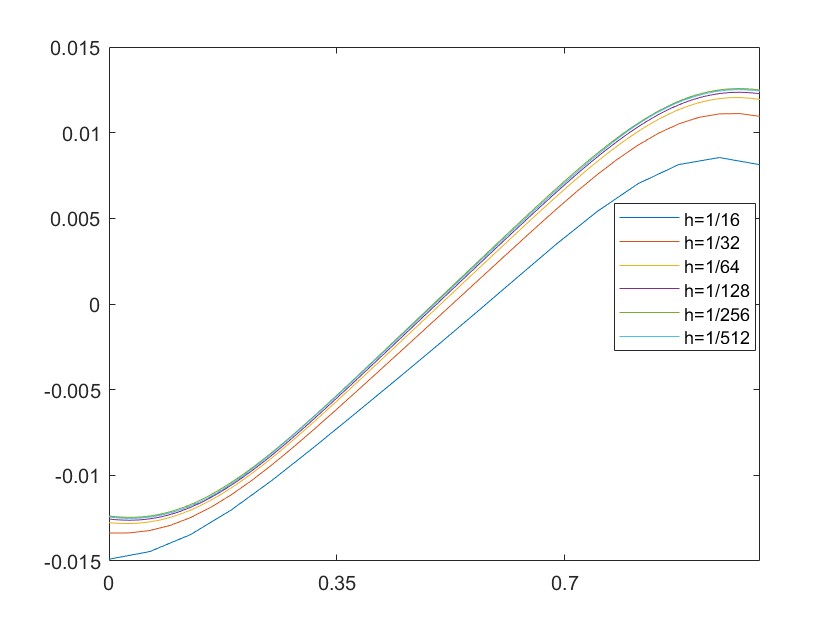}}
	\subfigure[]{\includegraphics[scale=0.2]{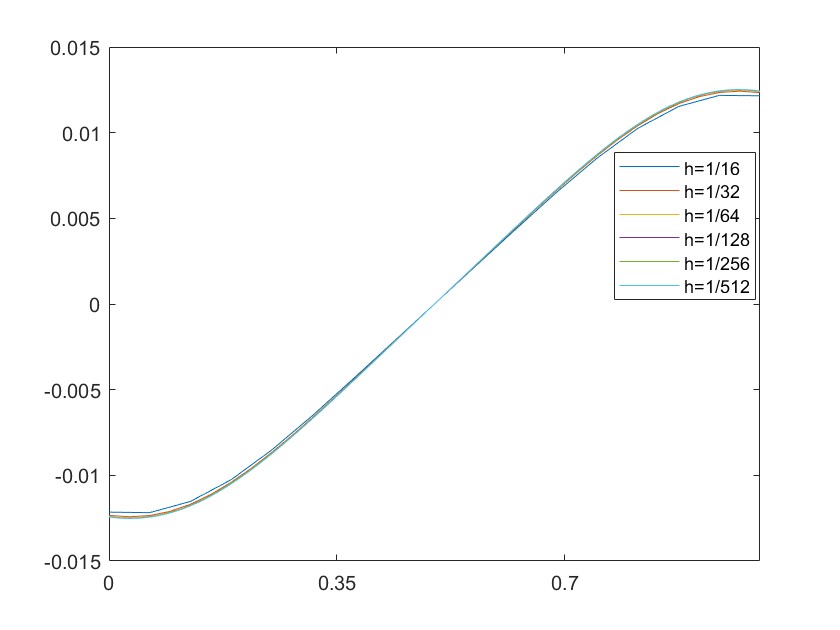}}\\
	\caption{Example \ref{nexample1}: Tangential displacement on $\Gamma_3$ for (a) FEM solution;
	(b) VEM solution on square mesh}\label{fig14}
\end{figure}

\begin{table}[htp]
	\begin{center}
	\caption{Example \ref{nexample1}: Relative errors of the displacements for the linear FEM}\label{tab1}
		\begin{tabular}{|c|c|c|c|c|c|} \hline
			$h$ &1/8 &1/16 &1/32 &1/64 &1/128  \\ \hline
			error &$20.51\%$&$11.47\%$&$6.53\%$&$3.7\%$&$1.96\%$\\ \hline
			order & -& 0.8385 &0.8127 &0.8196 &0.9167\\ \hline
		\end{tabular}
	\end{center}
\end{table}
\begin{table}[htp]
	\begin{center}
		\caption{Example \ref{nexample1}: Relative errors of the displacements on the square mesh for 
		the VEM}\label{tab1-2}
		\begin{tabular}{|c|c|c|c|c|c|} \hline
		$h$ &1/8 &1/16 &1/32 &1/64 &1/128  \\ \hline
			error &$10.06\%$&$5.84\%$&$3.39\%$&$1.94\%$&$1.07\%$\\ \hline
			order & -& 0.7846 &0.7847 &0.8052 &0.8584  \\ \hline
		\end{tabular}
	\end{center}
\end{table}

\begin{figure}[!htb]
	\centering
	\subfigure[]{\includegraphics[scale=0.2]{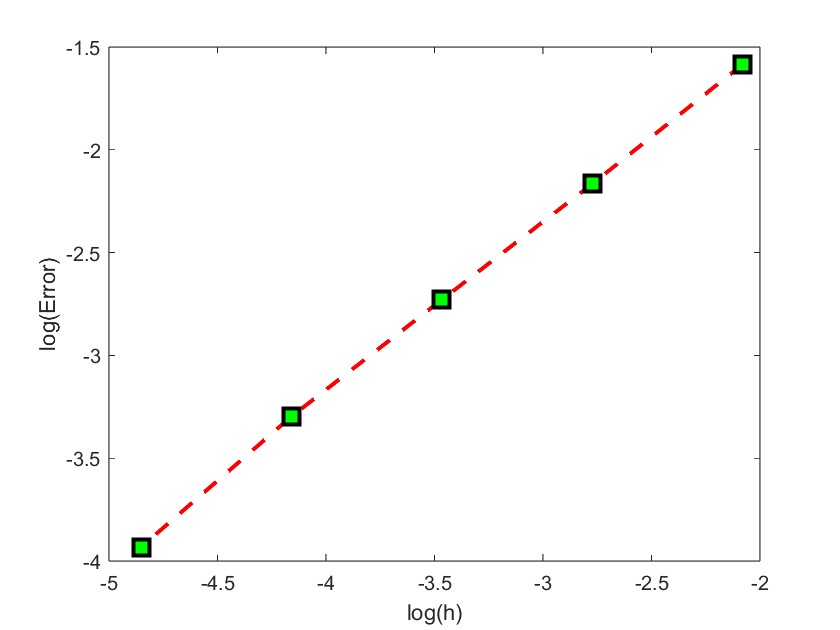}}
	\subfigure[]{\includegraphics[scale=0.2]{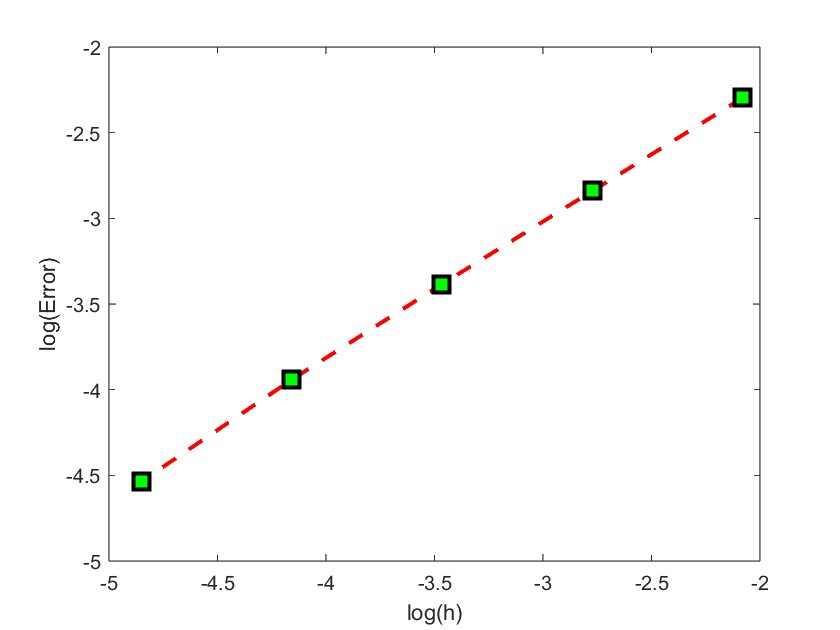}}\\
	\caption{Example \ref{nexample1}: Numerical convergence orders for (a) FEM; (b) VEM on the square mesh}\label{fig15}
\end{figure}

\begin{example} \label{nexample2}
{\rm In this example, we consider a frictionless normal compliance contact problem. On $\Gamma_C$, let
\begin{align*}
&-\sigma_\nu =
\begin{cases}
0\quad \text{if } u_\nu < 0,\\
[0,2]\quad \text{if } u_\nu =0,\\
2\quad \text{if } u_\nu \in (0,0.04],\\
4-50u_{\nu}\quad \text{if } u_\nu \in  (0.04,0.06],\\
20u_{\nu}-0.2\quad \text{if } u_\nu > 0.06,
\end{cases}\\
&\bsigma_\tau = \bzero.
\end{align*}
	The parameters are given as follows:
	\begin{align*}
		&E = 2000\,kg/cm^2,\quad \kappa = 0.3,\\
		&\fb_0 = \left(0,-0.05\right)\,kg/cm^2,\\
		&\fb_2 =
		\left\{
		\begin{array}{cc}
			\left(800,0\right)\,kg/cm& \text{ on } \{0\}\times [0.5,1),\\
			\left(-800,0\right)\,kg/cm& \text{ on } \{1\}\times [0.5,1).\\
		\end{array}
		\right.
	\end{align*}
}
\end{example}

In the VEM,we present the numerical solution on square mesh for different values of mesh numbers $N$ in 
Figure \ref{fig21}. In Figure \ref{fig22},  we present the initial and deformed meshes on voronoi meshes
corresponding to $N=8000$ for the VEM. The numerical solution obtained by linear FEM and lowest order VEM 
on the square grid  along the normal direction on the boundary $[0,1] \times \{0\}$ is shown in Figure \ref{fig24}. 
The relative errors and numerical convergence orders are reported in Table \ref{tab2}, Table \ref{tab2-2} 
and Figure \ref{fig2d}.  \hfill$\Box$

\begin{figure}[htb]
	\setlength{\abovecaptionskip}{0.cm}
	\setlength{\belowcaptionskip}{-0.cm}
	\centering
	\begin{minipage}[t]{0.5\linewidth}
		\centering
		\includegraphics[height=3.5cm,width=5cm]{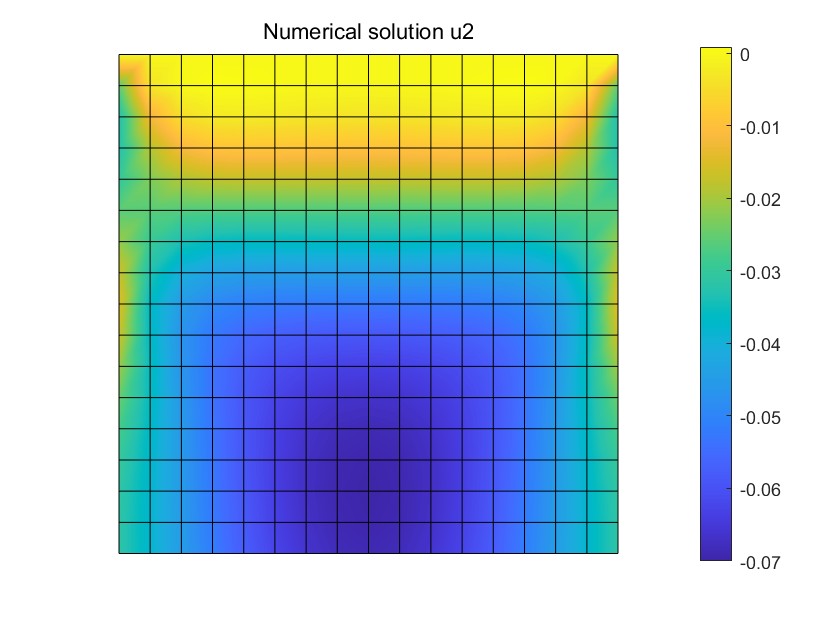}
	\end{minipage}%
	\begin{minipage}[t]{0.6\linewidth}
		\centering
		\includegraphics[height=3.5cm,width=5cm]{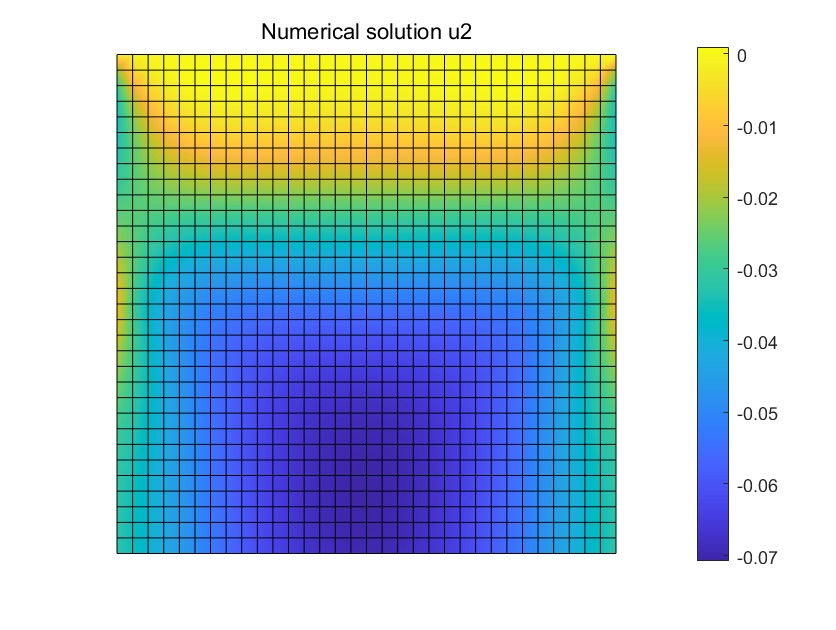}
	\end{minipage}
	\begin{minipage}[t]{0.5\linewidth}
		\centering
		\includegraphics[height=3.5cm,width=5cm]{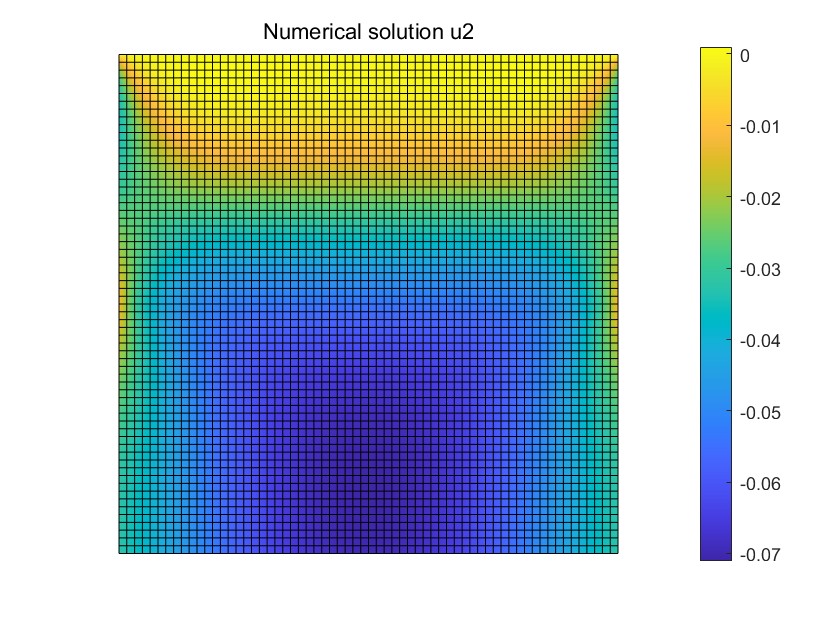}
	\end{minipage}%
	\begin{minipage}[t]{0.6\linewidth}
		\centering
		\includegraphics[height=3.5cm,width=5cm]{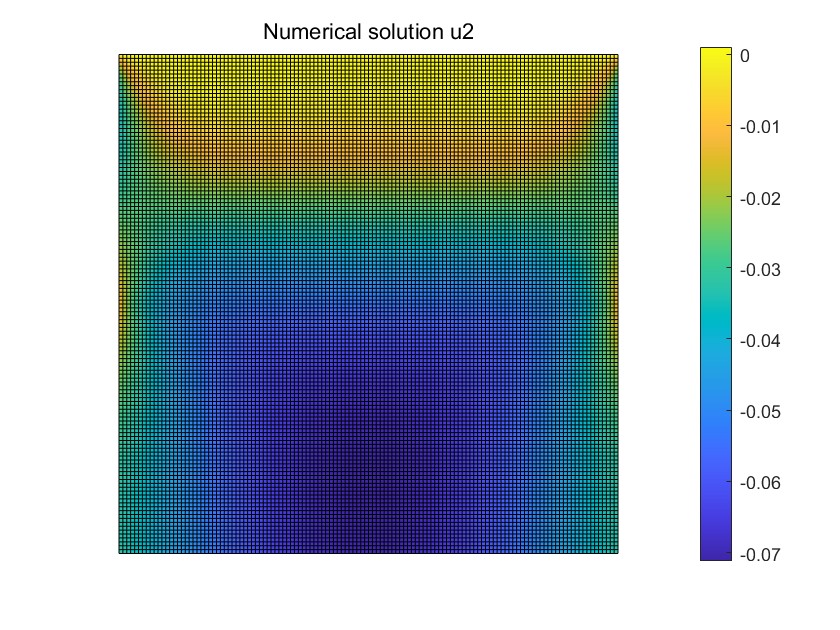}
	\end{minipage}
	\caption{Example \ref{nexample2}: Numerical solutions with $N$ elements: $N=300$ (upper left),
		$N=1000$ (upper right), $N=4000$ (bottom left) and $N=8000$ (bottom right).}
	\label{fig21}
\end{figure}

\begin{figure}[!htb]
	\centering
	\subfigure[]{\includegraphics[scale=0.2]{Ex1_initial_mesh.jpg}}
	\subfigure[]{\includegraphics[scale=0.2]{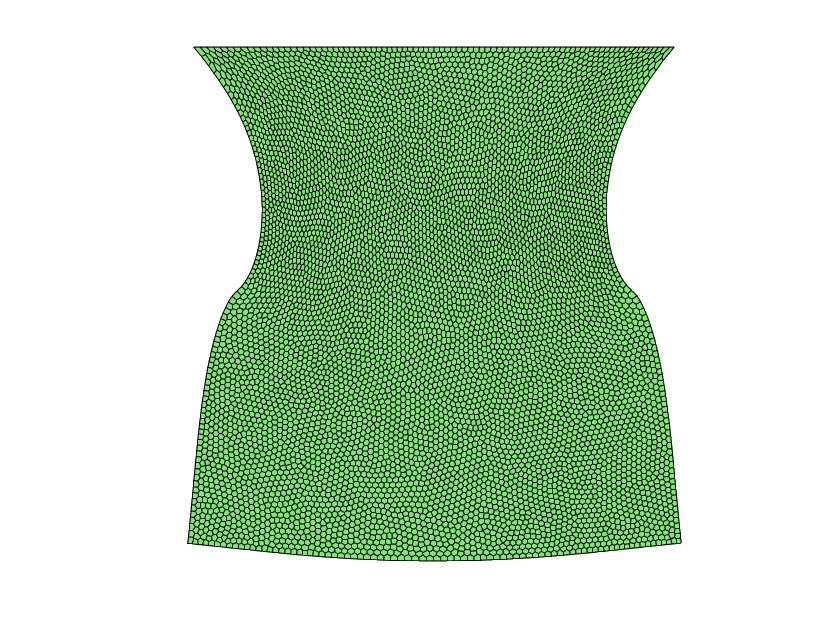}}\\
\caption{Example \ref{nexample2}: (a) Initial mesh with $N=8000$; (b) deformed meshes with $N=8000$}\label{fig22}
\end{figure}

\vskip 0.2cm
\begin{figure}[!htb]
	\centering
	\subfigure[]{\includegraphics[scale=0.2]{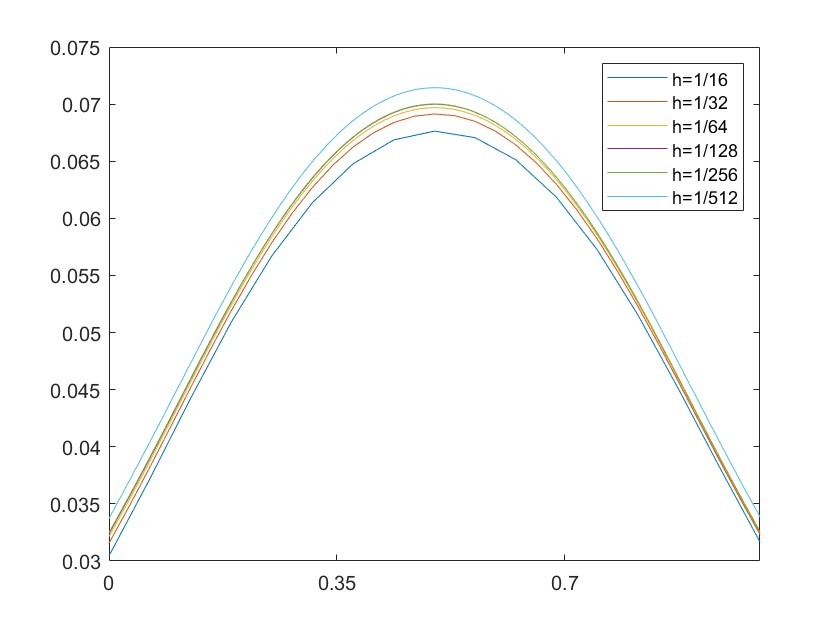}}
	\subfigure[]{\includegraphics[scale=0.2]{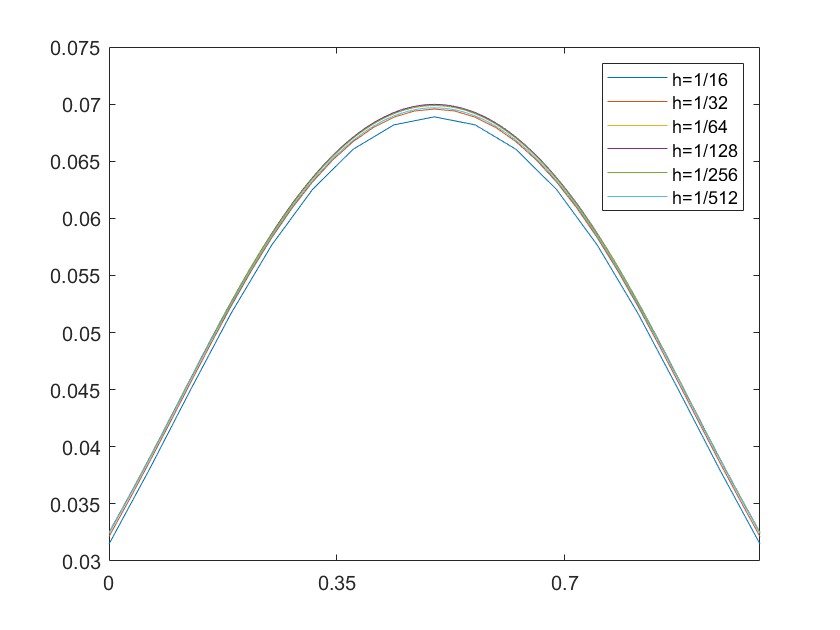}}\\
	\caption{Example \ref{nexample2}: Normal displacement on $\Gamma_3$ for (a) FEM; (b) VEM on square mesh}\label{fig24}
\end{figure}

\begin{table}[htp]
	\begin{center}
		\caption{Example \ref{nexample2}: Relative errors of the displacements for FEM}\label{tab2}
		\begin{tabular}{|c|c|c|c|c|c|c|} \hline
			$h$ &1/8 &1/16 &1/32 &1/64 &1/128\\ \hline
			error &$20.54\%$&$11.62\%$&$6.68\%$&$3.85\%$&$2.12\%$\\ \hline
			order & -& 0.8218 &0.7987 &0.7950 &0.8608\\ \hline
		\end{tabular}
	\end{center}
\end{table}

\begin{table}[htp]
	\begin{center}
\caption{Example \ref{nexample2}: Relative errors of the displacements on the square mesh for VEM}\label{tab2-2}
		\begin{tabular}{|c|c|c|c|c|c|c|} \hline
			$h$ &1/8 &1/16 &1/32 &1/64 &1/128\\ \hline
			error &$10.10\%$&$5.9\%$&$3.44\%$&$1.98\%$&$1.11\%$\\ \hline
			order & -& 0.7756 &0.7783 &0.7969 &0.8480\\ \hline
		\end{tabular}
	\end{center}
\end{table}

\begin{figure}[!htb]
	\centering
	\subfigure[]{\includegraphics[scale=0.2]{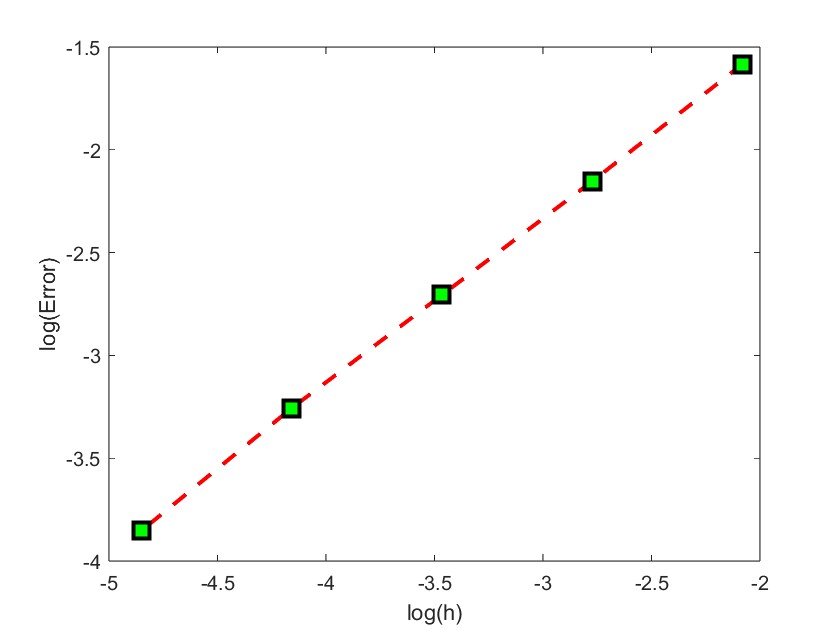}}
	\subfigure[]{\includegraphics[scale=0.2]{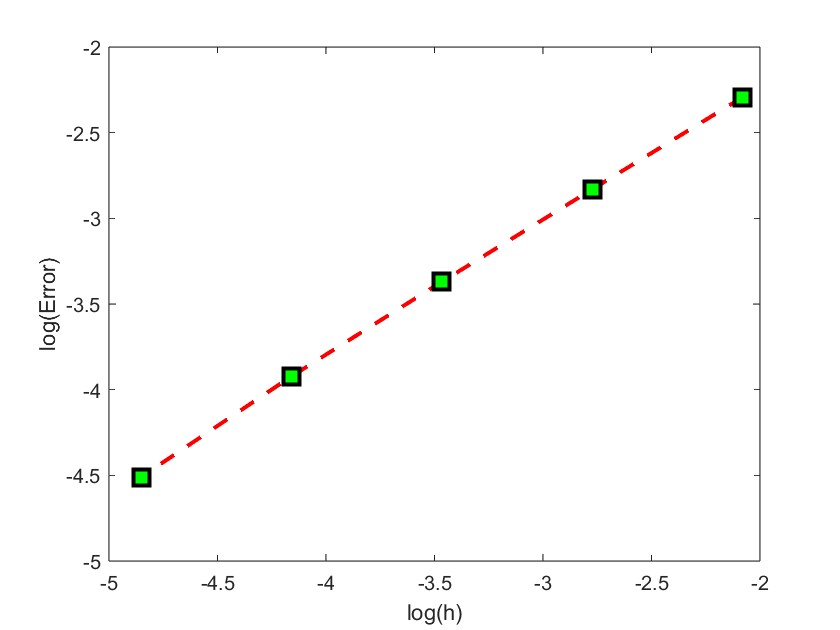}}\\
\caption{Example \ref{nexample2}: Numerical convergence orders for (a) FEM; (b) VEM on the square mesh}\label{fig2d}
\end{figure}

\begin{example} \label{nexample3}
{\rm The contact boundary conditions on $\Gamma_C$ are characterized by a frictionless multivalued normal compliance
contact	in which the penetration is restricted by unilateral constraint. For simulations, we let
\begin{align*}
& u_\nu \leq g, \quad \sigma_\nu+\xi_\nu \leq 0, \quad\left(u_\nu-g\right)\left(\sigma_\nu+\xi_\nu\right)=0\\
		&\xi_\nu =
			\begin{cases}
			0\quad \text{if } u_\nu < 0,\\
			[0,2]\quad \text{if } u_\nu =0,\\
			2\quad \text{if } u_\nu \in (0,0.04],\\
			4-50u_{\nu}\quad \text{if } u_\nu \in  (0.04,0.06],\\
			20u_{\nu}-0.2\quad \text{if } u_\nu > 0.06,
		\end{cases}\\
		&\bsigma_\tau = \bzero.
	\end{align*}
	
This time, we choose $g=0.06$.

}
\end{example}

In the VEM, we present the numerical solution on square mesh for different values of mesh numbers $N$ 
in Figure \ref{fig31}. In Figure \ref{fig32}, we present the initial and deformed meshes on Voronoi meshes 
corresponding to $N=8000$ for the VEM. The numerical solution obtained by linear FEM and lowest order VEM 
on the square grid along the normal direction on the boundary $[0,1] \times \{0\}$ is shown in 
Figure \ref{fig34}.  The relative errors and numerical convergence orders are reported in Table \ref{tab3}, 
Table \ref{tab3-2} and Figure \ref{fig35}.  \hfill$\Box$

\begin{figure}[htb]

	\setlength{\abovecaptionskip}{0.cm}
	\setlength{\belowcaptionskip}{-0.cm}
	\centering
	\begin{minipage}[t]{0.5\linewidth}
		\centering
		\includegraphics[height=3.5cm,width=5cm]{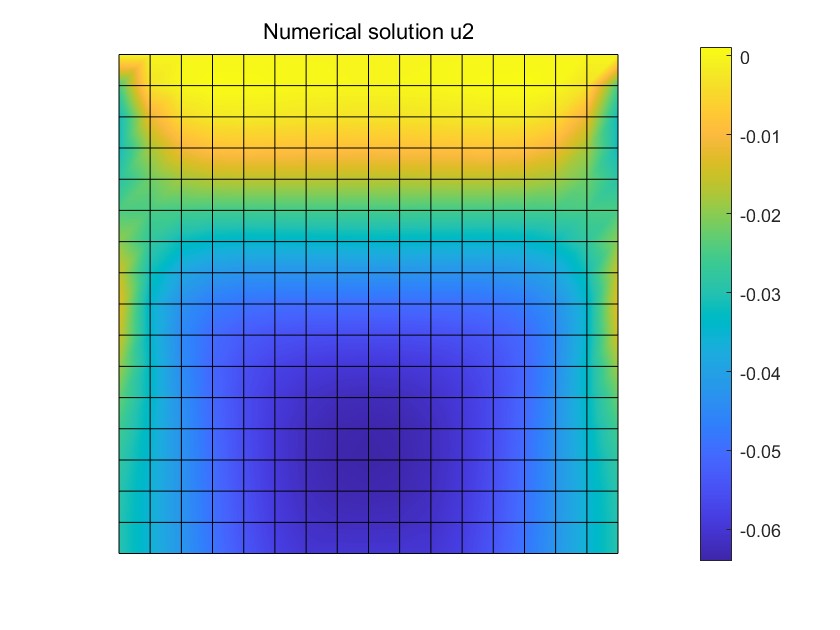}
	\end{minipage}%
	\begin{minipage}[t]{0.6\linewidth}
		\centering
		\includegraphics[height=3.5cm,width=5cm]{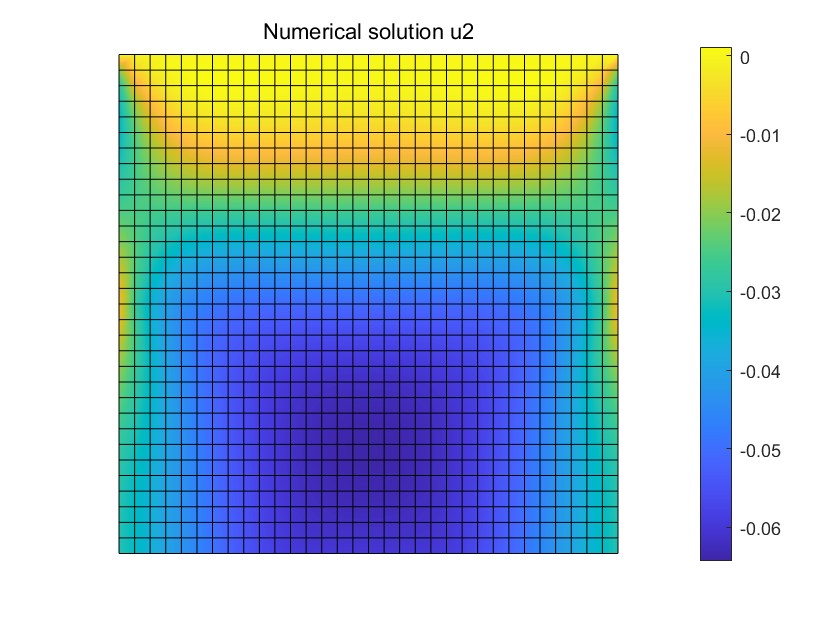}
	\end{minipage}
	\begin{minipage}[t]{0.5\linewidth}
		\centering
		\includegraphics[height=3.5cm,width=5cm]{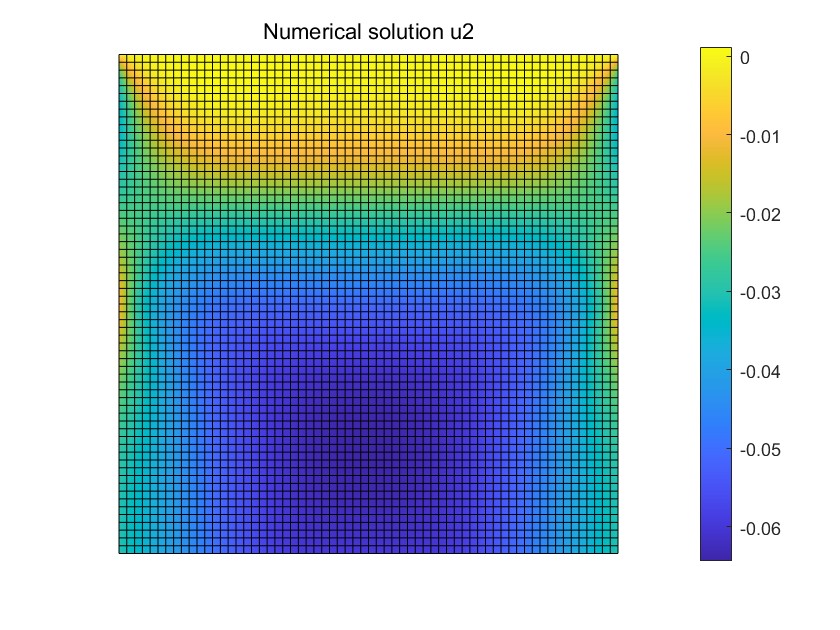}
	\end{minipage}%
	\begin{minipage}[t]{0.6\linewidth}
		\centering
		\includegraphics[height=3.5cm,width=5cm]{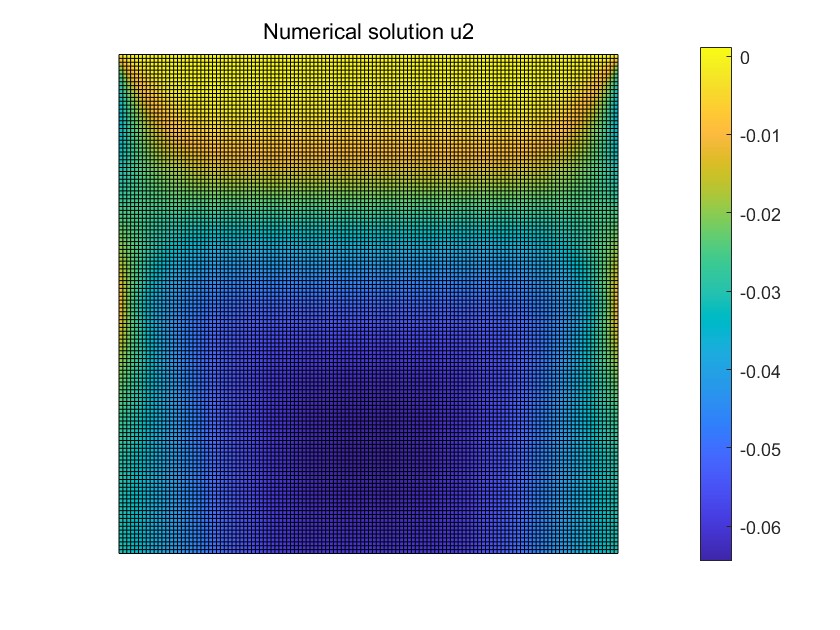}
	\end{minipage}
	\caption{Example \ref{nexample3}: Numerical solutions with $N$ elements: $N=256$ (upper left),
		$N=1024$ (upper right), $N=4096$ (bottom left) and $N=16384$ (bottom right).}
	\label{fig31}
\end{figure}

\begin{figure}[!htb]
	\centering
	\subfigure[]{\includegraphics[scale=0.2]{Ex1_initial_mesh.jpg}}
	\subfigure[]{\includegraphics[scale=0.2]{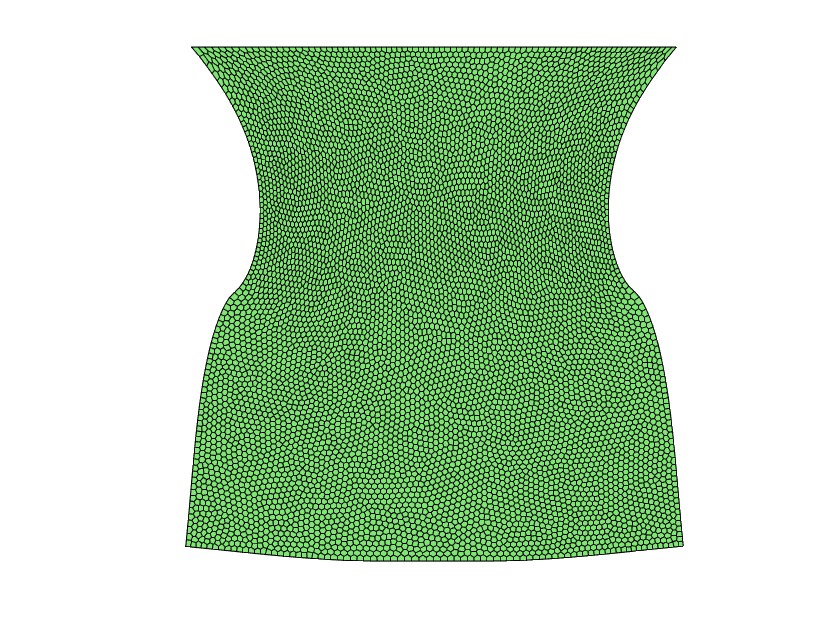}}\\
\caption{Example \ref{nexample3}: (a) Initial mesh with $N=8000$; (b) deformed meshes with $N=8000$}.\label{fig32}
\end{figure}

\begin{figure}[!htb]
	\centering
	\subfigure[]{\includegraphics[scale=0.2]{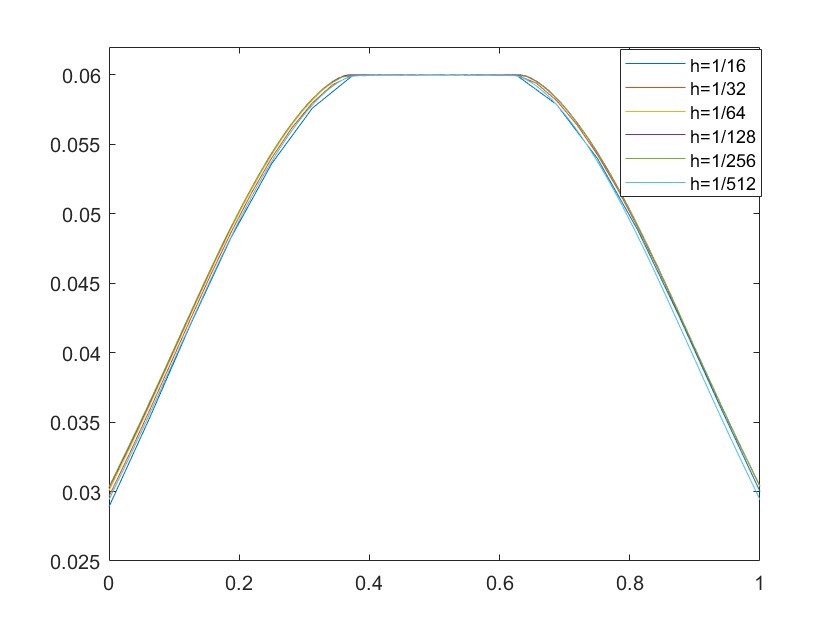}}
	\subfigure[]{\includegraphics[scale=0.2]{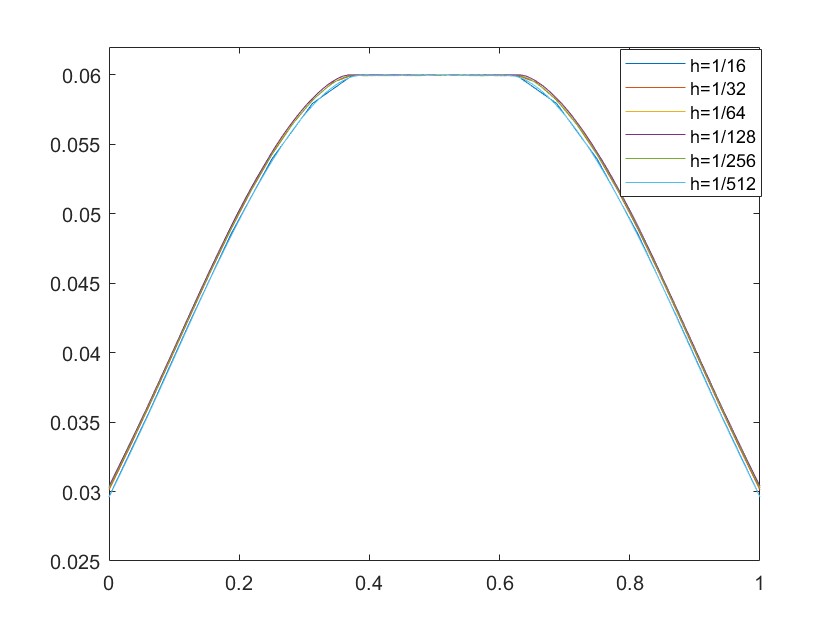}}\\
	\caption{Example \ref{nexample3}: Normal displacement on $\Gamma_3$ for (a) FEM; 
	(b) VEM on square mesh}\label{fig34}
\end{figure}

\begin{table}[htp]
	\begin{center}
		\caption{Example \ref{nexample3}: relative errors of the displacements for FEM}\label{tab3}
		\begin{tabular}{|c|c|c|c|c|c|c|} \hline
			$h$ &1/8 &1/16 &1/32 &1/64 &1/128  \\ \hline
			error &$20.43\%$&$11.57\%$&$6.63\%$&$3.79\%$&$2.04\%$\\ \hline
			order & -& 0.8203  &0.8033 &0.8068 &0.8936\\ \hline
		\end{tabular}
	\end{center}
\end{table}
\begin{table}[htp]
	\begin{center}
\caption{Example \ref{nexample3}: Relative errors of the displacements on the square mesh for VEM}\label{tab3-2}
		\begin{tabular}{|c|c|c|c|c|c|c|} \hline
			$h$ &1/8 &1/16 &1/32 &1/64 &1/128  \\ \hline
			error &$10.09\%$&$5.91\%$&$3.45\%$&$2.01\%$&$1.15\%$\\ \hline
			order & -& 0.7717 &0.7766 &0.7794 &0.8056\\ \hline
		\end{tabular}
	\end{center}
\end{table}

\begin{figure}[!htb]
	\centering
	\subfigure[]{\includegraphics[scale=0.2]{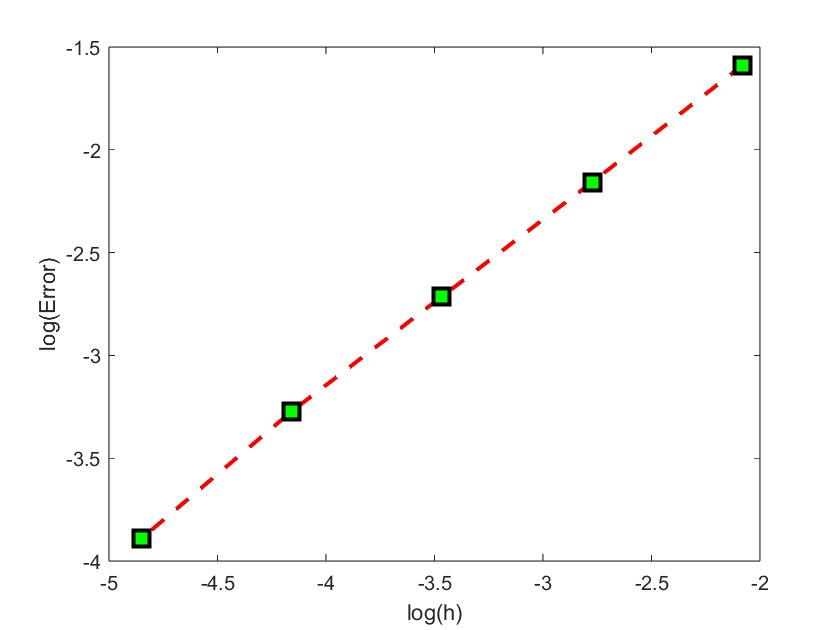}}
	\subfigure[]{\includegraphics[scale=0.2]{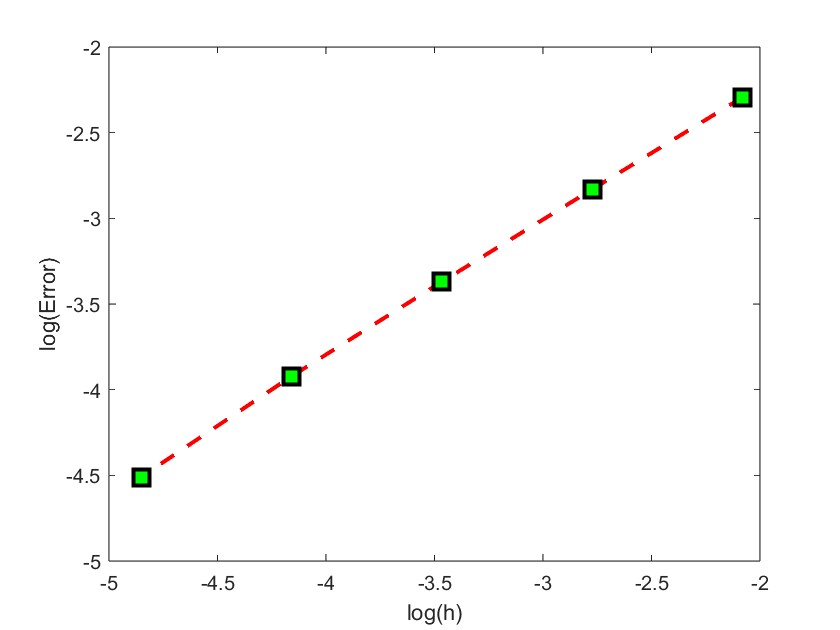}}\\
	\caption{Example \ref{nexample3}: Numerical convergence orders of (a) FEM solutions; 
	(b) VEM solutions on the square mesh}\label{fig35}
\end{figure}

\section{Concluding remarks}\label{sec:final}

This paper is devoted to numerical analysis of variational-hemivariational inequalities, especially those arising
in contact mechanics. Abstract frameworks are presented for the finite element method and the virtual element method
to solve the variational-hemivariational inequalities, and the results are applied to the numerical solution of
three representative contact problems.  In particular, a general convergence result is shown for Galerkin solutions 
of abstract variational-hemivariational inequalities under minimal solution regularity conditions available from 
the well-posedness theory, and optimal order error estimates are derived for the lowest order (linear) 
finite element solutions and virtual element solutions under certain solution regularity assumptions. 
Numerical examples are reported on the performance of both the finite element method and the 
virtual element method.  

Other numerical methods can be employed to solve the contact problems as well.  For instance, similar to the 
virtual element method, a polytopal method, called hybrid high-order method (HHO), has been applied to solve 
contact problems, cf.\ \cite{CCE20, CEP20, BKRWWW22}.  It will be interesting to study HHO to solve 
general variational-hemivariational inequalities.

For practical use of numerical methods, one important issue is the assessment of the reliability of numerical 
solutions, which is accomplished by a posteriori error estimates of numerical solution errors after the 
numerical solutions are found.  The interest in a posteriori error estimation for the finite element method 
began in the late 1970s (\cite{BR78a, BR78b}). Since then, a large number of papers and 
books have been published on this subject.  Historically, two of the influential books on
a posteriori error analysis are \cite{Ver1996, AO2000}.  Note that most of the publications 
on a posteriori error analysis deal with variational equation problems.  In \cite{Han2005}, 
a systematic approach was developed for a posteriori error analysis and adaptive solutions of 
variational inequalities, by employing the duality theory in convex analysis (\cite{ET1976}).  
Another approach was employed in deriving a posteriori 
error estimators for variational inequalities of the second kind in \cite{wang2013another}. 
Similar approaches were extended to perform a posteriori error analysis in the virtual element method 
for simplified friction problems. Specifically, a residual-based error estimator for VEM was proposed 
in \cite{deng2019posteriori}, while a gradient recovery-type a posteriori error estimator was introduced 
in \cite{wei2023gradient}.  In \cite{PS2025}, a posteriori error analysis of the elliptic obstacle problem 
was addressed using hybrid high-order methods.  A posteriori error analysis for $C^0$ interior penalty methods 
was performed for a fourth-order variational inequality of the second kind in \cite{GP2016} and that 
for the obstacle problem of clamped Kirchhoff plates in \cite{BGSZ17}.
It will be an interesting and important topic to establish a posteriori error estimates for numerical 
solutions of variational-hemivariational inequalities, and to apply the a posteriori error estimates 
to develop adaptive algorithms to solve contact problems in the form of variational-hemivariational inequalities.

\end{document}